\documentclass[12pt]{article}

\usepackage{amsfonts, amsthm, amsmath, mathrsfs, bbm, geometry, hyperref,enumitem,csvsimple,graphicx,nicefrac,comment,listings,cprotect,float}
\usepackage{cite}
\usepackage{todonotes}
\usepackage[super]{nth}

\usepackage{pgf,pgfmath}
\usepgflibrary{fpu}

\pgfkeys{/pgf/number format/.cd, int detect,
1000 sep={},
min exponent for 1000 sep=0} 

\usepackage[capitalise,sort]{cleveref}
\crefname{figure}{Figure}{Figures}

\crefformat{equation}{(#2#1#3)}
\crefmultiformat{equation}{ (#2#1#3)}{ and~(#2#1#3)}{, (#2#1#3)}{, and~(#2#1#3)}

\hypersetup{colorlinks=true}

\crefname{subsection}{Subsection}{Subsections}
\crefname{enumi}{item}{items}

\usepackage{makecell}

\DeclareFontEncoding{LS1}{}{}
\DeclareFontSubstitution{LS1}{stix}{m}{n}
\DeclareMathAlphabet{\mathscr}{LS1}{stixscr}{m}{n}

\renewcommand{\d}{\operatorname{d\!}}
\newcommand{\R}{\mathbb{R}}

\newcommand{\N}{\mathbb{N}}
\newcommand{\Z}{\mathbb{Z}}
\newcommand{\E}{\mathbb{E}}
\renewcommand{\P}{\mathbb{P}}
\newcommand{\F}{\mathcal{F}}

\newcommand{\Exp}[1]{ \E \! \left[ #1 \right]}

\newcommand{\EXPP}[1]{ \E \big[ #1 \big]}

\newcommand{\Var}[1]{\operatorname{Var}\left(#1\right)}
\newcommand{\Cov}[2]{\operatorname{Cov}\left(#1,#2\right)}

\newcommand{\qand}{\qquad\text{and}}
\newcommand{\qandq}{\qquad\text{and}\qquad}
\newcommand{\andq}{\text{and}\qquad}

\newcommand{\normalcdf}{\mathfrak{N}}

\newcommand{\Borel}{\mathcal{B}}

\newcommand{\dimParam}{\mathfrak{p}}
\newcommand{\Param}{\mathfrak{P}}
\newcommand{\ParamRV}{P}

\newcommand{\WRV}{\mathcal{W}}
\newcommand{\WRVs}{W}
\newcommand{\dimWRV}{\mathbf{d}}

\newcommand{\NumMC}{\mathfrak{M}}
\newcommand{\MCvariable}{\mathfrak{m}}

\newcommand{\NumRefMC}{\mathscr{M}}

\newcommand{\BatchVariable}{{\bf m}}
\newcommand{\Batchsize}{{\bf M}}

\newcommand{\LRVANN}{\mathscr{N}}

\newcommand{\BSrate}{r}
\newcommand{\BSinitprice}{\xi}
\newcommand{\BSdividend}{\delta}
\newcommand{\boundary}{B}

\newcommand{\numTimePoints}{N}
\newcommand{\eulerapprox}{\Phi}
\newcommand{\dimBM}{d}
\newcommand{\eulerRV}{\mathfrak{W}}

\newcommand{\MLMCdims}{\mathbf{d}}
\newcommand{\MLMCRV}{\mathfrak{W}}
\newcommand{\MLMCproj}{\mathcal{P}}
\newcommand{\MLMClevel}{L}
\newcommand{\MLMCreference}{\mathcal{L}}

\newcommand{\MLPuEval}{\mathcal{P}}
\newcommand{\MLPindexset}{\mathbf{I}}
\newcommand{\MLPindex}{\mathbf{i}}
\newcommand{\MLPAlg}{V}
\newcommand{\MLPRVs}{W}
\newcommand{\MLPsingleRV}{\mathfrak{W}}
\newcommand{\MLPcount}{C}
\newcommand{\MLPlevel}{N}
\newcommand{\MLPreference}{\mathcal{N}}
\newcommand{\MLPApproxDim}{\MLPcount_\MLPlevel \dimWRV}

\newcommand{\dimSolution}{k}

\newcommand{\ApproxAlg}{\Psi}
\newcommand{\ApproxAlgdim}{\mathfrak{d}}
\newcommand{\ApproxAlgRV}{\mathfrak{W}}

\newcommand{\ApproxRef}{\Xi}
\newcommand{\ApproxRefdim}{\mathbf{d}}
\newcommand{\ApproxRefRV}{\mathbf{W}}

\newcommand{\costfct}{\mathfrak{C}}

\newcommand{\dimProblem}{d}
\newcommand{\antithetic}{a}
\newcommand{\exact}{e}

\newcommand{\new}[1]{#1}

\newtheorem{lemma}{Lemma}[section]

\newtheorem{corollary}[lemma]{Corollary}

\newtheorem{algo}[lemma]{Framework}

\begin{document}

\title{Learning the random variables in Monte Carlo \\ simulations with stochastic gradient descent: \\Machine learning for parametric \\ PDEs and financial derivative pricing}

\author{
	Sebastian Becker$^1$,
	Arnulf Jentzen$^{2, 3}$, \\
	Marvin S.\ M\"{u}ller$^{4}$, and
	Philippe von Wurstemberger$^{5,6, *}$
	\bigskip
	\\
	\small{$^1$Risklab, Department of Mathematics, ETH Zurich, }  
	\\
	\small{Switzerland, e-mail: \texttt{sebastian.becker@math.ethz.ch}}
	\smallskip
	\\
	\small{$^2$School of Data Science and Shenzhen Research Institute of Big Data,}
	\\
	\small{The Chinese University of Hong Kong, Shenzhen, China, e-mail: \texttt{ajentzen@cuhk.edu.cn}}
	\smallskip
	\\
	\small{$^3$Applied Mathematics: Institute for Analysis and Numerics,}
	\\
	\small{University of M\"unster, Germany, e-mail: \texttt{ajentzen@uni-muenster.de} }
	\smallskip
	\\
	\small{$^4$2Xideas Switzerland AG, Switzerland, e-mail: \texttt{marvin.s.mueller@gmail.com}}
	\smallskip
	\\
	\small{$^5$Risklab, Department of Mathematics, ETH Zurich,}
	\\
	\small{Switzerland, e-mail: \texttt{philippe.vonwurstemberger@math.ethz.ch}	}
	\smallskip
	\\
	\small{$^6$School of Data Science, The Chinese University of}
	\\
	\small{Hong Kong, Shenzhen, China, e-mail: \texttt{philippevw@cuhk.edu.cn} }
	\smallskip
	\\
	\small{$^*$Corresponding author: \texttt{philippe.vonwurstemberger@math.ethz.ch}}
	\\
}

\maketitle
\pagebreak

\begin{abstract}

In financial engineering, prices of financial products are computed approximately many times each trading day with (slightly) different parameters in each calculation. 
In many financial models such prices can be approximated by means of Monte Carlo (MC) simulations. 
To obtain a good approximation the MC sample size usually needs to be considerably large resulting in a long computing time to obtain a single approximation. 
A natural deep learning approach to reduce the computation time when new prices need to be calculated as quickly as possible would be to train an artificial neural network (ANN) to learn the function which maps parameters of the model and of the financial product to the price of the financial product.
However, empirically it turns out that this approach leads to approximations with unacceptably high errors, in particular when the error is measured in the $L^\infty$-norm, and it seems that ANNs are not capable to closely approximate prices of financial products in dependence on the model and product parameters in real life applications.
This is not entirely surprising given the high-dimensional nature of the problem and the fact that it has recently been proved for a large class of algorithms, including the deep learning approach outlined above, 
that such methods are in general not capable to overcome the curse of dimensionality for such approximation problems in the $L^\infty$-norm.
In this paper we introduce a new numerical approximation strategy for parametric approximation problems including the parametric financial pricing problems described above and we illustrate by means of several numerical experiments that the introduced approximation strategy achieves a very high accuracy for a variety of high-dimensional parametric approximation problems, even in the $L^\infty$-norm.
A central aspect of the approximation strategy proposed in this article is to combine MC algorithms with machine learning techniques to, roughly speaking, \textit{learn the random variables} (LRV) in MC simulations.
In other words, we employ stochastic gradient descent (SGD) optimization methods not to train parameters of standard ANNs but instead to learn random variables appearing in MC approximations.
In that sense, the proposed LRV strategy has strong links to \textit{Quasi-Monte Carlo} (QMC) methods as well as to the field of algorithm learning.
Our numerical simulations strongly indicate that the LRV strategy might indeed be capable to overcome the curse of dimensionality in the $L^\infty$-norm in several cases where the standard deep learning approach has been proven not to be able to do so.
This is not a contradiction to the established lower bounds mentioned above because this new LRV strategy is outside of the class of algorithms for which lower bounds have been established in the scientific literature.
The proposed LRV strategy is of general nature and not only restricted to the parametric financial pricing problems described above, but applicable to a large class of approximation problems.
In this article we numerically test the LRV strategy 
	in the case of the pricing of European call options in the Black-Scholes model with one underlying asset,
	in the case of the pricing of European worst-of basket put options in the Black-Scholes model with three underlying assets,
	in the case of the pricing of European average put options in the Black-Scholes model with three underlying assets and knock-in barriers,
	as well as
	in the case of stochastic Lorentz equations.
For these examples the LRV strategy produces highly convincing numerical results when compared with standard MC simulations, QMC simulations using Sobol sequences, SGD-trained shallow ANNs, and SGD-trained deep ANNs.

\end{abstract}
\pagebreak

\tableofcontents
\allowdisplaybreaks

\section{Introduction}

Many computational problems from engineering and science can be cast as certain parametric approximation problems (cf., e.g., \cite{Cohen2015,Khoo2021,Kutyniok2021,Bhattacharya2020,hs99,berner2020numerically,Vidales18,Chkifa15,Villa22,MR4153855} and references mentioned therein). 
In particular, parametric PDEs are of fundamental importance in various applications, where one is not only interested in an approximation of the solution of the approximation problem at one fixed (space-time) point but where one is interested to evaluate the approximative solution again and again as, for instance, in financial engineering where prices of financial products are computed approximately many times each trading day with (slightly) different parameters in each calculation. 
Moreover, the problems appearing in such financial applications are often high-dimensional, as the dimension usually corresponds to the number of assets/financial contracts in the considered trading portfolio.

It is a widespread issue that the majority of algorithms for such parametric approximation problems suffer from the \textit{curse of dimensionality} (cf., e.g., Bellman \cite{Bellman1957} and Novak \& Wo\'{z}niakowski \cite[Chapter 1]{Novak08}) in the sense that the computational effort of the approximation methods grows exponentially in the dimension of the approximation problem or in the required approximation accuracy, making them useless for high-dimensional problems.
In the \textit{information based complexity} literature there are fundamental lower bounds which generally reveal the impossibility to approximate the solutions of certain classes of high-dimensional approximation problems without the curse of dimensionality among general classes of approximation algorithms; see, e.g.,
Grohs \& Voigtlaender \cite{Grohs2021},
Heinrich \cite{Heinrich2006},
Heinrich \& Sindambiwe \cite{hs99}, and
Novak \& Wo\'{z}niakowski \cite{NovakWozniakowski08}.
Developing methods which produce good approximations for high-dimensional problems is thus an exceedingly hard task and, among the general classes of algorithms considered in the above named references, essentially impossible.

In this paper we present a new method to tackle high-dimensional parametric approximation problems.
Roughly speaking, our strategy is based on the idea to \textit{combine Monte Carlo} (MC) \textit{algorithms} (such as, e.g., standard MC methods, multilevel Monte Carlo (MLMC) methods, or multilevel Picard (MLP) methods) \textit{with stochastic gradient descent} (SGD) \textit{optimization methods} by viewing the employed realizations of random variables in the MC approximation as training parameters for the SGD optimization method. 
In other words, in this approach we intend to employ SGD optimization methods not to train standard artificial neural networks (ANNs) but to learn random variables appearing in MC approximations.

To make this idea more concrete, we now sketch this \textit{learning the random variables} (LRV) approximation strategy in the context of a basic example of a parametric approximation problem.
Let $\dimParam \in \N$, suppose that we intend to approximate a target function $u \colon [0,1]^\dimParam \to \R$, 
let $\phi \colon [0,1]^\dimParam \times \R \to \R $ be measurable and bounded, 
let $ ( \Omega, \F, \P ) $ be a probability space, 
let $ \WRV \colon \Omega \to \R $ be a random variable, 
and suppose that $u \colon [0,1]^\dimParam \to \R$ admits the probabilistic representation that for all 
	$p \in [0, 1]^\dimParam$
we have that
\begin{equation}
\label{intro:eq0}
\begin{split} 
  u( p ) = 
  \mathbb{E}\!\left[ 
    \phi( 
      p, \WRV
    )
  \right]
\end{split}
\end{equation}
(parametric integration problem; cf., e.g.,
Cohen \& DeVore \cite{Cohen2015} and
Heinrich \& Sindambiwe \cite{hs99}). 
We  note that we only chose the random variable $\WRV$ to be 1-dimensional for simplicity and refer to \cref{sect:MC} for the case when $\WRV$ is a possibly high-dimensional random variable.
Our first step to derive the proposed approximation algorithm is to recall standard MC approximations for the parametric expectation in \eqref{intro:eq0}. 
Let $ \NumMC \in \N $, 
let 
$ 
  \WRVs^{ m, \MCvariable } \colon \Omega \to \R 
$, $ m, \MCvariable \in \N_0 $, 
be i.i.d.\ random variables which satisfy for all $ A \in \mathcal{B}( \R ) $ that 
$
  \P( W^{ 0, 0 } \in A ) = \P( \WRV \in A ) 
$, and observe 
for all $ p \in [0,1]^\dimParam $ that
\begin{equation}
\label{intro:eq1}
  u( p ) = 
  \mathbb{E}\!\left[ 
    \phi( 
      p, \WRV
    )
  \right]
  =
  \mathbb{E}\!\left[ 
    \phi( 
      p, \WRVs^{ 0, 1 }
    )
  \right]
  \approx 
  \tfrac{ 1 }{ \NumMC }
  \big[
    {\textstyle \sum\nolimits_{ \MCvariable = 1 }^{ \NumMC }}
    \phi( 
      p, \WRVs^{ 0, \MCvariable }
    )
  \big]
  .
\end{equation}
In the next step we introduce a parametric function on $[0,1]^\dimParam$ with parameter set $\R^{ \NumMC}$ to reformulate the standard MC approximation in \eqref{intro:eq1} in such a way that the random variables 
$
	\WRVs^{ 0, m } \colon \Omega \to \R
$, $m \in \{1, 2, \ldots, \NumMC\}$, 
correspond to the parameters of the parametric function.
More precisely, let $\LRVANN \colon [0,1]^\dimParam \times \R^{ \NumMC} \to \R$ satisfy for all
	$p \in [0,1]^\dimParam$, 
	$\theta = ( \theta_1, \dots, \theta_{ \NumMC } ) \in \R^{\NumMC}$
that
$
	\LRVANN(p, \theta)
=
    \frac{ 1 }{ \NumMC }
      \big[ 
        \sum_{ \MCvariable = 1 }^{ \NumMC }
        \phi( 
          p, 
          \theta_{ \MCvariable }
        )
      \big]
$
and note that \eqref{intro:eq1} suggests for all 
	$ p \in [0,1]^\dimParam $ 
that
\begin{equation}
\label{intro:eq3}
  u( p )
  \approx 
  \tfrac{ 1 }{ \NumMC }
  \big[
    {\textstyle \sum_{ \MCvariable = 1 }^{ \NumMC }
    \phi( 
      p, \WRVs^{ 0, \MCvariable }
    )}
  \big]
 =
 \LRVANN\!\big(p,  (\WRVs^{ 0, 1 }, \WRVs^{ 0, 2 }, \ldots, \WRVs^{ 0,  \NumMC})\big)
  .
\end{equation}
Next we employ the SGD optimization method to train the right hand side of \eqref{intro:eq3} in search of “better random realizations” to approximate the target function $u$ 
than those provided by the random variables 
$
	\WRVs^{ 0, m } \colon \Omega \to \R
$, $m \in \{1, 2, \ldots, \NumMC\}$, 
in the standard MC approximation in \eqref{intro:eq3}.
The random variables 
$
	\WRVs^{ 0, m } \colon \Omega \to \R
$, $m \in \{1, 2, \ldots, \NumMC\}$, 
on the right hand side of \eqref{intro:eq3} then only supply the initial guess in the SGD training procedure. 
More specifically, 
let 
$  
  \ParamRV_{ 
    m, \BatchVariable 
  }
  \colon \Omega \to [0,1]^\dimParam
$, $ m, \BatchVariable \in \N $, 
be i.i.d.\ random variables with continuous uniform distribution, 
let $ \Batchsize \in \N $, 
for every 
	$m \in \N$
let 
$ 
  F_m 
  \colon \R^{ \NumMC } 
  \times \Omega \to \R 
$
satisfy 
for all 
$ 
  \theta = 
  ( \theta_1, \dots, \theta_{ \NumMC } ) 
  \in \R^{ \NumMC } 
$ 
that
\begin{equation}
\label{intro:eq4}
\begin{split}
  F_m( \theta ) 
  &=
  \tfrac{ 1 }{ \Batchsize }
  \big[
    {\textstyle \sum_{ \BatchVariable = 1 }^{ \Batchsize }
    \left|
      \phi( 
        \ParamRV_{ 
          m, \BatchVariable 
        }, 
        \WRVs^{ m, \BatchVariable } 
      )
      -
      \LRVANN(
        \ParamRV_{ 
          m, \BatchVariable 
        }, 
        \theta
      )
    \right|^2
    }
  \big],
\end{split}
\end{equation}
assume for all 
$p \in [0,1]^\dimParam$ that
$
	(\R \ni w \mapsto \phi(p, w) \in \R) \in C^1(\R, \R)
$,
for every 
	$m \in \N$
let 
$ 
  G_m 
  \colon \R^{ \NumMC } \times \Omega 
  \to \R^{ \NumMC } 
$
satisfy 
for all 
	$ \theta \in \R^{ \NumMC } $, $ \omega \in \Omega $
that
$
  G_m( \theta, \omega )
  =
  ( \nabla_{ \theta } F_m )( \theta, \omega )
$,
let $ (\gamma_m)_{m \in \N} \subseteq (0,\infty) $, and 
let 
$ 
  \Theta \colon \N_0 \times \Omega \to \R^{ \NumMC } 
$ 
satisfy 
for all $ m \in \N $ that 
\begin{equation}
\label{intro:eq5}
  \Theta_0 = 
  ( \WRVs^{ 0, 1 }, \WRVs^{ 0, 2 }, \dots, \WRVs^{ 0, \NumMC } )
 \qandq
  \Theta_{ m }
  =
  \Theta_{ m - 1 }
  - 
  \gamma_m 
  G_m( \Theta_{ m - 1 } ) 
  .
\end{equation}
Note that the recursion in \eqref{intro:eq5} describes nothing else but the standard SGD optimization method with the learning rate schedule
$
\N \ni m \mapsto
\gamma_m \in \R
$.
For every sufficiently large 
$m \in \N$
we then propose to employ the random function 
\begin{equation}
\label{intro:eq6}
\begin{split} 
	[0,1]^\dimParam \times \Omega \ni (p, \omega) \mapsto \LRVANN(p, \Theta_m(\omega)) \in \R
\end{split}
\end{equation}
as an approximation for the target function
$
	[0,1]^\dimParam \ni p \mapsto u(p) \in \R
$. 
Note that the set 
$
	[0,1]^\dimParam \ni p \mapsto \LRVANN(p, \theta) \in \R
$, $\theta \in \R^{\NumMC}$,
of potential approximating functions for $u \colon [0,1]^\dimParam \to \R$
does not consist of standard fully-connected feedforward ANNs but
is a very problem specific class of approximating functions determined by the MC method in \eqref{intro:eq1}.
In light of this and of the fact that the expression in \eqref{intro:eq3} resembles the definition of single layer fully-connected feedforward ANNs we refer to this class of approximating functions as \textit{MC neural networks}.
Moreover, observe that through \eqref{intro:eq3} the MC method also naturally specifies a favorable initializing law for the SGD training procedure in \eqref{intro:eq5}.

In the special case of the LRV strategy illustrated in \eqref{intro:eq1}--\eqref{intro:eq6} above, the standard MC approximation method serves as a \emph{proposal algorithm} for the SGD training procedure. 
More generally, the LRV strategy can, in principle, be used on any high-dimensional approximation problem as soon as there is a reasonable stochastic proposal algorithm for the considered approximation problem available.
The LRV strategy thereby naturally specifies the compositional
architecture of the involved approximating functions and also naturally specifies the initializing law in the SGD
training procedure for each specific approximation problem.
In particular, the LRV strategy can be used 
	with the standard MC method (see \cref{sect:MC}),
	the MC-Euler-Maruyama method (see \cref{sect:MCEUler}), or 
	the MLMC method (see \cref{sect:MLMC}) 
as the proposal algorithms to approximate solutions of stochastic differential equations (SDEs) and Kolmogorov PDEs, respectively, and 
the LRV strategy can be used with 
	the MLP method (see \cref{sect:MLP}) 
as the proposal algorithm to approximate solutions of semilinear PDEs.

\begin{table}[htb] \tiny
\resizebox{\textwidth}{!}{
\begin{filecontents*}{Table_1.tex}
method;num-samples; num-params; l1-error; l-2-error; l-inf-error; train-time; eval-time
ANNs (2 hidden layers, 64 neurons each);-;4609;0.003855;0.006479;0.266781;181.95;0.78
ANNs (2 hidden layers, 128 neurons each);-;17409;0.002519;0.004411;0.197840;191.05;0.74
ANNs (3 hidden layers, 64 neurons each);-;8769;0.001761;0.002791;0.124099;196.73;0.81
ANNs (3 hidden layers, 128 neurons each);-;33921;0.001244;0.001813;0.097105;199.35;0.84
MC;32768; -;0.051127;0.081914;1.079885;0;17.12
MC (antithetic);32768; -;0.026839;0.045043;0.593046;0;35.07
QMC;32768; -;0.001920;0.002585;0.011879;0;15.48
QMC (antithetic);32768; -;0.000286;0.000466;0.002769;0;33.65
LRV;8192;8192;0.000065;0.000134;0.002138;1773.54;4.17
LRV (antithetic);8192;8192;0.000064;0.000132;0.002210;2799.30;7.73
\end{filecontents*}
\csvreader[tabular=|c|c|c|c|c|c|c|,
	separator=semicolon,
     table head=
     \hline Approximation method &\thead{Number of \\ trainable \\ parameters} & \thead{Number\\ of \\ MC/QMC \\ samples} & \thead{$L^2$-approx. \\ error}& \thead{$L^\infty$-approx. \\ error}& \thead{Training \\ time in \\ seconds} & \thead{Evaluation \\ time in \\ seconds}  \\\hline,
    late after line=\\\hline]{Table_1.tex}
{
	method = \method, 
	num-params = \nparams,
	num-samples = \nsamp,
	l1-error=\lll, 
	l-2-error=\llll, 
	l-inf-error=\linf, 
	train-time=\train, 
	eval-time=\eval
}
{ \method  & \nparams & \nsamp & \llll & \linf &\train & \eval }}
\caption{\label{table:BS1_Summary}Pricing European call options in the Black-Scholes model}
\end{table}

We now illustrate the effectiveness of the LRV strategy on a simple but famous numerical example, where the LRV strategy leads to impressively precise approximations.
In \cref{table:BS1_Summary} we present results for the approximative computation of prices of European call options in the Black-Scholes model 
	by means of the deep learning method induced by Becker et al.\ \cite{BeckJafaari21} with $140000$ Adam (see Kingma \& Ba \cite{Kingma2014}) training steps
		(rows 2--5 in \cref{table:BS1_Summary}),
	by means of the standard MC method with 32768 MC samples 
		(row 6 in \cref{table:BS1_Summary}), 
	by means of the antithetic MC method with 32768 MC samples 
			(row 7 in \cref{table:BS1_Summary}), 
	by means of the Quasi-Monte Carlo (QMC) method using Sobol sequences with 32768 QMC samples 
		(row 8 in \cref{table:BS1_Summary}), and
	by means of the antithetic QMC method using Sobol sequences with 32768 QMC samples 
		(row 9 in \cref{table:BS1_Summary}), 
	by means of the LRV strategy with the standard MC method as the proposal algorithm with $140000$ Adam training steps
		(row 10 in \cref{table:BS1_Summary}), and
	by means of the LRV strategy with the antithetic MC method as the proposal algorithm with $140000$ Adam training steps
		(row 11 in \cref{table:BS1_Summary}).
In \cref{table:BS1_Summary} the $L^2$-error and the $L^\infty$-error have been computed approximately on the region
\begin{equation}
\begin{split} 
\label{intro:eq7}
	(\BSinitprice,T,r,\sigma,K) 
\in 
	[90,110]
	\times[ \tfrac{1}{100} ,1]
	\times[- \tfrac{1}{10} , \tfrac{1}{10} ]
	\times[ \tfrac{1}{100} , \tfrac{1}{2}]
	\times[90,100]
\end{split}
\end{equation}
using 8\,192\,000 evaluations of the considered approximation method. 
In \eqref{intro:eq7} we have 
	that $\BSinitprice$ stands for the initial price, 
	that $T$ stands for the time of maturity, 
	that $r$ stands for the drift rate, 
	that $\sigma$ stands for the volatility, and
	that $K$ stands for the strike price; see, e.g., \cite[Lemma 4.4]{Becker2019published} (with c = 0 in the notation of \cite[Lemma 4.4]{Becker2019published}). 
The reference solution values to compute the errors in \cref{table:BS1_Summary} have been computed with the famous Black-Scholes formula (see, e.g., \cite[Lemma 4.4]{Becker2019published} or \eqref{BS1:eq1} in \cref{simul:BS1}). 
The evaluation time corresponds to the time required to compute 8\,192\,000 evaluations.
We note that the training and evaluation times\footnote{
	The numerical experiments have been 
	performed in {\sc TensorFlow} 2.12
	running on a system equipped with an
	\textsc{NVIDIA GeForce RTX 4090} GPU with 24 GB Graphics RAM.
}
of the LRV strategy are significantly longer when compared to the deep learning method induced by Becker et al.\ \cite{BeckJafaari21} even when the considered MC neural networks involve less arithmetic operations than the considered ANNs. 
This is likely attributed to the efficient implementation and parallelization of standard feedforward ANNs in Tensorflow and our own implementation of MC neural networks, which may be slightly less computationally efficient.
The numbers in \cref{table:BS1_Summary} are taken from \cref{table:BS_LRV,table:BS_MC,table:BS_ANNs,table:BS_QMC} in \cref{simul:BS1}. 
We refer to \cref{simul:BS1} for more details on the results in \cref{table:BS1_Summary}.

Note that \cref{table:BS1_Summary} indicates that the algorithm obtained by the LRV strategy not only produces very accurate prices in the $L^2$-norm, but even has a very high accuracy in the uniform $L^\infty$-norm.
Concretely, for this 5-dimensional approximation problem, this strongly suggests that for any choice of parameters in the region considered in \eqref{intro:eq7}, the LRV strategy offers an approximation with an error smaller than $\frac{2}{1000}$.
Based on 
	this, 
	on the other numerical results presented in this paper, as well as
	on preliminary analytic investigations (cf., e.g., Gonon et al.\ \cite[Lemma 2.16]{Gonon19Uniform})
	we conjecture that the LRV strategy can overcome the curse of dimensionality for certain classes of parametric PDE problems in the $L^\infty$-norm. 
We would like to emphasize that large classes of algorithms for such approximation problems have been shown not to be able to overcome the curse of dimensionality, see, e.g., Heinrich \cite[Theorem 2.4]{hs99}, Heinrich \& Sindambiwe \cite[Theorem 1]{Heinrich2006}, and Grohs \& Voigtlaender \cite{Grohs2021}.
However our conjecture does not contradict the general lower bounds  established in the
above mentioned references, due to the fact that the LRV strategy does not belong to the class of algorithms considered in the
above mentioned references: roughly speaking, in the LRV strategy there are two stages of computational procedures, the \textit{main computational procedure} in which the “best random variables” are learned through SGD (corresponding to the 6th column in \cref{table:BS1_Summary}) and the \textit{evaluation procedure} where the computed approximation of the target function is evaluated (corresponding to the 7th column in \cref{table:BS1_Summary}). In the LRV strategy we consider the situation where it is allowed to perform function evaluations both during the {main computational} procedure \textit{and} the {evaluation procedure} while the lower bounds in \cite{hs99,Heinrich2006,Grohs2021} consider the situation where function evaluations are only allowed during the main computational procedure \textit{but not} during the evaluation procedure. 
The LRV strategy being outside of the classes of algorithms considered in the above named references is thus not constrained by the established lower bounds and hence holds the potential to overcome the curse of dimensionality, even in the $L^\infty$-norm.
We note that in practically relevant situations, just as in the considered derivative pricing problem, it is often possible to perform function evaluations also in the evaluation procedure, thus making the LRV strategy an applicable method for practically relevant approximation problems.

We now compare the proposed LRV strategy to existing algorithms and computational methods in the scientific literature.
As the LRV strategy employs SGD-type methods to "learn" parametric functions it is related to deep learning methods, where instead of random variables, optimal weights of ANNs are "learned".
From this point of view, the LRV strategy can be seen as a machine learning approach where instead of employing generic ANNs, very problem specific parametric functions are used, which contain a lot of human insight about the problem at hand.
There is a plethora of deep learning methods for the approximation of PDEs, which have been developed recently, and seem to be very effective for the approximation of high-dimensional PDEs:
cf., e.g., 
\cite{Weinan2017,han2018solving,sirignano2018dgm,E17Ritz,Fujii19Asymptotic,dissanayake1994neural,lagaris1998artificial,jianyu2003numerical,chen2020friedrichs,ChanMikaelWarin2019,Germain2020,Han20,han2020solving,henry2017deep,Hure20,Jacquier2019a,Shaolin20,Kremsner2020,WangEtal2019,pham2019neural,raissi2018forward,DeepSplitting,BeckJafaari21,berner2020numerically,HanEtal2020,Nusken21,samaniego2020energy,Becker2018,Becker2019published,Chen2019,Goudenege2019,Lyu2020,Lu2019,Zang2020weak,Han2019deep,Lye20,Zhu19,LongoEtal2020,Nuesken2021,Martin2020}. We refer to the survey articles \cite{beck2020overviewPublished,blechschmidt2021three,Germain2021,Weinan2020a} for a more detailed overview.
For methods which are specifically designed for parametric PDEs we refer to, e.g., \cite{Vidales18,Khoo2021,Bhattacharya2020,berner2020numerically,Villa22}.
In addition, we want to highlight a connection between the LRV strategy and the deep learning method for Kolmogorov PDEs developed in Becker et al.\ \cite{BeckJafaari21} and further specialized to parametric Kolmogorov PDEs in Berner et al.\ \cite{berner2020numerically}. Very roughly speaking, 
the work \cite{berner2020numerically} is concerned with approximating a function 
$u \colon \Gamma \to \R$ given for all
	$\gamma \in \Gamma$
by
$u(\gamma) = \Exp{\varphi_\gamma(S_\gamma)}$
where $\Gamma$ is a parameter set, where $d \in \N$, where $(\Omega, \mathcal{F}, \P)$ is a probability space, where $\varphi_\gamma \colon \R^d \to \R$, $\gamma \in \Gamma$, are parametric functions, and where for every $\gamma \in \Gamma$ the random variable $S_\gamma \colon \Omega \to \R^d$ is the solution of an SDE parametrized by the parameter $\gamma$. 
In \cite{berner2020numerically} they propose to approximate $u$ by looking for a function $f \colon \Gamma \to \R$ within a certain class of ANNs which aims to minimize 
$
	\Exp{(f(\Lambda) - \varphi_\Lambda(Y_\Lambda))^2}
$
where $\Lambda \colon \Omega \to \Gamma$ is a random variable and for every $\gamma \in \Gamma$ we have that $Y_\gamma \colon \Omega \to \R^d$ is an approximation of $S_\gamma$ (such as, e.g., an Euler-Maruyama approximation).
If the class of ANNs in which an optimal function $f$ is looked for is replaced by parametric functions induced by a proposal algorithm this becomes a special case of the LRV framework presented in this paper.

Furthermore, the LRV strategy can be associated with a broad subcategory of machine learning called \textit{algorithm learning} (cf., e.g., \cite{Towell1994,Chen2020,Chen2020RNA,gregor2010learning,Borgerding2016,Chen2018,Yoon2013,mensch18a,Wilder2019}).
Roughly speaking, algorithm learning refers to the idea to employ existing algorithms with known empirical or theoretical qualities as a basis to construct or extend ANNs or more general parametric function families.
In many cases the employed algorithm relies on certain hyper-parameters 
(such as, e.g., 
the learning rates in case of SGD-type methods: cf., e.g., Chen et al.\ \cite{Chen2020})
which typically are added to the set of trainable parameters of the ANN.
It is in this point that the LRV strategy differs from existing algorithm learning methods since the LRV strategy considers the random variables and not the hyper-parameters of the proposal algorithm as learnable parameters.
It thereby has the advantage that the initialization of all trainable parameters is implicitly given through the proposal algorithm.
However, the LRV strategy could very well be combined with ideas of algorithm learning. We leave this task open for future research.

Another tranche of literature connected to the LRV strategy are QMC methods (cf., e.g., \cite{drovandi2018improving,Morokoff1995,Caflisch1998}).
Roughly speaking, the idea of QMC methods is to replace uniformely distributed random variables of the MC method by a more suitable sequence of deterministic points to obtain a higher rate of convergence.
The idea to improve the choice of random variables in MC methods is a common feature of QMC methods with the LRV strategy. A key difference between the two methods is that QMC methods construct new integration points which have good properties for a wide class of integrands while the LRV strategy aims to "learn" new integration points which are specific to each considered integrand.

Next we illustrate a link between the LRV strategy and quantization methods (cf., e.g., \cite{Pages98,pages2015introduction,pages2004optimal,pages2003optimal,pages2005functional,dereich2013constructive,frikha2012quantization,pham2005approximation,rudd2017fast,Gronbach13Quantization,Altmayer14}). 
Quantization methods are concerned with approximating a continuously distributed random variable by a random variable with a finite image.
Quantization was first introduced for signal processing and information theory (cf., e.g., \cite{Graf20Foundations,pages2003optimal}) but can also be used for the numerical approximation of expectations involving the original random variable (cf., e.g., \cite{Pages98,pages2003optimal}). 
For the latter, the original random variable in the expectation is replaced by its quantization, resulting in an expected value which can easily be computed by a finite sum of weighted function evaluations.
A good quantization
is typically found with optimization methods such as, e.g., the Newton method in low dimensions or SGD-type optimization methods in high-dimensions (cf., e.g., \cite{pages2003optimal,pages2015introduction,pages2004optimal,pham2005approximation}).
To make this more concrete we now roughly illustrate this in the context of the setting described above.
For all 
	$\theta = (\theta_1, \ldots, \theta_\NumMC) \in \R^\NumMC$, 
	$C_1, C_2, \ldots, C_\NumMC \in \Borel(\R)$ 
with 
	$\forall \, m,n \in \{1, 2, \ldots, \NumMC\}, m \neq n \colon C_m \cap C_n = \{\}$ 
	and 
	$\cup_{m = 1}^\NumMC C_m = \R$ 
we consider a quantization
	$Q_{\theta,C_1, C_2, \ldots, C_\NumMC} \colon \Omega \to \R^\NumMC$ 
of $\WRV$ given by
$
	Q_{\theta,C_1, C_2, \ldots, C_\NumMC}
=
	\sum_{\mathfrak{m} = 1}^\NumMC \theta_{\mathfrak{m}} \mathbbm{1}_{C_\mathfrak{m}}(\WRV)
$.
A good quantization is found by minimizing an error between the quantization and the original random variable such as, e.g., the squared $L^2$-error given for all 
	$\theta = (\theta_1, \ldots, \theta_\NumMC) \in \R^\NumMC$, 
	$C_1, C_2, \ldots, C_\NumMC \in \Borel(\R)$ 
by
\begin{equation}
\label{intro:eq8}
\begin{split} 
\Exp{\Vert\WRV - Q_{\theta,C_1, C_2, \ldots, C_\NumMC}\Vert^2},
\end{split}
\end{equation}
with an SGD-type optimization method.
Once appropriate 
	$\theta = (\theta_1, \ldots, \theta_\NumMC) \in \R^\NumMC$, 
	$C_1, C_2, \ldots, C_\NumMC \in \Borel(\R)$ 
	with
		$\forall \, m,n \in \{1, 2, \ldots, \NumMC\}, m \neq n \colon C_m \cap C_n = \{\}$ 
		and 
		$\cup_{m = 1}^\NumMC C_m = \R$ 
have been found,
they can be employed to approximate the expectation in \eqref{intro:eq0} for every $p \in \Param$ by
\begin{equation}
\label{intro:eq9}
\begin{split} 
\textstyle
	u(p) 
=
	\Exp{\phi(p, \WRV)}
\approx
 	\Exp{\phi(p, Q_{\theta,C_1, C_2, \ldots, C_\NumMC})}
=
	\sum_{\mathfrak{m} = 1}^\NumMC \phi(p, \theta_{\mathfrak{m}}) \P( \WRV \in C_\mathfrak{m}).
\end{split}
\end{equation}
Note that this approximation has a similar form as the proposed LRV approximation in \eqref{intro:eq6} and
in both cases the points at which $\phi$ is evaluated are found through an SGD-type optimization method.
The main difference is that the optimization problem in \eqref{intro:eq8}, which is used to determine the evaluation points in \eqref{intro:eq9}, only depends on the random variable $\WRV$ whereas the optimization problem (see \eqref{intro:eq4}) to determine the evaluation points in \eqref{intro:eq6} in the LRV strategy depends on the random variable $\WRV$ and the function $\phi$.

The concept of optimizing random MC samples, a central aspect of the LRV strategy, also appears in the Bayesian statistics literature in the context of inducing points in Bayesian learning (cf., e.g., \cite{snelson2005sparse,ranganath2014black,dellaportas2019gradient,Betancourt2017}).
Such inducing points are employed to find sparse representation for large data sets.
One way to find good inducing points is to first sample them randomly and then optimize them for example with gradient based methods (cf., e.g., Snelson \& Ghahramani \cite{snelson2005sparse}), which bears some resemblance to the LRV strategy.

The reminder of this article is structured as follows.
In \cref{sect:MC,sect:MCEUler,sect:MLMC,sect:MLP} we present the LRV strategy for increasingly complex approximation problems and proposal algorithms. 
As proposal algorithms we consider 
	the standard MC method in \cref{sect:MC},
	the MC-Euler method in \cref{sect:MCEUler},
	the MLMC method in \cref{sect:MLMC}, and
	the MLP method in \cref{sect:MLP}.
In \cref{sect:general} we present the most general case of a generic proposal algorithm, which includes all previous sections as special cases.
The results of numerical experiments for the LRV approximation strategy are presented in \cref{sect:numerics}.
Specifically, 
	we consider 1-dimensional Black-Scholes equations for European call options (resulting in a 5-dimensional parametric approximation problem) in \cref{simul:BS1},
	we consider 3-dimensional Black-Scholes equations for worst-of basket put options (resulting in a 15-dimensional parametric approximation problem)  in \cref{simul:BSworstPut},
	we consider 3-dimensional Black-Scholes equations for average basket put options with knock-in barriers (resulting in a 16-dimensional parametric approximation problem) in \cref{simul:BSaverageBarrier},
	and
	we consider stochastic Lorentz equations (resulting in a 10-dimensional parametric approximation problem) in \cref{simul:Lorentz}.

\section{Learning the random variables (LRV) strategy in the case of Monte Carlo approximations}
\label{sect:MC}

In this section we employ the LRV strategy for the approximation of parametric expectations (see \eqref{MC:eq1} in  \cref{MC:target}).
This includes as a special case the approximative pricing of European options in the Black-Scholes model, as illustrated in \cref{MC:example}.
The LRV strategy 
based on standard MC averages as proposal algorithms 
for the general parametric approximation problem of \cref{MC:target}
is elaborated in \cref{MC:proposal,MC:replacing,MC:Loss,MC:SGD}.
The resulting method is summarized in a single framework in \cref{MC:description}.

\subsection{Parametric expectations involving vector valued random variables}
\label{MC:target}

Let $ \dimParam,\dimWRV  \in \N $, 
let $ \Param \subseteq \R^{ \dimParam } $ be measurable, 
let $ u \colon \Param \to \R $ be a function, 
let $ ( \Omega, \F, \P ) $ be a probability space, 
let $ \WRV \colon \Omega \to \R^{ \dimWRV } $ be a random variable, 
let $ \phi \colon \Param \times \R^{ \dimWRV } \to \R $ be measurable, 
assume for all $ p \in \Param $ that
$
  \R^{ \dimWRV } \ni w \mapsto \phi( p, w ) \in \R 
$
is continuously differentiable, 
and assume for all $ p \in \Param $ that 
$
  \E\big[
    | \phi( p, \WRV ) |
  \big]
  < \infty
$
and
\begin{equation}
\label{MC:eq1}
  u( p ) = 
  \mathbb{E}\!\left[ 
    \phi( 
      p, \WRV
    )
  \right]
  .
\end{equation}
The goal of this section is to derive an algorithm to approximately compute the function
$
	u \colon \Param \to \R
$
given through the parametric expectation in \eqref{MC:eq1}.

\subsection{Approximative pricing of European options in the Black-Scholes model}
\label{MC:example}

\begin{lemma}\label{cor:black_scholes}
Let
$ \BSinitprice, T, \sigma \in ( 0, \infty ) $,
$ \BSrate \in \R $,
let $ \normalcdf \colon \R \to \R $ satisfy for all $ x \in \R $ that
$
\normalcdf( x ) = \int_{ - \infty }^x \frac{ 1 }{ \sqrt{ 2 \pi } } \, \exp({ - \frac{ y^2 }{ 2 }  }) \d y
$,
let $ ( \Omega, \mathcal{F}, ( \mathbb{F}_t )_{ t \in [0,T] }, \P ) $ be a filtered probability space
which satisfies the usual conditions,
let $ W \colon [0,T] \times \Omega \to \R $
be a standard $( \mathbb{F}_t )_{ t \in [0,T] }$-Brownian motion
with continuous sample paths,
and
let $ X \colon [0,T] \times \Omega \to \R $
be an $( \mathbb{F}_t )_{ t \in [0,T] }$-adapted stochastic process with continuous sample paths
which satisfies that
for all $ t \in [ 0, T ] $
it holds $ \P $-a.s.\ that
\begin{equation}
X_t
= \BSinitprice +
\BSrate \int_{0}^{t} X_s \, \d s +
\sigma \int_{0}^{t} X_s \, \d W_s
.
\end{equation}
Then 
\begin{enumerate}[label=(\roman *)]
\item \label{cor:black_scholes:item1}
it holds for all $ t \in [ 0, T ] $ that
$
\P\big(X_t
=
\exp\bigl(
\bigl[ \BSrate - \tfrac{ \sigma^2 }{2} \bigr] t
+
\sigma \, W_t \bigr)
\BSinitprice
\big)
=
1
$
and
\item \label{cor:black_scholes:item2}
it holds
for all $ K \in \R $ that
\begin{equation}
\begin{split}
& \E \bigl[
\exp({ -\BSrate T })
\max \{ X_T - K, 0 \}
\bigr]
\\ & =
\begin{cases}
\BSinitprice
\,
\normalcdf \Bigl(
\tfrac{ ( \BSrate + \frac{ \sigma^2 }{2} ) T + \ln( {\BSinitprice}/{K} ) }{ \sigma \sqrt{T} }
\Bigr)
-
K \exp({ - r \, T})
\,
\normalcdf \Bigl(
\tfrac{ ( \BSrate - \frac{ \sigma^2 }{2} ) T + \ln( {\BSinitprice}/{K} ) }{ \sigma \sqrt{T} }
\Bigr)
&
\colon
K > 0
\\[1ex]
 \BSinitprice
- K \exp({ - \BSrate \, T})
&
\colon
K \leq 0.
\end{cases}
\end{split}
\end{equation}

\end{enumerate}

\end{lemma}
\begin{proof}[Proof of \cref{cor:black_scholes}]
Observe that, e.g.,  \cite[item (i) in Proposition~4.3]{Becker2019published} implies that
for all $ t \in [ 0, T ] $
it holds $ \P $-a.s.\ that
\begin{equation}
\label{cor:black_scholes:eq1}
X_t
=
\exp\bigl(
\bigl[ \BSrate - \tfrac{ \sigma^2 }{2} \bigr] t
+
\sigma \, W_t \bigr)
\BSinitprice
=
\exp\big({
 ( \BSrate - \tfrac{ \sigma^2 }{2} ) t
+
\ln( \BSinitprice )
+
\sigma W_t}\big)
.
\end{equation}
This establishes item~\ref{cor:black_scholes:item1}.
Moreover, note that, e.g., \cite[Lemma~4.4]{Becker2019published} and \eqref{cor:black_scholes:eq1}  show
for all $ K \in \R $ that
\begin{equation}
\begin{split}
& \E \bigl[
\exp({ -\BSrate T })
\max \{ X_T - K, 0 \}
\bigr]
=
\exp({ -\BSrate T }) \,
\E \Bigl[
\max \Bigl\{ \exp\big({
( \BSrate - \tfrac{ \sigma^2 }{2} ) T
+
\ln( \BSinitprice )
+
\sigma W_T }\big) - K, 0 \Bigr\}
\Bigr]
\\ & =
\begin{cases}
 \BSinitprice
\,
\normalcdf \Bigl(
\tfrac{ ( \BSrate + \frac{ \sigma^2 }{2} ) T + \ln( {\BSinitprice}/{K} ) }{ \sigma \sqrt{T} }
\Bigr)
-
K \exp({ - \BSrate \, T})
\,
\normalcdf \Bigl(
\tfrac{ ( \BSrate - \frac{ \sigma^2 }{2} ) T + \ln( {\BSinitprice}/{K} ) }{ \sigma \sqrt{T} }
\Bigr)
&
\colon
K > 0
\\[1ex]
\BSinitprice
- K \exp({ - \BSrate \, T})
&
\colon
K \leq 0
\end{cases}
.
\end{split}
\end{equation}
The proof of \cref{cor:black_scholes} is thus complete.
\end{proof}

In the case 
where 
  $ \dimParam = 5 $,
  $ \dimWRV = 1 $, 
   and
  $
    \Param = [90, 110] \times [0.01,1] \times [-0.1,0.1] \times [0.01,0.5] \times  [90, 110]
  $, 
where 
for all 
$
  p = (\BSinitprice, T, \BSrate, \sigma, K ) \in \Param
$,
$
  w \in \R
$
it holds
that
$
  \phi( p, w ) = 
  \exp({-\BSrate T})
  \max\!\big\{ \!
    \exp\!\big(
      [ \BSrate - \frac{ \sigma^2 }{ 2 } ] T + T^{ 1 / 2 } \sigma w  
    \big) \BSinitprice
    - K
    , 0
  \big\}
  $,
  where $ \WRV $ is a standard normal random variable, 
  and
  where $ \normalcdf \colon \R \to \R $ satisfies for all $ x \in \R $ that
$
\normalcdf( x ) = \int_{ - \infty }^x \frac{ 1 }{ \sqrt{ 2 \pi } } \, \exp({ - \frac{ y^2 }{ 2 } }) \, \d y
$
observe that 
\begin{enumerate}[label=(\roman *)]
\item 
it holds for all $ p = (\BSinitprice, T, \BSrate, \sigma, K )  \in \Param $
that
\begin{equation}
  \begin{split}
    u(p) &= \E \bigl[
     \exp({-\BSrate T})
      \max\!\big\{ \!
        \exp\!\big(
          [ \BSrate - \tfrac{ \sigma^2 }{ 2 } ] T + T^{ 1 / 2 } \sigma \WRV  
        \big) \BSinitprice
        - K
        , 0
      \big\}
\bigr]
\\
  &=
  \BSinitprice
\,
\normalcdf \Bigl(
\tfrac{ ( \BSrate  + \frac{ \sigma^2 }{2} ) T + \ln( {\BSinitprice}/{K} ) }{ \sigma \sqrt{T} }
\Bigr)
-
K \exp({ - \BSrate \, T})
\,
\normalcdf \Bigl(
\tfrac{ ( \BSrate - \frac{ \sigma^2 }{2} ) T + \ln( {\BSinitprice}/{K} ) }{ \sigma \sqrt{T} }
\Bigr)
  \end{split}
\end{equation}
(cf.\ \eqref{MC:eq1} and \cref{cor:black_scholes})
and
\item 
it holds for all 
	$p = (\BSinitprice, T, \BSrate, \sigma, K ) \in \Param$ 
that
$u(p) \in [0,\infty)$ is
the price of an European call option 
in the Black-Scholes model 
with 
initial underlying price $ \BSinitprice $, 
time of maturity $ T $, 
drift rate $ \BSrate $, 
volatility $ \sigma $, 
and strike price $ K $.
\end{enumerate}

\subsection{Monte Carlo approximations}
\label{MC:proposal}

Let $ \NumMC \in \N $
and let 
$ 
  \WRVs^{ m, \MCvariable } \colon \Omega \to \R^{ \dimWRV } 
$, 
$ m, \MCvariable \in \N_0 $, 
be i.i.d.\ random variables 
which satisfy for all $ A \in \mathcal{B}( \R^{ \dimWRV } ) $ that 
$
  \P( W^{ 0, 0 } \in A ) = \P( \WRV \in A ) 
$.
Observe that \eqref{MC:eq1} suggests that 
for all $ p \in \Param $ it holds that
\begin{equation}
\label{mc:eq1}
  u( p ) = 
  \mathbb{E}\!\left[ 
    \phi( 
      p, \WRV
    )
  \right]
  =
  \mathbb{E}\!\left[ 
    \phi( 
      p, \WRVs^{ 0, 1 }
    )
  \right]
  \approx 
  \frac{ 1 }{ \NumMC }
  \left[
    \sum_{ \MCvariable = 1 }^{ \NumMC }
    \phi( 
      p, \WRVs^{ 0, \MCvariable }
    )
  \right]
  .
\end{equation}

\subsection{Replacing the random variables in Monte Carlo approximations}
\label{MC:replacing}

Let $\LRVANN \colon \Param \times \R^{ \NumMC \dimWRV} \to \R$ satisfy for all
	$p \in \Param$, 
	$\theta = ( \theta_1, \dots, \theta_{ \NumMC \dimWRV } ) \in \R^{\NumMC \dimWRV}$
that
\begin{equation}
\label{replacing:eq1}
\begin{split} 
	\LRVANN(p, \theta)
=
    \frac{ 1 }{ \NumMC }
      \left[ 
        \sum_{ \MCvariable = 1 }^{ \NumMC }
        \phi\big( 
          p, 
          ( 
            \theta_{ ( \MCvariable - 1 ) \dimWRV + k }
          )_{k \in \{1, 2, \ldots, \dimWRV \}}
        \big)
      \right].
\end{split}
\end{equation}
Note that \eqref{mc:eq1} and \eqref{replacing:eq1} suggest that 
for all $ p \in \Param $ it holds that
\begin{equation}
  u( p )
  \approx 
  \frac{ 1 }{ \NumMC }
  \left[
    \sum_{ \MCvariable = 1 }^{ \NumMC }
    \phi( 
      p, \WRVs^{ 0, \MCvariable }
    )
  \right]
 =
 \LRVANN\!\big(p,  (\WRVs^{ 0, 1 }, \WRVs^{ 0, 2 }, \ldots, \WRVs^{ 0,  \NumMC})\big)
  .
\end{equation}

\subsection{Random loss functions for fixed random variables in Monte Carlo approximations} 
\label{MC:Loss}

Let $ \Batchsize \in \N $, 
let 
$  
  \ParamRV_{ 
    m, \BatchVariable 
  }
  \colon \Omega \to \Param
$, $ m, \BatchVariable \in \N $, 
be i.i.d.\ random variables, 
for every 
	$m \in \N$
let 
$ 
  F_m 
  \colon \R^{ \NumMC \dimWRV } 
  \times \Omega \to \R 
$
satisfy 
for all 
$ 
  \theta = 
  ( \theta_1, \dots, \theta_{ \NumMC \dimWRV } ) 
  \in \R^{ \NumMC \dimWRV } 
$ 
that
\begin{equation}
\begin{split}
  F_m( \theta ) 
  &=
  \frac{ 1 }{ \Batchsize }
  \left[
    \sum_{ \BatchVariable = 1 }^{ \Batchsize }
    \left|
      \phi( 
        \ParamRV_{ 
          m, \BatchVariable 
        }, 
        \WRVs^{ m, \BatchVariable } 
      )
      -
      \LRVANN(
        \ParamRV_{ 
          m, \BatchVariable 
        }, 
        \theta
      )
    \right|^2
  \right] 
\\&= 
  \frac{ 1 }{ \Batchsize }
  \left[
    \sum_{ \BatchVariable = 1 }^{ \Batchsize }
    \left|
      \phi( 
        \ParamRV_{ 
          m, \BatchVariable 
        }, 
        \WRVs^{ m, \BatchVariable } 
      )
      -
      \frac{ 1 }{ \NumMC }
      \left[ 
        \sum_{ \MCvariable = 1 }^{ \NumMC }
        \phi\big( 
          \ParamRV_{ 
            m, \BatchVariable 
          }, 
          ( 
            \theta_{ ( \MCvariable - 1 ) \dimWRV + k }
          )_{k \in \{1, 2, \ldots, \dimWRV \}}
        \big)
      \right]
    \right|^2
  \right] ,
\end{split}
\end{equation}
and for every 
	$m \in \N$
let 
$ 
  G_m 
  \colon \R^{ \NumMC \dimWRV } \times \Omega 
  \to \R^{ \NumMC \dimWRV } 
$
satisfy 
for all 
	$ \theta \in \R^{ \NumMC \dimWRV } $, 
	$ \omega \in \Omega $
that
\begin{equation}
  G_m( \theta, \omega )
  =
  ( \nabla_{ \theta } F_m )( \theta, \omega ).
\end{equation}

\subsection{Learning the random variables with stochastic gradient descent}
\label{MC:SGD}

Let $ (\gamma_m)_{m \in \N} \subseteq (0,\infty) $ and 
let 
$ 
  \Theta \colon \N_0 \times \Omega \to \R^{ \NumMC \dimWRV } 
$ 
satisfy 
for all $ m \in \N $ that 
$ 
  \Theta_0 = 
  ( \WRVs^{ 0, 1 }, \WRVs^{ 0, 2 }, \dots, \WRVs^{ 0, \NumMC } )
$
and
\begin{equation}
  \Theta_{ m }
  =
  \Theta_{ m - 1 }
  - 
  \gamma_m 
  G_m( \Theta_{ m - 1 } ) 
  .
\end{equation}
\new{
For every sufficiently large 
$m \in \N$
we propose to employ the random function 
$
	\Param \times \Omega \ni (p, \omega) \mapsto \LRVANN(p, \Theta_m(\omega)) \in \R
$
as an approximation for the target function
$
	\Param \ni p \mapsto u(p) \in \R
$ in \eqref{MC:eq1}. 
}

\subsection{Description of the proposed approximation algorithm}
\label{MC:description}

\begin{algo}
Let $ \dimParam,\dimWRV, \NumMC, \Batchsize\in \N $, 
$ (\gamma_m)_{m \in \N} \subseteq (0,\infty) $,
let $ \Param \subseteq \R^{ \dimParam } $ be measurable, 
let $ \phi \in C( \Param \times \R^{ \dimWRV } , \R ) $, 
let $ ( \Omega, \mathcal{F}, \P ) $ be a probability space, 
let 
$  
  \ParamRV_{ 
    m, \BatchVariable 
  }
  \colon \Omega \to \Param
$, $ m, \BatchVariable \in \N $, 
be i.i.d.\ random variables, 
let 
$ 
  \WRVs^{ m, \MCvariable } \colon \Omega \to \R^{ \dimWRV } 
$, 
$ m, \MCvariable \in \N_0 $, 
be i.i.d.\ random variables, 
assume that 
$
  (
    \ParamRV_{ 
      m, \BatchVariable 
    }
  )_{
    (m, \BatchVariable) \in \N^2
  }
$
and 
$
  (
    \WRVs^{ m, \MCvariable }
  )_{
    (m, \MCvariable) \in \N^2
  }
$
are independent, 
assume for all $ p \in \Param $ that
$
  \R^{ \dimWRV } \ni w \mapsto \phi( p, w ) \in \R 
$
is continuously differentiable, 
for every 
	$m \in \N$
let 
$ 
  F_m 
  \colon \R^{ \NumMC \dimWRV } 
  \times \Omega \to \R 
$
satisfy 
for all 
$ 
  \theta = 
  ( \theta_1, \dots, \theta_{ \NumMC \dimWRV } ) 
  \in \R^{ \NumMC \dimWRV } 
$ 
that
\begin{equation}
  F_m( \theta ) 
  =
  \frac{ 1 }{ \Batchsize }
  \left[
    \sum_{ \BatchVariable = 1 }^{ \Batchsize }
    \left|
      \phi( 
        \ParamRV_{ 
          m, \BatchVariable 
        }, 
        \WRVs^{ m, \BatchVariable } 
      )
      -
      \frac{ 1 }{ \NumMC }
      \left[ 
        \sum_{ \MCvariable = 1 }^{ \NumMC }
        \phi\big( 
          \ParamRV_{ 
            m, \BatchVariable 
          }, 
          ( 
            \theta_{ ( \MCvariable - 1 ) \dimWRV + k }
          )_{k \in \{1, 2, \ldots, \dimWRV \}}
        \big)
      \right]
    \right|^2
  \right] ,
\end{equation}
for every 
	$m \in \N$
let 
$ 
  G_m 
  \colon \R^{ \NumMC \dimWRV } \times \Omega 
  \to \R^{ \NumMC \dimWRV } 
$
satisfy 
for all 
	$ \theta \in \R^{ \NumMC \dimWRV } $, $ \omega \in \Omega $ 
that
$
  G_m( \theta, \omega )
  =
  ( \nabla_{ \theta } F_m )( \theta, \omega )
$, 
and
let 
$ 
  \Theta \colon \N_0 \times \Omega \to \R^{ \NumMC \dimWRV } 
$ 
satisfy 
for all $ m \in \N $ that 
$ 
  \Theta_0 = 
  ( \WRVs^{ 0, 1 }, \WRVs^{ 0, 2 }, \dots, \WRVs^{ 0, \NumMC } )
$
and
\begin{equation}
  \Theta_{ m }
  =
  \Theta_{ m - 1 }
  - 
  \gamma_m 
  G_m( \Theta_{ m - 1 } ) 
  .
\end{equation}
\end{algo}

\section{LRV strategy in the case of Monte Carlo-Euler-Maruyama approximations}
\label{sect:MCEUler}

In the previous section we derived the LRV strategy for the approximation of parametric expectations involving finite-dimensional random variables.
In this section we consider parametric expectations involving standard Brownian motions as random variables (see \eqref{MCEUler:eq0} in \cref{MCEUler:target}).
To derive an LRV algorithm for these approximation problems we first discretize the Brownian motions in \cref{MCEUler:discretization} to obtain parametric expectations involving only finite-dimensional random variables. 
In \cref{MCEUler:proposal,MCEUler:replacing,MCEUler:Loss,MCEUler:SGD} the LRV strategy is then subsequently applied to these discretized parametric expectations as in \cref{sect:MC} above (cf.\ \cref{MC:proposal,MC:replacing,MC:Loss,MC:SGD}).
In \cref{MCEUler:example} we illustrate a special instance
of the parametric expectations in \eqref{MCEUler:eq0} and of associated discretizations in the situation where the parametric expectations in \eqref{MCEUler:eq0} involve solutions of SDEs.
Finally, in \cref{MCEUler:description} we summarize the algorithm derived in this section in one single framework.

\subsection{Parametric expectations involving vector valued stochastic processes}
\label{MCEUler:target}

Let $ \dimParam,\dimBM \in \N $, $T \in (0, \infty)$,
let $ \Param \subseteq \R^{ \dimParam } $ be measurable, 
let $ u \colon \Param \to \R $ be a function, 
let $ ( \Omega, \F, \P ) $ be a probability space, 
let $ \WRV \colon [0,T] \times \Omega \to \R^{ \dimBM } $ be a 
standard Brownian motion with continuous sample paths, 
let $ \phi \colon \Param \times C( [0,T], \R^{ \dimBM } ) \to \R $ be measurable, 
and
assume for all $ p \in \Param $ that 
$
  \E\big[
    | \phi( p, \WRV ) |
  \big]
  < \infty
$
and
\begin{equation}
\label{MCEUler:eq0}
  u( p ) = 
  \mathbb{E}\!\left[ 
    \phi( 
      p, \WRV
    )
  \right].
\end{equation}
The goal of this section is to derive an  algorithm to approximately compute the function
$
	u \colon \Param \to \R
$
given through the parametric expectation in \eqref{MCEUler:eq0}.

\subsection{Temporal discretizations of the Brownian motion}
\label{MCEUler:discretization}

Let $\numTimePoints \in \N$,
let $ \eulerRV \colon \Omega \to \R^{N\dimBM}$ satisfy 
$
	\eulerRV
=
	\sqrt{{\numTimePoints}/{T}}\,
	(
		\WRV_{T / \numTimePoints} - \WRV_0, 
		\WRV_{ 2 T / \numTimePoints } - \WRV_{ T / \numTimePoints }, 
		\dots, 
		\WRV_T -\WRV_{ (\numTimePoints-1) T / \numTimePoints } 
	)
$,
let
$
\eulerapprox \colon
\Param \times \R^{\numTimePoints \dimBM} \to
\R
$
satisfy for all $ p \in \Param $ that
\(
  \R^{\numTimePoints \dimBM} \ni w \mapsto \eulerapprox(p, w) \in \R
\)
is continuously differentiable,
and assume for all $p \in \Param$ that
$\E[|\eulerapprox(p, \eulerRV)|]<\infty$.
We think of
$\eulerapprox \colon
\Param \times \R^{\numTimePoints \dimBM} \to
\R$
as a suitable approximation of
$\phi \colon  \Param \times C( [0,T], \R^{ \dimBM } ) \to \R$
in the sense that for all $ p \in \Param $ it holds that
\begin{equation}
  \label{eq:u_SP_approx} 
    \phi( 
      p, \WRV
    )
  \approx 
    \eulerapprox\big( 
    p, \sqrt{{\numTimePoints}/{T}}\,
    (\WRV_{T / \numTimePoints} - \WRV_0, \WRV_{ 2 T / \numTimePoints } - \WRV_{ T / \numTimePoints }, \dots, \WRV_T -\WRV_{ (\numTimePoints-1) T / \numTimePoints } )
    \big)
  =
 \eulerapprox(p, \eulerRV).
\end{equation}

\subsection{Euler-Maruyama approximations for parametric stochastic differential equations}
\label{MCEUler:example}

In the case 
where $ \dimParam = 1 + \dimBM $,
where $\boundary \in (0,\infty)$,
where
$
\Param = [0,T] \times [-\boundary,\boundary]^\dimBM
$,
where
$\mu \colon \R^\dimBM \to \R^\dimBM$ is globally Lipschitz continuous,
where
$
X = (X^{\xi, w}_t)_{(t, \xi, w) \in [0,T] \times \R^{\dimBM} \times C([0,T],\R^\dimBM)} \colon [0,T] \times \R^{\dimBM} \times C([0,T],\R^\dimBM) \to \R^{\dimBM}
$
satisfies for all
	$t \in [0,T]$,
	$\xi \in \R^{\dimBM}$,
	$w=(w_s)_{s\in [0,T]} \in C([0,T], \R^\dimBM)$
that 
\begin{equation}
\label{MCeuler:eq1}
  X_t^{\xi, w} = \xi + \int_0^t \mu(X_s^{\xi, w}) \d s + w_t,
\end{equation}
where $g\colon [0,T] \times \R^{\dimBM} \to \R$
satisfies for all $t \in [0,T]$ that 
$\R^{\dimBM} \ni x \mapsto g(t, x) \in \R$
is continuously differentiable,
where it holds for all
$p=(t,\xi)\in \Param$,
$
w \in C([0,T], \R^\dimBM)
$
that
$\phi(p,w) = g(t, X^{\xi,w}_t)$,
where
$
\mathcal{X}^{\xi, \theta} = (\mathcal{X}^{\xi, \theta}_t)_{t \in [0,T]}\colon [0,T] \to \R^\dimBM
$,
	$\theta \in \R^{N \dimBM}$,
	$\xi \in \R^{\dimBM}$,
satisfy for all 
	$\theta \in \R^{N \dimBM}$,
	$\xi \in \R^{\dimBM}$,
	$n\in \{1, 2, \dots \numTimePoints\}$,
	$t\in [\tfrac{(n-1)T}{\numTimePoints},\tfrac{nT}{\numTimePoints}]$
that $\mathcal{X}_0^{\xi,\theta} = \xi$
and
\begin{equation}
\label{MCeuler:eq2}
  \mathcal{X}_{t}^{\xi, \theta} =  \mathcal{X}_{(n-1)T/\numTimePoints}^{\xi, \theta} + \big(\tfrac{t\numTimePoints}{T} - n + 1\big)\bigl(\tfrac{T}{\numTimePoints} \mu(\mathcal{X}_{(n-1)T/N}^{\xi, \theta}) + \tfrac{\sqrt{T}}{\sqrt{\numTimePoints}} (\theta_{(n-1)\dimBM + k})_{k \in \{1, 2, \ldots, \dimBM\}}\bigr),
\end{equation}
and where it holds for all
$
p \in \Param
$,
$
\theta \in \R^{\numTimePoints\dimBM}
$
that
$\eulerapprox( p, \theta ) = g(t, \mathcal{X}^{\xi, \theta}_t)$
observe that 
\begin{enumerate}[label=(\roman *)]
\item 
it holds 
for all 
	$ p = (t,\xi)\in \Param $
that
\begin{equation}
\begin{split} 
	u(p) = u(t, \xi) = \E[g(t,X^{\xi,\WRV}_t) ]
\end{split}
\end{equation}
is the expectation of the test function $\R^d \ni x\mapsto g(t,x) \in \R$ evaluated at time $t$ of the solution process $(X^{\xi,\WRV}_s)_{s\in [0,T]}$ of the additive noise driven SDE in \eqref{MCeuler:eq1}
and
\item 
it holds for all
	$p = (t,\xi)\in \Param $
that
$\eulerapprox(p, \eulerRV)$ is an approximation
\begin{equation}
\begin{split} 
	\eulerapprox(p, \eulerRV) = g(t, \mathcal{X}^{\xi, \eulerRV}_t)
\approx 
	g(t, X^{\xi,\WRV}_t) = \phi(p,\WRV)
\end{split}
\end{equation}
of $\phi(p,\WRV)$ based on linearly interpolated Euler-Maruyama approximations $(\mathcal{X}_{s}^{\xi, \eulerRV})_{s \in [0,T]}$ with $\numTimePoints$ timesteps of the solution $(X_{s}^{\xi, \WRV})_{s \in [0,T]}$ of the SDE in \eqref{MCeuler:eq1}.

\end{enumerate}

\subsection{Monte Carlo approximations}
\label{MCEUler:proposal}

Let $ \NumMC \in \N $
and 
let 
$ 
  \WRVs^{ m, \MCvariable } \colon \Omega \to \R^{ \numTimePoints \dimBM } 
$, 
$ m, \MCvariable \in \N_0 $, 
be i.i.d.\ standard normal random vectors.
Observe that \eqref{MCEUler:eq0} and \eqref{eq:u_SP_approx} suggest that 
for all $ p \in \Param $ it holds that
\begin{equation}
\label{mc:eq2}
    u( p ) = 
  \mathbb{E}\!\left[ 
    \phi( 
      p, \WRV
    )
  \right]
  \approx
  \E[
  	\eulerapprox(p, \eulerRV)
  ]
  =
  \E[
    \eulerapprox(p, \WRVs^{0, 1})
  ]
  \approx
  \frac{ 1 }{ \NumMC }
  \left[
    \sum_{ \MCvariable = 1 }^{ \NumMC }
    \eulerapprox\big( 
    p, 
    \WRVs^{0,\MCvariable}
    \big)
  \right]  
  .
\end{equation}

\subsection{Replacing the random variables in Monte Carlo approximations}
\label{MCEUler:replacing}

Let $\LRVANN \colon \Param \times\R^{ \NumMC \numTimePoints \dimBM }  \to \R$ satisfy for all
	$p \in \Param$, 
	$\theta = (\theta_1, \ldots, \theta_{\NumMC \numTimePoints\dimBM}) \in \R^{ \NumMC \numTimePoints \dimBM } $
that
\begin{equation}
\label{replacing:eq2}
\begin{split} 
	\LRVANN(p, \theta)
=
   \frac{ 1 }{ \NumMC }
      \left[ 
        \sum_{ \MCvariable = 1 }^{ \NumMC }
        \eulerapprox\Big( 
          p, 
          (\theta_{( \MCvariable - 1 ) \numTimePoints \dimBM + k})_{k \in\{1, 2, \dots, \numTimePoints \dimBM\}}          
        \Big)
      \right].
\end{split}
\end{equation}
Note that \eqref{mc:eq2} and \eqref{replacing:eq2} suggest that 
for all $ p \in \Param $ it holds that
\begin{equation}
  u( p )
  \approx
    \frac{ 1 }{ \NumMC }
    \left[
      \sum_{ \MCvariable = 1 }^{ \NumMC }
      \eulerapprox\big( 
      p, 
      \WRVs^{0, \MCvariable}
      \big)
    \right]  
 =
  \LRVANN\!\big(
    p,  
      (
        \WRVs^{ 0,  1  },
        \WRVs^{ 0,  2  },
        \ldots,
        \WRVs^{ 0,  \NumMC  }
       )
  \big)
  .
\end{equation}

\subsection{Random loss functions for fixed random variables in Monte Carlo approximations} 
\label{MCEUler:Loss}
Let $ \Batchsize \in \N $,  
let 
$  
  \ParamRV_{ 
    m, \BatchVariable 
  }
  \colon \Omega \to \Param
$, $ m, \BatchVariable \in \N $, 
be i.i.d.\ random variables, 
for every 
	$m \in \N$
let
$
F_m
  \colon \allowbreak\R^{ \NumMC \numTimePoints \dimBM } 
  \times\Omega\to\R
$
satisfy 
for all 
$
\theta = (\theta_1, \ldots, \theta_{\NumMC \numTimePoints\dimBM}) \in \R^{\NumMC\numTimePoints\dimBM}
$
that
\begin{equation}
\begin{split}
  F_m( \theta ) 
  &=
  \frac{ 1 }{ \Batchsize }
  \left[
    \sum_{ \BatchVariable = 1 }^{ \Batchsize }
    \left|
      \eulerapprox( 
        \ParamRV_{ 
          m, \BatchVariable 
        }, 
        \WRVs^{ m, \BatchVariable } 
      )
      -
      \LRVANN(
        \ParamRV_{ 
          m, \BatchVariable 
        }, 
        \theta
      )
    \right|^2
  \right] 
\\&= 
  \frac{ 1 }{ \Batchsize }
  \left[
    \sum_{ \BatchVariable = 1 }^{ \Batchsize }
    \left|
      \eulerapprox( 
        \ParamRV_{ 
          m, \BatchVariable 
        }, 
        \WRVs^{ m, \BatchVariable } 
      )
      -
      \frac{ 1 }{ \NumMC }
      \left[ 
        \sum_{ \MCvariable = 1 }^{ \NumMC }
        \eulerapprox\big( 
          \ParamRV_{ 
            m, \BatchVariable 
          }, 
          ( 
            \theta_{ ( \MCvariable - 1 ) \numTimePoints\dimBM + k }
          )_{k \in \{1, 2, \ldots, \numTimePoints\dimBM \}}
        \big)
      \right]
    \right|^2
  \right] ,
\end{split}
\end{equation}
and for every 
	$m \in \N$
let 
$ 
  G_m 
  \colon \R^{\NumMC \numTimePoints \dimBM } \times \Omega 
  \to \R^{\NumMC \numTimePoints \dimBM } 
$
satisfy 
for all 
	$ \theta \in \R^{\NumMC\numTimePoints\dimBM} $, $ \omega \in \Omega $ 
that
\begin{equation}
  G_m( \theta, \omega )
  =
  ( \nabla_{ \theta } F_m )( \theta, \omega ).
\end{equation}

\subsection{Learning the random variables with stochastic gradient descent}
\label{MCEUler:SGD}

Let
$
  (\gamma_m)_{m \in \N} \subseteq (0,\infty) 
$
and let 
$ 
  \Theta \colon \N_0 \times \Omega \to \R^{\NumMC\numTimePoints\dimBM} 
$ 
satisfy for all
$ m \in \N $ that 
$
\Theta_0
=
(\WRVs^{0,1}, \WRVs^{0,2}, \ldots, \WRVs^{0,\NumMC})
$
and
\begin{equation}
  \Theta_{ m }
  =
  \Theta_{ m - 1 }
  - 
  \gamma_m 
  G_m( \Theta_{ m - 1 } ) 
  .
\end{equation}
\new{
For every sufficiently large 
$m \in \N$
we propose to employ the random function 
$
	\Param \times \Omega \ni (p, \omega) \mapsto \LRVANN(p, \Theta_m(\omega)) \in \R
$
as an approximation for the target function
$
	\Param \ni p \mapsto u(p) \in \R
$ in \eqref{MCEUler:eq0}.}

\subsection{Description of the proposed approximation algorithm}
\label{MCEUler:description}

\begin{algo}
Let $\dimParam, \dimBM, \numTimePoints, \NumMC, \Batchsize \in \N $, 
$ (\gamma_m)_{m \in \N} \subseteq (0,\infty) $,
let $ \Param \subseteq \R^{ \dimParam } $ be measurable, 
let $ ( \Omega, \mathcal{F}, \P ) $ be a probability space, 
let
$  
  \ParamRV_{ 
    m, \BatchVariable 
  }
  \colon \Omega \to \Param
  $, $ m, \BatchVariable \in \N $,
be i.i.d.\ random variables, 
let 
$ 
  \WRVs^{ m, \MCvariable } \colon [0, T] \times \Omega \to \R^{ \dimBM } 
$, 
$ m, \MCvariable \in \N_0 $, 
be i.i.d.\ standard Brownian motions,
assume that 
$
  (
    \ParamRV_{ 
      m, \BatchVariable 
    }
  )_{
    (m, \BatchVariable) \in \N^2
  }
$
and 
$
  (
    \WRVs^{ m, \MCvariable }
  )_{
    (m, \MCvariable) \in \N^2
  }
$
are independent, 
let
$
\eulerapprox \colon
\Param \times \R^{\numTimePoints \dimBM} \to
\R
$ satisfy for all $ p \in \Param $ that
\(
  \R^{\numTimePoints \dimBM} \ni w \mapsto \eulerapprox(p, w) \in \R
\)
is continuously differentiable,
for every 
	$m \in \N$
let 
$ 
  F_m 
  \colon \R^{\NumMC\numTimePoints\dimBM} 
  \times \Omega \to \R 
$
satisfy 
for all 
$
\theta = (\theta_1, \ldots, \theta_{\NumMC\numTimePoints\dimBM})
  \in \R^{\NumMC\numTimePoints\dimBM} 
$
that
\begin{multline}
  F_m( \theta ) 
  =
  \frac{ 1 }{ \Batchsize }
  \left[
    \sum_{ \BatchVariable = 1 }^{ \Batchsize }
    \left|
    \eulerapprox\Big( 
        \ParamRV_{ 
          m, \BatchVariable 
        }, 
        \sqrt{{\numTimePoints}/{T}}\,
        (\WRVs^{ m,  \BatchVariable  }_{T / \numTimePoints} - \WRVs^{ m,  \BatchVariable  }_{ 0 }, \WRVs^{ m,  \BatchVariable  }_{ 2 T / \numTimePoints } - \WRVs^{ m,  \BatchVariable  }_{ T / \numTimePoints }, \dots, \WRVs^{ m,  \BatchVariable  }_T -\WRVs^{ m,  \BatchVariable  }_{ (\numTimePoints-1) T / \numTimePoints } )    
       \Big)
       \vphantom{
       - \frac{ 1 }{ \NumMC }
      \left[ 
        \sum_{ \MCvariable = 1 }^{ \NumMC }
        \eulerapprox\Big( 
          \ParamRV_{ 
            m, \BatchVariable 
          }, 
          (\theta_{( \MCvariable - 1 ) \dimBM + k, n })_{(k, n) \in\{1, 2, \dots, \dimBM\} \times \{1, 2, \dots, \numTimePoints\}}          
        \Big)
      \right]
       }       \right.
       \vphantom{\left.
       - \frac{ 1 }{ \NumMC }
      \left[ 
        \sum_{ \MCvariable = 1 }^{ \NumMC }
        \eulerapprox\Big( 
          \ParamRV_{ 
            m, \BatchVariable 
          }, 
          (\theta_{( \MCvariable - 1 ) \dimBM + k, n })_{(k, n) \in\{1, 2, \dots, \dimBM\} \times \{1, 2, \dots, \numTimePoints\}}
        \Big)
      \right]
      \right|^2}
       \right.\\
      \left.\left.
       - \frac{ 1 }{ \NumMC }
      \left[ 
        \sum_{ \MCvariable = 1 }^{ \NumMC }
        \eulerapprox\Big( 
          \ParamRV_{ 
            m, \BatchVariable 
          }, 
          ( 
            \theta_{ ( \MCvariable - 1 ) \numTimePoints\dimBM + k }
          )_{k \in \{1, 2, \ldots, \numTimePoints\dimBM \}}         
        \Big)
      \right]
    \right|^2
  \right] ,
\end{multline}
for every 
	$m \in \N$
let 
$ 
  G_m 
  \colon \R^{\NumMC \numTimePoints\dimBM } \times \Omega 
  \to \R^{\NumMC \numTimePoints\dimBM } 
$
satisfy 
for all 
	$ \theta \in \R^{\NumMC\numTimePoints\dimBM } $, 
	$ \omega \in \Omega $ 
that
$
  G_m( \theta, \omega )
  =
  ( \nabla_{ \theta } F_m )( \theta, \omega )
  $,
and
let 
$ 
  \Theta = (\Theta^{(1)}, \ldots, \Theta^{(\NumMC\numTimePoints\dimBM)}) \colon \N_0 \times \Omega \to \R^{\NumMC\numTimePoints\dimBM } 
$ 
satisfy 
for all
	$ m \in \N $,
	$k \in \{1, 2, \ldots, \NumMC\}$,	
	$n \in \{1, 2, \ldots, \numTimePoints\}$
that 
$
	(
		\Theta_0^{([(k-1)\numTimePoints + n-1]\dimBM + 1)},
		\Theta_0^{([(k-1)\numTimePoints + n-1]\dimBM + 2)}, \allowbreak
		\ldots,
		\Theta_0^{([(k-1)\numTimePoints + n-1]\dimBM + \dimBM)}
	)
=
  \sqrt{{\numTimePoints}/{T}}
   (\WRVs^{0, k}_{nT/ \numTimePoints} - \WRVs^{0, k}_{(n-1)T/ \numTimePoints})
$
and
$
  \Theta_{ m }
  =
  \Theta_{ m - 1 }
  - 
  \gamma_m 
  G_m( \Theta_{ m - 1 } ) 
  .
$
\end{algo}

\section{LRV strategy in the case of multilevel Monte Carlo approximations}
\label{sect:MLMC}

In this section we consider the problem of approximating a parametric expectation involving a general measure space-valued random variable for which a sequence of finite-dimensional approximations is available (see \cref{MLMC:target,MLMC:discretization}).
We illustrate such a situation in \cref{MLMC:example} in which the original random variable is a Brownian motion driving an SDE and the finite-dimensional approximations consist of Euler discretizations of the SDE with decreasing step sizes. 
The employed proposal algorithm for the LRV strategy in this section is a form of MLMC method 
(cf.\ Heinrich \cite{Heinrich98} and Giles \cite{Giles2008} and cf., e.g., \cite{Heinrich2001,Creutzig09,hs99,Giles15})
as described in \cref{MLMC:proposal}.
Based on this proposal algorithm, the LRV strategy is applied to the considered approximation problem in \cref{MLMC:replacing,MLMC:Loss,MLMC:SGD} and, thereafter, summarized in one single framework in \cref{MLMC:description}.

\subsection{Parametric expectations involving measure space valued random variables}
\label{MLMC:target}

Let $ \dimParam \in \N $,
let $ \Param \subseteq \R^{ \dimParam } $ be measurable, 
let $(V, \mathcal{V})$ be a measurable space,
let $ \phi \colon \Param \times V \to \R $ be continuous, 
let $ u \colon \Param \to \R $ be a function, 
let $ ( \Omega, \F, \P ) $ be a probability space, 
let $ \WRV \colon \Omega \to V $ be a random variable,
and
assume for all $ p \in \Param $ that 
$
  \E\big[
    | \phi( p, \WRV ) |
  \big]
  < \infty
$
and
\begin{equation}
\label{MLMC:eq1}
  u( p ) = 
  \mathbb{E}\!\left[ 
    \phi( 
      p, \WRV
    )
  \right].
\end{equation}
The goal of this section is to derive an  algorithm to approximately compute the function
$
	u \colon \Param \to \R
$
given through the parametric expectation in \eqref{MLMC:eq1}.

\subsection{Finite-dimensional approximations}
\label{MLMC:discretization}

Let $\MLMCdims = (\MLMCdims_l)_{l \in \N_0} \colon \N_0 \to \N$ be non-decreasing,
let $\MLMCproj_l \colon \R^{\MLMCdims_l} \to \R^{\MLMCdims_{l-1}}$, $l \in \N$, be continuously differentiable,
let
$\MLMCRV_l \colon \Omega \to \R^{\MLMCdims_l}$, $l \in \N_0$, be random variables which satisfy for all $l \in \N$ that 
$\MLMCproj_{l}(\MLMCRV_{l}) = \MLMCRV_{l-1}$,
let 
$
\eulerapprox_l \colon
\Param \times \R^{\MLMCdims_l} \to
\R
$,
$l \in \N_0$,
satisfy for all 
	$l \in \N_0$,
	$ p \in \Param $
that
\(
  \R^{\MLMCdims_l} \ni w \mapsto \eulerapprox_l(p, w) \in \R
\)
is continuously differentiable,
and assume for all $ p \in \Param $, $l \in \N_0$ that
$\E[|\eulerapprox_l(p, \MLMCRV_l)|]<\infty$.
We think of
$\MLMCRV_l$, $l \in \N_0$, 
and 
$\eulerapprox_l$,
$l \in \N_0$,
as suitable approximations of $\WRV$ and 
$\phi$
in the sense that for all $ p \in \Param $ and all sufficiently large $l \in \N_0$ it holds that
\begin{equation}
  \label{MLMC:eq2} 
    \phi( 
      p, \WRV
    )
  \approx 
    \eulerapprox_l( 
    p,
    \MLMCRV_l
    ).
\end{equation}

\subsection{Multilevel Monte Carlo approximations}
\label{MLMC:proposal}

Let $\MLMClevel \in \N_0$, 
$ \NumMC_0, \NumMC_1, \ldots, \NumMC_\MLMClevel \in \N $
satisfy $\NumMC_0 \geq \NumMC_1 \geq \ldots \geq \NumMC_\MLMClevel$, 
let 
$ 
  W_l^{ m, \MCvariable } \colon \Omega \to \R^{ \MLMCdims_l } 
$, 
$ l, m, \MCvariable \in \N_0 $, 
be independent random variables,
and
assume 
for all
$ l,m, \MCvariable \in \N_0 $, $A \in \mathcal{B}( \R^{ \MLMCdims_l } ) $
that
$
  \P( W_l^{ m, \MCvariable} \in A ) = \P( \MLMCRV_l \in A ) 
$.
Observe that \eqref{MLMC:eq1} and \eqref{MLMC:eq2} suggest that 
for all $ p \in \Param $ it holds that
\begin{equation}
\label{MLMC:eq3}
\begin{split} 
    u( p ) = 
  \mathbb{E}\!\left[ 
    \phi( 
      p, \WRV
    )
  \right]
  &\approx
  \E[
  	\eulerapprox_\MLMClevel(p, \MLMCRV_\MLMClevel)
  ]
  =
  \E[
  	\eulerapprox_0(p, \MLMCRV_0)
  ]  
  +
  \sum_{l = 1}^\MLMClevel
  	\E[
  	  	\eulerapprox_l(p, \MLMCRV_l)
  	  	-
  	  	\eulerapprox_{l-1}(p, \MLMCproj_l(\MLMCRV_l))
  	  ] 
\\
  &\approx
  \frac{1}{\NumMC_0}
  \left[
	  \sum_{m = 1}^{\NumMC_0}
		 	\eulerapprox_0(p, W_0^{0,m})
  \right]
  +
  \sum_{l = 1}^\MLMClevel
    \frac{1}{\NumMC_l}
    \left[
      \sum_{m = 1}^{\NumMC_l}
      \eulerapprox_l(p, W_l^{0,m})
      -
      \eulerapprox_{l-1}(p, \MLMCproj_l(W_l^{0,m}))
    \right].
\end{split}
\end{equation}

\subsection{Multilevel Monte Carlo approximations for parametric stochastic differential equations}
\label{MLMC:example}

In the case 
where $\dimBM \in \N$,
$T, \boundary \in (0,\infty)$,
where 
	$ \dimParam = 1 + \dimBM $, 
	$
	\Param = [0,T] \times [-\boundary,\boundary]^\dimBM
	$, and
	$(V, \mathcal{V}) = \linebreak (C([0,T],\R^\dimBM), \mathcal{B}(C([0,T],\R^\dimBM)))$,
where
$\mu \colon \R^\dimBM \to \R^\dimBM$ is globally Lipschitz continuous,
where
$
X = \linebreak (X^{\xi, w}_t)_{(t, \xi, w) \in [0,T] \times \R^\dimBM \times C([0,T],\R^\dimBM)} \colon [0,T] \times \R^\dimBM \times C([0,T],\R^\dimBM) \to \R^{\dimBM}
$
satisfies for all
$t \in [0,T]$,
$\xi \in \R^\dimBM$,
$
w=(w_s)_{s\in [0,T]} \in C([0,T], \R^\dimBM)
$
that 
\begin{equation}
  X_t^{\xi, w} = \xi + \int_0^t \mu(X_s^{\xi, w}) \d s + w_t,
\end{equation}
where $g\colon [0,T] \times \R^{\dimBM} \to \R$
satisfies for all $t \in [0,T]$ that 
$\R^{\dimBM} \ni x \mapsto g(t, x) \in \R$
is continuously differentiable,
where it holds for all
	$p=(t,\xi)\in \Param$,
	$w \in C([0,T], \R^\dimBM)$
that
	$\phi(p,w) = g(t, X^{\xi,w}_t)$,
where 
$\WRV \colon \Omega \to C([0,T],\R^\dimBM) $ is a standard Brownian motion, 
where for all $l \in \N_0$ it holds that $\MLMCdims_l = 2^l \dimBM$,
where for all 
	$l \in \N_0$
it holds that
\begin{equation}
\label{MLMC:eq5}
\begin{split} 
	\MLMCRV_{l}
=
	\sqrt{{2^l}/{T}}\,
    (\WRV_{T / 2^l} - \WRV_0, \WRV_{ 2 T / 2^l } - \WRV_{ T / 2^l }, \dots, \WRV_T -\WRV_{ (2^l-1) T / 2^l } ),
\end{split}
\end{equation}
where
$
\mathcal{X}^{\xi, \theta, N} = (\mathcal{X}^{\xi, \theta, N}_t)_{t \in [0,T]}\colon [0,T] \to \R^\dimBM
$,
	$N \in \N$,
	$\theta \in \R^{N \dimBM}$,
	$\xi \in \R^{\dimBM}$,
satisfy for all 
	$N \in \N$,
	$\theta \in \R^{N \dimBM}$,
	$\xi \in \R^{\dimBM}$,
	$n\in \{1, 2, \dots N\}$,
	$t\in [\tfrac{(n-1)T}{N},\tfrac{nT}{N}]$
that $\mathcal{X}_0^{\xi,\theta,N} = \xi$
and
\begin{equation}
\label{MLMC:eq6}
  \mathcal{X}_{t}^{\xi, \theta, N} =  \mathcal{X}_{(n-1)T/N}^{\xi, \theta, N} + \big(\tfrac{tN}{T} - n + 1\big)\big(\tfrac{T}{N} \mu(\mathcal{X}_{(n-1)T/N}^{\xi, \theta, N}) + \tfrac{\sqrt{T}}{\sqrt{N}} (\theta_{(n-1)\dimBM + k})_{k \in \{1, 2, \ldots, \dimBM\}}\big),
\end{equation}
where for all
$l \in \N_0$,
$
p = (t, \xi) \in \Param
$,
$
\theta \in \R^{2^l \dimBM}
$
it holds that
$\eulerapprox_l( p, \theta ) = g(t, \mathcal{X}^{\xi, \theta, 2^l}_t)$,
and where for all
	$l \in \{0, 1, \ldots, L\}$
it holds that
$\NumMC_l = 2^{L-l}$
observe that
\begin{enumerate}[label=(\roman *)]
\item
it holds for all 
	$l \in \N$, 
	$\theta_1, \theta_2\, \ldots, \theta_{2^l}, \vartheta_1, \vartheta_2\, \ldots, \vartheta_{2^{l-1}} \in \R^{\dimBM}$ 
with 
$
	\forall \, n \in \{1, 2, \ldots, 2^{l-1}\}
	\colon
	\vartheta_n = \frac{1}{\sqrt{2}}(\theta_{2n} + \theta_{2n-1})
$
that
\begin{equation}
\begin{split} 
	\MLMCproj_l(\theta_1, \theta_2\, \ldots, \theta_{2^l}) = (\vartheta_1, \vartheta_2\, \ldots, \vartheta_{2^{l-1}}),
\end{split}
\end{equation}

\item 
it holds for all 
	$ p = (t,\xi)\in \Param $
that
\begin{equation}
\begin{split} 
	u(p) = u(t, \xi) = \E[g(t,X^{\xi,\WRV}_t) ]
\end{split}
\end{equation}
is the expectation of the test function $\R^d \ni x\mapsto g(t,x) \in \R$ evaluated at time $t$ of the solution process $(X^{\xi,\WRV}_s)_{s\in [0,T]}$ of the additive noise driven SDE in \eqref{MLMC:eq5},

\item 
it holds for all 
	$l \in \N_0$, 
	$p = (t,\xi)\in \Param $
that
$\eulerapprox_l(p, \MLMCRV_l)$ is an approximation
\begin{equation}
\begin{split} 
	\eulerapprox_l(p, \MLMCRV_l) = g(t, \mathcal{X}^{\xi, \MLMCRV_l, 2^l}_t)
\approx 
	g(t, X^{\xi,\WRV}_t) = \phi(p,\WRV)
\end{split}
\end{equation}
of $\phi(p,\WRV)$ based on linearly interpolated Euler-Maruyama approximations $(\mathcal{X}_{s}^{\xi, \MLMCRV_l, 2^l})_{s \in [0,T]}$ with $2^l$ timesteps of the solution $(X^{\xi,\WRV}_s)_{s\in [0,T]}$ of the SDE in \eqref{MLMC:eq5},
and

\item 
it holds for all 
	$p = (t,\xi)\in \Param $
that
\begin{equation}
\begin{split} 
  &\frac{1}{\NumMC_0}
  \left[
	  \sum_{m = 1}^{\NumMC_0}
		 	\eulerapprox_0(p, W_0^{0,m})
  \right]
  +
  \sum_{l = 1}^\MLMClevel
    \frac{1}{\NumMC_l}
    \left[
      \sum_{m = 1}^{\NumMC_l}
      \eulerapprox_l(p, W_l^{0,m})
      -
      \eulerapprox_{l-1}(p, \MLMCproj_l(W_l^{0,m}))
    \right]\\
&=
  \frac{1}{2^\MLMClevel}
  \left[
	  \sum_{m = 1}^{2^\MLMClevel}
		 	g(t, \mathcal{X}^{\xi, W_0^{0,m}, 1}_t)
  \right]
  +
  \sum_{l = 1}^\MLMClevel
    \frac{1}{2^{L-l}}
    \left[
      \sum_{m = 1}^{2^{L-l}}
      g(t, \mathcal{X}^{\xi, W_l^{0,m}, 2^l}_t)
      -
      g(t, \mathcal{X}^{\xi, \MLMCproj_l(W_l^{0,m}), 2^{l-1}}_t)
    \right]   
\end{split}
\end{equation}
is a MLMC approximation of $u(t,\xi)$ as proposed in Giles \cite{Giles2008}.
\end{enumerate}

\subsection{Replacing the random variables in multilevel Monte Carlo approximations}
\label{MLMC:replacing}

Let $\LRVANN \colon \Param \times \R^{  \NumMC_0 \MLMCdims_0 + \NumMC_1 \MLMCdims_1 + \ldots + \NumMC_\MLMClevel \MLMCdims_\MLMClevel}  \to \R$ satisfy for all
	$p \in \Param$, 
	$\theta = (\theta_1, \ldots, \theta_{\NumMC_0 \MLMCdims_0 + \NumMC_1 \MLMCdims_1 + \ldots + \NumMC_\MLMClevel \MLMCdims_\MLMClevel})
	\in \R^{ \NumMC_0 \MLMCdims_0 + \NumMC_1 \MLMCdims_1 + \ldots + \NumMC_\MLMClevel \MLMCdims_\MLMClevel } $
that
\begin{multline}  
\label{MLMC:eq4}
	\LRVANN(p, \theta)
=
     \frac{1}{\NumMC_0}
     \left[
   	  \sum_{m = 1}^{\NumMC_0}
   		 	\eulerapprox_0\big(p, (\theta_{(m-1)\MLMCdims_0 + k})_{k \in \{1, 2, \ldots, \MLMCdims_0\}}\big)
     \right]\\
      +
     \sum_{l = 1}^\MLMClevel
       \frac{1}{\NumMC_l}
       \Bigg[
         \sum_{m = 1}^{\NumMC_l}
         \eulerapprox_l\big(p, (\theta_{
         	\NumMC_0 \MLMCdims_0 + \NumMC_1 \MLMCdims_1 + \ldots + \NumMC_{l-1} \MLMCdims_{l-1} + (m-1)\MLMCdims_l + k
         })_{k \in \{1, 2, \ldots, \MLMCdims_l\}}\big)\\
          -
         \eulerapprox_{l-1}\big(p, \MLMCproj_l(
         	(\theta_{
         	         	\NumMC_0 \MLMCdims_0 + \NumMC_1 \MLMCdims_1 + \ldots + \NumMC_{l-1} \MLMCdims_{l-1} + (m-1)\MLMCdims_l + k
         	         })_{k \in \{1, 2, \ldots, \MLMCdims_l\}}
         )\big)
       \Bigg].
\end{multline}
Note that \eqref{MLMC:eq3} and \eqref{MLMC:eq4} suggest that 
for all $ p \in \Param $ it holds that
\begin{equation}
\begin{split}
  u( p )
  &\approx
  \frac{1}{\NumMC_0}
  \left[
	  \sum_{m = 1}^{\NumMC_0}
		 	\eulerapprox_0(p, W_0^{0,m})
  \right]
  +
  \sum_{l = 1}^\MLMClevel
    \frac{1}{\NumMC_l}
    \left[
      \sum_{m = 1}^{\NumMC_l}
      \eulerapprox_l(p, W_l^{0,m})
      -
      \eulerapprox_{l-1}(p, \MLMCproj_l(W_l^{0,m}))
    \right] \\
  &=
    \LRVANN\big(p, (
      W_0^{0, 1}, W_0^{0, 2}, \ldots, W_0^{0, \NumMC_1},
      W_1^{0, 1}, W_1^{0, 2}, \ldots, W_1^{0, \NumMC_2},
      \ldots,
      W_\MLMClevel^{0, 1}, W_\MLMClevel^{0, 2}, \ldots, W_\MLMClevel^{0, \NumMC_\MLMClevel}
    )\big).
\end{split}
\end{equation}

\subsection{Random loss functions for fixed random variables in multilevel Monte Carlo approximations}
\label{MLMC:Loss}
Let $\Batchsize \in \N$, $\MLMCreference \in \N_0$, 
let 
$  
  \ParamRV_{ 
    m, \BatchVariable 
  }
  \colon \Omega \to \Param
$, $ m, \BatchVariable \in \N $, 
be i.i.d.\ random variables, 
assume that 
$
  (
    \ParamRV_{ 
      m, \BatchVariable 
    }
  )_{
    (m, \BatchVariable) \in \N^2
  }
$
and
$
  (
  W_l^{ m, \MCvariable }
  )_{
    (l, m, \MCvariable) \in (\N_0)^3
  }
$
are independent,
for every 
	$m \in \N$
let
$
  F_m 
  \colon \allowbreak\R^{  \NumMC_0 \MLMCdims_0 + \NumMC_1 \MLMCdims_1 + \ldots + \NumMC_\MLMClevel \MLMCdims_\MLMClevel } 
  \times\Omega\to\R
$
satisfy 
for all 
$\theta = (\theta_1, \ldots, \theta_{\NumMC_0 \MLMCdims_0 + \NumMC_1 \MLMCdims_1 + \ldots + \NumMC_\MLMClevel \MLMCdims_\MLMClevel})
\in \R^{ \NumMC_0 \MLMCdims_0 + \NumMC_1 \MLMCdims_1 + \ldots + \NumMC_\MLMClevel \MLMCdims_\MLMClevel } $
that
\begin{equation}
\begin{split} 
	F_m( \theta ) 
&=
	\frac{ 1 }{ \Batchsize }
	\Bigg[
		\sum_{ \BatchVariable = 1 }^{ \Batchsize }
		\left|
			\eulerapprox_\MLMCreference\big( 
				\ParamRV_{ 
					m, \BatchVariable 
				},
				W_\MLMCreference^{m,  \BatchVariable}
			\big)
			-
			\LRVANN(
				\ParamRV_{ 
					m, \BatchVariable 
				},
				\theta
			)
		\right|^2
	\Bigg],
\end{split}
\end{equation}
and for every 
	$m \in \N$
let 
$ 
  G_m 
  \colon \R^{\NumMC_0 \MLMCdims_0 + \NumMC_1 \MLMCdims_1 + \ldots + \NumMC_\MLMClevel \MLMCdims_\MLMClevel} \times \Omega 
  \to \R^{\NumMC_0 \MLMCdims_0 + \NumMC_1 \MLMCdims_1 + \ldots + \NumMC_\MLMClevel \MLMCdims_\MLMClevel } 
$
satisfy 
for all 
  $ \theta \in \R^{\NumMC_0 \MLMCdims_0 + \NumMC_1 \MLMCdims_1 + \ldots + \NumMC_\MLMClevel \MLMCdims_\MLMClevel } $, 
  $ \omega \in \Omega $ 
that
\begin{equation}
  G_m( \theta, \omega )
  =
  ( \nabla_{ \theta } F_m )( \theta, \omega ).
\end{equation}

\subsection{Learning the random variables with stochastic gradient descent}
\label{MLMC:SGD}
Let $ (\gamma_m)_{m \in \N} \subseteq (0,\infty) $ and 
let 
$ 
  \Theta \colon \N_0 \times \Omega \to \R^{ \NumMC_0 \MLMCdims_0 + \NumMC_1 \MLMCdims_1 + \ldots + \NumMC_\MLMClevel \MLMCdims_\MLMClevel } 
$ 
satisfy 
for all $ m \in \N $ that 
$ 
  \Theta_0 = 
  (
    W_0^{0, 1}, W_0^{0, 2}, \ldots, W_0^{0, \NumMC_1},
    W_1^{0, 1}, W_1^{0, 2}, \ldots, W_1^{0, \NumMC_2},
    \ldots,
    W_\MLMClevel^{0, 1}, W_\MLMClevel^{0, 2}, \ldots, W_\MLMClevel^{0, \NumMC_\MLMClevel}
  )
$
and
\begin{equation}
  \Theta_{ m }
  =
  \Theta_{ m - 1 }
  - 
  \gamma_m 
  G_m( \Theta_{ m - 1 } ) 
  .
\end{equation}
\new{
For every sufficiently large 
$m \in \N$
we propose to employ the random function 
$
	\Param \times \Omega \ni (p, \omega) \mapsto \LRVANN(p, \Theta_m(\omega)) \in \R
$
as an approximation for the target function
$
	\Param \ni p \mapsto u(p) \in \R
$ in \eqref{MLMC:eq1}.}

\subsection{Description of the proposed approximation algorithm}
\label{MLMC:description}
\begin{algo}
Let $ \dimParam, \Batchsize \in \N $, $\MLMClevel, \MLMCreference \in \N_0$, 
$ \NumMC_0, \NumMC_1, \ldots, \NumMC_\MLMClevel \in \N $,
$(\MLMCdims_l)_{l \in \N_0} \subseteq \N$, $ (\gamma_m)_{m \in \N} \subseteq (0,\infty) $,
let $ \Param \subseteq \R^{ \dimParam } $ be measurable, 
let 
$
\eulerapprox_l \colon
\Param \times \R^{\MLMCdims_l} \to
\R
$,
$l \in \N_0$,
satisfy for all 
	$l \in \N_0$,
	$ p \in \Param $ 
that
\(
  \R^{\MLMCdims_l} \ni w \mapsto \eulerapprox_l(p, w) \in \R
\)
is continuously differentiable,
let $\MLMCproj_l \colon \R^{\MLMCdims_l} \to \R^{\MLMCdims_{l-1}}$, $l \in \{1, 2, \ldots, L \}$, be continuously differentiable,
let $\LRVANN \colon \Param \times \R^{  \NumMC_0 \MLMCdims_0 + \NumMC_1 \MLMCdims_1 + \ldots + \NumMC_\MLMClevel \MLMCdims_\MLMClevel}  \to \R$ satisfy for all
	$p \in \Param$, 
	$\theta = (\theta_1, \ldots, \theta_{\NumMC_0 \MLMCdims_0 + \NumMC_1 \MLMCdims_1 + \ldots + \NumMC_\MLMClevel \MLMCdims_\MLMClevel})
	\in \R^{ \NumMC_0 \MLMCdims_0 + \NumMC_1 \MLMCdims_1 + \ldots + \NumMC_\MLMClevel \MLMCdims_\MLMClevel } $
that
\begin{multline}  
	\LRVANN(p, \theta)
=
     \frac{1}{\NumMC_0}
     \left[
   	  \sum_{m = 1}^{\NumMC_0}
   		 	\eulerapprox_0\big(p, (\theta_{(m-1)\MLMCdims_0 + k})_{k \in \{1, 2, \ldots, \MLMCdims_0\}}\big)
     \right]\\
      +
     \sum_{l = 1}^\MLMClevel
       \frac{1}{\NumMC_l}
       \Bigg[
         \sum_{m = 1}^{\NumMC_l}
         \eulerapprox_l\big(p, (\theta_{
         	\NumMC_0 \MLMCdims_0 + \NumMC_1 \MLMCdims_1 + \ldots + \NumMC_{l-1} \MLMCdims_{l-1} + (m-1)\MLMCdims_l + k
         })_{k \in \{1, 2, \ldots, \MLMCdims_l\}}\big)\\
          -
         \eulerapprox_{l-1}\big(p, \MLMCproj_l(
         	(\theta_{
         	         	\NumMC_0 \MLMCdims_0 + \NumMC_1 \MLMCdims_1 + \ldots + \NumMC_{l-1} \MLMCdims_{l-1} + (m-1)\MLMCdims_l + k
         	         })_{k \in \{1, 2, \ldots, \MLMCdims_l\}}
         )\big)
       \Bigg],
\end{multline}
let $ ( \Omega, \F, \P ) $ be a probability space, 
let 
$ 
  W_l^{ m, \MCvariable } \colon \Omega \to \R^{ \dimBM } 
$, 
$l, m, \MCvariable \in \N_0 $, 
be independent random variables which
satisfy 
for all
$l, m, \MCvariable \in \N_0 $, $A \in \mathcal{B}( \R^{ \MLMCdims_l } ) $
that
$
  \P( W_l^{ m, \MCvariable} \in A ) = \P( W_l^{ 0, 0} \in A ) 
$,
let 
$  
  \ParamRV_{ 
    m, \BatchVariable 
  }
  \colon \Omega \to \Param
$, $ m, \BatchVariable \in \N $, 
be i.i.d.\ random variables, 
assume that 
$
  (
    \ParamRV_{ 
      m, \BatchVariable 
    }
  )_{
    (m, \BatchVariable) \in \N^2
  }
$
and 
$
  (
  W_l^{ m, \MCvariable }
  )_{
    (l, m, \MCvariable) \in (\N_0)^3
  }
$
are independent,
for every 
	$m \in \N$
let
$
F_m 
  \colon \allowbreak\R^{  \NumMC_0 \MLMCdims_0 + \NumMC_1 \MLMCdims_1 + \ldots + \NumMC_\MLMClevel \MLMCdims_\MLMClevel } 
  \times\Omega\to\R
$
satisfy 
for all 
$
\theta 
  \in \R^{ \NumMC_0 \MLMCdims_0 + \NumMC_1 \MLMCdims_1 + \ldots + \NumMC_\MLMClevel \MLMCdims_\MLMClevel} 
$
that
\begin{equation}
\label{T_B_D}
\begin{split} 
	F_m( \theta ) 
&=
	\frac{ 1 }{ \Batchsize }
	\Bigg[
		\sum_{ \BatchVariable = 1 }^{ \Batchsize }
		\left|
			\eulerapprox_\MLMCreference\big( 
				\ParamRV_{ 
					m, \BatchVariable 
				},
				W_\MLMCreference^{m,  \BatchVariable}
			\big)
			-	
			\LRVANN(
				\ParamRV_{ 
					m, \BatchVariable 
				},
				\theta)
		\right|^2
	\Bigg],
\end{split}
\end{equation}
for every 
	$m \in \N$
let 
$ 
  G_m 
  \colon \R^{\NumMC_0 \MLMCdims_0 + \NumMC_1 \MLMCdims_1 + \ldots + \NumMC_\MLMClevel \MLMCdims_\MLMClevel} \times \Omega 
  \to \R^{\NumMC_0 \MLMCdims_0 + \NumMC_1 \MLMCdims_1 + \ldots + \NumMC_\MLMClevel \MLMCdims_\MLMClevel } 
$
satisfy 
for all 
$ \theta \in \linebreak \R^{\NumMC_0 \MLMCdims_0 + \NumMC_1 \MLMCdims_1 + \ldots + \NumMC_\MLMClevel \MLMCdims_\MLMClevel } $, 
$ \omega \in \Omega $
that
$
  G_m( \theta, \omega )
  =
  ( \nabla_{ \theta } F_m )( \theta, \omega )
$,
and
let 
$ 
  \Theta \colon \N_0 \times \Omega \to \R^{ \NumMC_0 \MLMCdims_0 + \NumMC_1 \MLMCdims_1 + \ldots + \NumMC_\MLMClevel \MLMCdims_\MLMClevel } 
$ 
satisfy 
for all $ m \in \N $ that 
$ 
  \Theta_0 = 
  (
    W_0^{0, 1}, W_0^{0, 2}, \ldots, W_0^{0, \NumMC_1},
    W_1^{0, 1}, W_1^{0, 2}, \allowbreak \ldots, W_1^{0, \NumMC_2}, 
    \ldots,
    W_\MLMClevel^{0, 1}, W_\MLMClevel^{0, 2}, \ldots, W_\MLMClevel^{0, \NumMC_\MLMClevel}
  )
$
and
\begin{equation}
  \Theta_{ m }
  =
  \Theta_{ m - 1 }
  - 
  \gamma_m 
  G_m( \Theta_{ m - 1 } ) 
  .
\end{equation}
\end{algo}

\section{LRV strategy in the case of multilevel Picard approximations}
\label{sect:MLP}

In this section we consider the case where the target function to be approximated with the LRV strategy is given as the solution of a suitable stochastic fixed point equation (SFPE); see \eqref{MLP:eq1} in \cref{MLP:target}.
The setup in \cref{MLP:target} includes, as important special cases, the solutions of several semilinear parabolic PDEs. 
Two examples of such semilinear parabolic PDEs
(heat PDEs with Lipschitz nonlinearities and
Black-Scholes PDEs with Lipschitz nonlinearities)
are presented  in \cref{MLP:example}.
As proposal algorithms for the LRV strategy we present in \cref{MLP:proposal} a slight generalization of the MLP algorithm for semilinear PDEs in Hutzenthaler et al.\ \cite{Hutzenthaler_2020} (cf.\ also E et al.\ \cite{E2019,EHutzenthalerJentzenKruse16published}).
In analogy to the previous sections,  in \cref{MLP:replacing,MLP:Loss,MLP:SGD} an algorithm for the considered approximation problem is derived based on the LRV strategy with this MLP algorithm as proposal algorithm. 
Finally, the problem and the algorithm are summarized in one single framework in \cref{MLP:description}.

\subsection{Parametric solutions of stochastic fixed point equations}
\label{MLP:target}

Let $ \dimParam, \dimWRV \in \N $,
let $ \Param \subseteq \R^{ \dimParam } $ be measurable, 
let 
$ \phi \colon \Param \times \R^\dimWRV \times \R \to \R $,
$\MLPuEval \colon \Param \times \R^\dimWRV \to \Param$, and
$ u \colon \Param \to \R $ be measurable, 
let $ ( \Omega, \F, \P ) $ be a probability space, 
let $ \WRV \colon \Omega \to \R^\dimWRV $ be a random variable,
and
assume for all $ p \in \Param $ that 
$
  \E\big[
    | \phi( p, \WRV, u( \MLPuEval(p, \WRV))) |
  \big]
  < \infty
$
and
\begin{equation}
\label{MLP:eq1}
  u( p ) = 
  \mathbb{E} \big[ 
    \phi\big( p, \WRV, u\big( \MLPuEval(p, \WRV)\big)\big)
  \big].
\end{equation}
The goal of this section is to derive an  algorithm to approximately compute the function
$
	u \colon \Param \to \R
$
given through the SFPE in \eqref{MLP:eq1}.

\subsection{Approximations for parametric semilinear partial differential equations}
\label{MLP:example}
\subsubsection{Heat partial differential equations with Lipschitz nonlinearities}
In the case 
where $T \in (0,\infty)$, $\dimBM \in \N$,
where 
	$\dimParam = \dimWRV = 1 + \dimBM$ and
	$\Param = [0,T] \times\R^\dimBM$,
where $u \in C([0,T] \times\R^\dimBM,\R)$ is at most polynomially growing,
where $f,g \in C(\R^\dimBM, \R)$ are at most polynomially growing,
where for all 
	$p = (t, x) \in \Param$,
	$w = (w_1, w_2) \in \R^{\dimBM} \times [0,1]$,
	$v \in \R$
it holds that
$
	\phi(p, w, v)
=
	g(x + (T-t)^{\nicefrac{1}{2}}w_1) + (T-t)f(v)
$
and
$
	\MLPuEval(p, w)
=
	(t + w_2(T-t), x + (w_2(T-t))^{\nicefrac{1}{2}}w_1)
$,
where $N \colon \Omega \to \R^{\dimBM}$ is a standard normal random vector,
where $R \colon \Omega \to [0,1]$ is a continuous uniformly distributed random variable on $[0,1]$ (is a $\mathcal{U}_{[0,1]}$-distributed random variable),
where $N$ and $R$ are independent, and
where $\WRV = (N, R)$
observe that 
\begin{enumerate}[label=(\roman *)]
\item 
it holds for all
	$p = (t, x) \in \Param$
that
\begin{equation}
\begin{split} 
	u(p) = u(t, x) 
&= 
	\Exp{
		g(x + (T-t)^{\nicefrac{1}{2}}N) + (T-t) f\big(u(t + R(T-t), x + (R(T-t))^{\nicefrac{1}{2}}N)\big)
	}\\
&=
	\Exp{
		g(x + (T-t)^{\nicefrac{1}{2}}N) + \int_t^T f\big(u(s, x + (s-t)^{\nicefrac{1}{2}}N)\big) \d s 
	}
\end{split}
\end{equation}
and
\item 
it holds that $u$ is a viscosity solution of 
\begin{equation}
	(\tfrac{\partial u}{\partial t})(t,x) 
	+
	\tfrac{1}{2} 
	(\Delta_x u)(t, x) 
	+
	f(u(t, x))
=
	0
\end{equation}
with $u(T, x) = g(x)$ for $(t, x) \in (0,T) \times \R^\dimBM$
(cf.\ Beck et al.\ \cite[Theorem 1.1]{beck21}).
\end{enumerate}

\subsubsection{Black-Scholes partial differential equations with Lipschitz nonlinearities}

In the case 
where $T \in (0,\infty)$, $\dimBM \in \N$,
where 
	$\dimParam = \dimWRV = 1 + 3\dimBM $ and
	$\Param = [0,T] \times\R^\dimBM\times\R^\dimBM\times\R^\dimBM$,
where for all 
	$\alpha, \beta \in \R^\dimBM$
it holds that 
$
	[0,T] \times\R^\dimBM \ni (t, x) 
	\mapsto
	u(t,x,\alpha,\beta) \in \R
$ 
is continuous and at most polynomially growing,
where 
	$X = (X^{(i)})_{i \in \{1, 2, \ldots, \dimBM\}} \colon [0,T] \times \R^\dimBM\times\R^\dimBM\times\R^\dimBM \times \R^\dimBM \to \R^\dimBM$
satisfies for all 
	$i \in \{1, 2, \ldots, \dimBM\}$
	$t \in [0,T]$,
	$
		x = (x_1,\ldots,x_\dimBM), 
		\alpha = (\alpha_1,\ldots,\alpha_\dimBM), 
		\beta = (\beta_1,\ldots,\beta_\dimBM),
		n = (n_1, \ldots, n_\dimBM)
	\in
		\R^{\dimBM}
	$
that
\begin{equation}
\begin{split} 
	X^{(i)}(t, x, \alpha, \beta, n)
=
	x_i \exp\!\left(
		\big(\alpha_i - \tfrac{|\beta_i|^2}{2}\big)t + t^{1/2}\beta_i n_i 
	\right),
\end{split}
\end{equation}
where $f,g \in C(\R^\dimBM, \R)$ are at most polynomially growing,
where for all 
	$p = (t, x, \alpha, \beta) \in \Param$,
	$w = (w_1, w_2) \in \R^{\dimBM} \times [0,1]$,
	$v \in \R$
it holds that
$
	\phi(p, w, v)
=
	g(X(T-t, x, \alpha, \beta, w_1)) + (T-t)f(v)
$
and
$
	\MLPuEval(p, w)
=
	(t + w_2(T-t),  X(w_2(T-t), x, \alpha, \beta, w_1))
$,
where $N \colon \Omega \to \R^{\dimBM}$ is a standard normal random vector,
where $R \colon \Omega \to [0,1]$ is an $\mathcal{U}_{[0,1]}$-distributed random variable,
where $N$ and $R$ are independent, and
where $\WRV = (N, R)$
observe that 
\begin{enumerate}[label=(\roman *)]
\item 
it holds for all
	$p = (t, x, \alpha, \beta) \in \Param$
that
\begin{equation}
\begin{split} 
	u(p) 
&= 
	u(t, x, \alpha, \beta) \\
&= 
	\Exp{
		g\big(
			X(T-t, x, \alpha, \beta, N)
		\big) 
		+ 
		(T-t)
		f\Big(
			u\big(t + R(T-t),  X(R(T-t), x, \alpha, \beta, N)\big)
		\Big)
	}\\
&=
	\Exp{
		g\big(
			X(T-t, x, \alpha, \beta, N)
		\big) 
		+ 
		\int_t^T
			f\big(
				u\big(s,  X(s-t, x, \alpha, \beta, N)\big)
			\big)
		\d s
	}
\end{split}
\end{equation}
and
\item 
it holds for all
$
	\alpha = (\alpha_1,\ldots,\alpha_\dimBM), 
	\beta = (\beta_1,\ldots,\beta_\dimBM)
	\in \R^\dimBM
$
that 
$
	[0,T] \times\R^\dimBM \ni (t, x) 
	\mapsto
	u(t,x,\alpha,\beta) \in \R
$ 
is a viscosity solution of 
\begin{multline}
	\big(\tfrac{\partial u}{\partial t}\big)(t,x,\alpha,\beta) 
	+
	\left[
  	\sum_{i = 1}^d 
	    \tfrac{ | \beta_i |^2 |x_i|^2}{2} 
  	    \big( \tfrac{\partial^2 u}{\partial (x_i)^2 } \big)(t,x,\alpha,\beta)
    \right] \\
    +
	\left[
    \sum_{i = 1}^d 
	    \alpha_i
	     x_i
	    \big( \tfrac{\partial u}{\partial x_i } \big)(t,x,\alpha,\beta)
	\right]
	+
  f( u(t, x,\alpha,\beta) )
=
	0
\end{multline}
with $u(T, x) = g(x)$ for $(t, x) = (t, x_1, x_2, \ldots, x_\dimBM) \in (0,T) \times \R^\dimBM$
(cf., e.g., Beck et al.\ \cite[Theorem 1.1]{beck21}).
\end{enumerate}

\subsection{Multilevel Picard approximations}
\label{MLP:proposal}
Let $\NumMC \in \N$, $\MLPindexset = \cup_{n = 1}^\infty \Z^n$, 
let $\MLPRVs^{\MLPindex} \colon \Omega \to \R^\dimWRV$, $\MLPindex \in \MLPindexset$, be i.i.d.\ random variables which satisfy for all 
	$ A \in \mathcal{B}( \R^{ \dimWRV } ) $ 
that 
$
  \P( \MLPRVs^{ 0 } \in A ) = \P( \WRV \in A ) 
$, and
let $\MLPAlg_n^\MLPindex \colon \Param \times \Omega \to \R$, $n \in \Z$, $\MLPindex \in \MLPindexset$, satisfy\footnote
{
Note that for all
	$n \in \N$,
	$\MLPindex  = (\MLPindex_1, \ldots, \MLPindex_n) \in \Z^{n}$,
	$k, m \in \Z$ we denote by 
	$(\MLPindex, k, m) \in \Z^{n+2}$
	the vector given by
	$(\MLPindex, k, m) = (\MLPindex_1, \MLPindex_2, \ldots, \MLPindex_n, k, m)$.
} for all 
	$n \in \Z$, 
	$\MLPindex \in \MLPindexset$,
	$p \in \Param$
that
\begin{multline}
\label{MLP:eq2}
	\MLPAlg_n^\MLPindex(p)
=
	\sum_{k = 0}^{n-1}
		\frac{1}{\NumMC^{n-k}}
		\sum_{\MCvariable = 1}^{\NumMC^{n-k}} \Big[
			\phi\big(
				p, 
				\MLPRVs^{(\MLPindex,k,\MCvariable)},
				\MLPAlg_k^{(\MLPindex,k,\MCvariable)}\big(
					\MLPuEval(p, \MLPRVs^{(\MLPindex,k,\MCvariable)})
				\big)
			\big)\\
			-
			\mathbbm{1}_\N(k)
			\phi\big(
				p, 
				\MLPRVs^{(\MLPindex,k,\MCvariable)},
				\MLPAlg_{k-1}^{(\MLPindex,k-1,-\MCvariable)}\big(
					\MLPuEval(p, \MLPRVs^{(\MLPindex,k,\MCvariable)})
				\big)
			\big)
		\Big].
\end{multline}
Note that under suitable integrability conditions \eqref{MLP:eq2} implies for all
	$n \in \N$,
	$p \in \Param$,
	$\MLPindex \in \MLPindexset$
that
$
	\MLPAlg_0^\MLPindex(p) = 0
$
and
\begin{equation}
\label{T_B_D}
\begin{split} 
	\Exp{
		\MLPAlg_n^\MLPindex(p)
	}
&=
	\sum_{k = 0}^{n-1}
		\frac{1}{\NumMC^{n-k}}
		\sum_{\MCvariable = 1}^{\NumMC^{n-k}} \Big[
			\Exp{
				\phi\big(
					p, 
					\MLPRVs^{(\MLPindex,k,\MCvariable)},
					\MLPAlg_k^{(\MLPindex,k,\MCvariable)}\big(
						\MLPuEval(p, \MLPRVs^{(\MLPindex,k,\MCvariable)})
					\big)
				\big)
			}\\
			&\quad-
			\mathbbm{1}_\N(k)
			\Exp{
				\phi\big(
					p, 
					\MLPRVs^{(\MLPindex,k,\MCvariable)},
					\MLPAlg_{k-1}^{(\MLPindex,k-1,-\MCvariable)}\big(
						\MLPuEval(p, \MLPRVs^{(\MLPindex,k,\MCvariable)})
					\big)
				\big)
			}
		\Big]\\
&=
	\sum_{k = 0}^{n-1}
	\Big[
			\Exp{
				\phi\big(
					p, 
					\MLPRVs^{0},
					\MLPAlg_k^{0}\big(
						\MLPuEval(p, \MLPRVs^{0})
					\big)
				\big)
			}
			-
			\mathbbm{1}_\N(k)
			\Exp{
				\phi\big(
					p, 
					\MLPRVs^{0},
					\MLPAlg_{k-1}^{0}\big(
						\MLPuEval(p, \MLPRVs^{0})
					\big)
				\big)
			}
		\Big]\\
&=
	\Exp{
		\phi\big(
			p, 
			\MLPRVs^{0},
			\MLPAlg_{n-1}^{0}\big(
				\MLPuEval(p, \MLPRVs^{0})
			\big)
		\big)
	}.
\end{split}
\end{equation}
Induction thus shows that for every $n \in \N$ the identically distributed random functions 
$\MLPAlg_n^\MLPindex \colon \Param \times \Omega \to \R$,  $\MLPindex \in \MLPindexset$, correspond in expectation to the $n$-th fixed point iterate of the fixed point equation in \eqref{MLP:eq1}.
For every $n \in \N$ the recursive definition of $(\MLPAlg_n^\MLPindex)_{\MLPindex \in \MLPindexset}$ in \eqref{MLP:eq2} thus represents an approximated fixed point iteration step in which the expectation of the fixed point iteration is approximated by a MLMC sum over previously computed approximated fixed point iterates $(\MLPAlg_{k}^\MLPindex)_{(\MLPindex, k) \in \MLPindexset \times \{0, 1, \ldots, n-1\}}$.
Under suitable assumptions
(see, e.g., Hutzenthaler et al.\ \cite{HutzenthalerJentzenvW20} for precise assumptions and a more detailed derivation of MLP algorithms in the case of semilinear PDEs)
we expect for sufficiently large $n \in \N$ that
\begin{equation}
\label{MLP:eq3}
\begin{split} 
	u(p)
\approx
	\MLPAlg_n^0(p).
\end{split}
\end{equation}

\subsection{Replacing the random variables in multilevel Picard approximations}
\label{MLP:replacing}

\newcommand{\MLPsubcount}{c}

Let $(\MLPcount_n)_{n \in \Z} \subseteq \N_0$ satisfy for all
	$n \in \Z$
that
\begin{equation}
\label{MLP:eq4}
\begin{split} 
	\MLPcount_n = \sum_{k = 0}^{n-1} \NumMC^{n-k}(1 + \MLPcount_k + \MLPcount_{k-1}),
\end{split}
\end{equation}
for every 
	$n, \MCvariable \in \N$, $l \in \{1, 2, \ldots, n\}$
let $\MLPsubcount_{n,l, \MCvariable} \in \N$ satisfy
\begin{equation}
\label{MLP:eq4.1}
\begin{split} 
	\MLPsubcount_{n,l, \MCvariable} 
= 
	\left[
		\sum_{k = 0}^{l-1} \NumMC^{n-k}(1 + \MLPcount_k +\MLPcount_{k-1})
	\right]
	+
	(\MCvariable-1) (1 + \MLPcount_l + \MLPcount_{l-1})
	+ 
	1,
\end{split}
\end{equation}
let $\LRVANN_n \colon \Param \times \R^{\MLPcount_n \dimWRV}$, $n \in \N_0$, satisfy for all
	$n \in \N$,
	$p \in \Param$,
	$\theta_1,\theta_2,\ldots, \theta_{\MLPcount_n} \in  \R^{\dimWRV}$
that
	$\LRVANN_0(p) = 0$
and
\begin{multline}
\label{MLP:eq5}
	\LRVANN_n(p, (\theta_r)_{r \in \{1, 2, \ldots, \MLPcount_n \}})
=
	\left[
		\frac{1}{\NumMC^{n}}
		\sum_{\MCvariable = 1}^{\NumMC^{n}}
			\phi\big(
				p, 
				\theta_{\MCvariable},
				0
			\big)
	\right]
\\ 
	+
	\sum_{k = 1}^{n-1}
		\frac{1}{\NumMC^{n-k}}
		\sum_{\MCvariable = 1}^{\NumMC^{n-k}} \bigg[
			\phi\Big(
				p, 
				\theta_{\MLPsubcount_{n,k,\MCvariable}}, 
				\LRVANN_k \big(
					\MLPuEval(p, \theta_{\MLPsubcount_{n,k,\MCvariable}})
					,
					(\theta_{\MLPsubcount_{n,k,\MCvariable} + r})_{r \in \{1, 2, \ldots, \MLPcount_k\}}					
				\big)
			\Big)
\\ 
			-
			\phi\Big(
				p, 
				\theta_{\MLPsubcount_{n,k,\MCvariable}},
				\LRVANN_{k-1}\big(
					\MLPuEval(p, \theta_{\MLPsubcount_{n,k,\MCvariable}})
					,
					(\theta_{\MLPsubcount_{n,k,\MCvariable} + \MLPcount_k + r})_{r \in \{1, 2, \ldots, \MLPcount_{k-1}\}}
				\big)
			\Big)
		\bigg],
\end{multline}
let $\MLPlevel \in \N$,
assume for all $p \in \Param$ that
$\R^{ \MLPApproxDim } \ni \theta \mapsto \LRVANN( p, \theta ) \in \R$ is continuously differentiable,  
and
let 
$\MLPsingleRV_k \colon \Omega \to \R^{\dimWRV}$, 
	$k \in \{1, 2, \ldots, \MLPcount_\MLPlevel\}$,
be i.i.d.\ random variables which satisfy
for all
	$B \in \mathcal{B}(\R^\dimWRV)$
that
$
	\P(\MLPsingleRV_1 \in B)
=
	\P(\WRV \in B)
$.
Observe that induction shows that for all
	$n \in \N_0$,
	$p \in \Param$,
	$\MLPindex \in \MLPindexset$
the number $\MLPcount_n \in \N$ corresponds, roughly speaking, to the number of realizations of $\dimWRV$-dimensional random variables required to compute one random realization of $\MLPAlg_n^\MLPindex(p)$.
Moreover, note that \eqref{MLP:eq2},  \eqref{MLP:eq4},  \eqref{MLP:eq4.1}, and \eqref{MLP:eq5} assure that for all 
	$p \in \Param$,
	$B \in \mathcal{B}(\R)$
it holds that
\begin{equation}
\label{MLP:eq6}
\begin{split} 
	\P\big(
		\LRVANN_\MLPlevel\big(
			p, (\MLPsingleRV_1, \MLPsingleRV_2, \ldots, \MLPsingleRV_{\MLPcount_\MLPlevel})
		\big) 
		\in B
	\big)
=
	\P(\MLPAlg_\MLPlevel^0(p) \in B).
\end{split}
\end{equation}
Combining this and \eqref{MLP:eq3} suggests that  for all
	$p \in \Param$
it holds that
\begin{equation}
\label{MLP:eq7}
\begin{split} 
	\LRVANN_\MLPlevel(p, \MLPsingleRV)
\approx
	u(p).
\end{split}
\end{equation}

\subsection{Random loss functions for fixed random variables in multilevel Picard approximations}
\label{MLP:Loss}
Let $\Batchsize, \MLPreference \in \N$, 
let 
$  
  \ParamRV_{ 
    m, \BatchVariable 
  }
  \colon \Omega \to \Param
$, $ m, \BatchVariable \in \N $, 
be i.i.d.\ random variables, 
assume that 
$
  (
    \ParamRV_{ 
      m, \BatchVariable 
    }
  )_{
    (m, \BatchVariable) \in \N^2
  }
$,
$
  (
    \MLPRVs^\MLPindex
  )_{\MLPindex \in \MLPindexset}
$,
and 
$(\MLPsingleRV_k)_{k \in \{1, 2, \ldots, \MLPcount_\MLPlevel\} }$
are independent,
for every 
	$m \in \N$
let
$
F_m 
  \colon \allowbreak\R^{  \MLPApproxDim } 
  \times\Omega\to\R
$
satisfy 
for all 
$
\theta 
  \in \R^{\MLPApproxDim} 
$
that
\begin{equation}
\label{T_B_D}
\begin{split} 
	F_m( \theta ) 
&=
	\frac{ 1 }{ \Batchsize }
	\Bigg[
		\sum_{ \BatchVariable = 1 }^{ \Batchsize }
		\left|
			\MLPAlg_{\MLPreference}^{(m, \BatchVariable)}( 
				\ParamRV_{ 
					m, \BatchVariable 
				}
			)
			-
			\LRVANN_{\MLPlevel}(
				\ParamRV_{ 
					m, \BatchVariable 
				},
				\theta
			)
		\right|^2
	\Bigg],
\end{split}
\end{equation}
and for every 
	$m \in \N$
let 
$ 
  G_m 
  \colon \R^{\MLPApproxDim} \times \Omega 
  \to \R^{\MLPApproxDim} 
$
satisfy 
for all 
	$ \theta \in \R^{\MLPApproxDim} $, $ \omega \in \Omega $ 
that
\begin{equation}
  G_m( \theta, \omega )
  =
  ( \nabla_{ \theta } F_m )( \theta, \omega ).
\end{equation}

\subsection{Learning the random variables with stochastic gradient descent}
\label{MLP:SGD}
Let $ (\gamma_m)_{m \in \N} \subseteq (0,\infty) $ and 
let 
$ 
  \Theta \colon \N_0 \times \Omega \to \R^{ \MLPApproxDim} 
$ 
satisfy 
for all $ m \in \N $ that 
$ 
	\Theta_0 
= 
	(\MLPsingleRV_1, \MLPsingleRV_2, \ldots, \MLPsingleRV_{\MLPcount_\MLPlevel})
$
and
\begin{equation}
  \Theta_{ m }
  =
  \Theta_{ m - 1 }
  - 
  \gamma_m 
  G_m( \Theta_{ m - 1 } ) 
  .
\end{equation}
\new{
For every sufficiently large 
$m \in \N$
we propose to employ the random function 
$
	\Param \times \Omega \ni (p, \omega) \mapsto \LRVANN_\MLPlevel(p, \Theta_m(\omega)) \in \R
$
as an approximation for the target function
$
	\Param \ni p \mapsto u(p) \in \R
$ in \eqref{MLP:eq1}.}

\subsection{Description of the proposed  approximation algorithm}
\label{MLP:description}

\begin{algo}
Let 
	$ \dimParam, \dimWRV, \NumMC, \MLPlevel, \MLPreference, \Batchsize  \in \N $, 
	$ (\gamma_m)_{m \in \N} \subseteq (0,\infty) $,
let $ \Param \subseteq \R^{ \dimParam } $ be measurable, 
let 
$ \phi \colon \Param \times \R^\dimWRV \times \R \to \R $ and
$\MLPuEval \colon \Param \times \R^\dimWRV \to \Param$ 
be measurable, 
let $(\MLPcount_n)_{n \in \Z} \subseteq \N_0$ satisfy for all
	$n \in \Z$
that
\begin{equation}
\begin{split} 
	\MLPcount_n = \sum_{k = 0}^{n-1} \NumMC^{n-k}(1 + \MLPcount_k + \MLPcount_{k-1}),
\end{split}
\end{equation}
for every
	$n, \MCvariable \in \N$, $l \in \{1, 2, \ldots, n\}$
let $ \MLPsubcount_{n,l, \MCvariable}\in \N $ satisfy
\begin{equation}
\begin{split} 
	\MLPsubcount_{n,l, \MCvariable} 
= 
	\left[
		\sum_{k = 0}^{l-1} \NumMC^{n-k}(1 + \MLPcount_k +\MLPcount_{k-1})
	\right]
	+
	(\MCvariable-1) (1 + \MLPcount_l + \MLPcount_{l-1})
	+ 
	1,
\end{split}
\end{equation}
let $\LRVANN_n \colon \Param \times \R^{\MLPcount_n \dimWRV}$, $n \in \N_0$, satisfy for all
	$n \in \N$,
	$p \in \Param$,
	$\theta_1,\theta_2,\ldots, \theta_{\MLPcount_n} \in  \R^{\dimWRV}$
that
	$\LRVANN_0(p) = 0$
and
\begin{multline}
	\LRVANN_n(p, (\theta_1,\theta_2,\ldots, \theta_{\MLPcount_n}))
=
	\left[
		\frac{1}{\NumMC^{n}}
		\sum_{\MCvariable = 1}^{\NumMC^{n}}
			\phi\big(
				p, 
				\theta_{\MCvariable},
				0
			\big)
	\right]
\\ 
	+
	\sum_{k = 1}^{n-1}
		\frac{1}{\NumMC^{n-k}}
		\sum_{\MCvariable = 1}^{\NumMC^{n-k}} \bigg[
			\phi\Big(
				p, 
				\theta_{\MLPsubcount_{n,k,\MCvariable}}, 
				\LRVANN_k \big(
					\MLPuEval(p, \theta_{\MLPsubcount_{n,k,\MCvariable}})
					,
					(\theta_{\MLPsubcount_{n,k,\MCvariable} + r})_{r \in \{1, 2, \ldots, \MLPcount_k\}}
			\Big)
\\ 
			-
			\phi\Big(
				p, 
				\theta_{\MLPsubcount_{n,k,\MCvariable}},
				\LRVANN_{k-1}\big(
					\MLPuEval(p, \theta_{\MLPsubcount_{n,k,\MCvariable}})
					,
					(\theta_{\MLPsubcount_{n,k,\MCvariable} + \MLPcount_k + r})_{r \in \{1, 2, \ldots, \MLPcount_{k-1}\}}
				\big)
			\Big)
		\bigg],
\end{multline}
assume for all $p \in \Param$ that
$\R^{ \MLPApproxDim } \ni \theta \mapsto \LRVANN( p, \theta ) \in \R$ is continuously differentiable,  
let $ ( \Omega, \F, \P ) $ be a probability space, 
let $\WRV \colon \Omega \to \R^\dimWRV$ be a random variable,
let $\MLPsingleRV_k \colon \Omega \to \R^{\dimWRV}$,
	$k \in \{1, 2, \ldots, \MLPcount_\MLPlevel\}$,
be i.i.d.\ random variables,
let $\MLPRVs^{ m, \BatchVariable}_k \colon \Omega \to \R^\dimWRV$, $m, \BatchVariable \in \N$, $k \in \{1, 2, \ldots, \MLPcount_\MLPreference\}$, be i.i.d.\ random variables,
assume for all
	$B \in \mathcal{B}(\R^\dimWRV)$
that
$
	\P(\MLPsingleRV_1 \in B)
=
	\P(\WRV \in B)
=
	\P(\MLPRVs^{1,1}_1 \in B)
$,
let 
$  
  \ParamRV_{ 
    m, \BatchVariable 
  }
  \colon \Omega \to \Param
$, $ m, \BatchVariable \in \N $, 
be i.i.d.\ random variables, 
assume that 
$
  (
    \ParamRV_{ 
      m, \BatchVariable 
    }
  )_{
    (m, \BatchVariable) \in \N^2
  }
$, 
$
  (
    \MLPRVs^{ m, \BatchVariable}_k
  )_{(m, \BatchVariable,k) \in \N^2 \times \{1, 2, \ldots, \MLPcount_\MLPreference\}}
$,
and 
$(\MLPsingleRV_k)_{k \in \{1, 2, \ldots, \MLPcount_\MLPlevel\} }$
are independent,
for every 
	$m \in \N$
let
$
F_m 
  \colon \allowbreak\R^{  \MLPApproxDim} 
  \times\Omega\to\R
$
satisfy 
for all 
$
\theta 
  \in \R^{\MLPApproxDim} 
$
that
\begin{equation}
\begin{split} 
	F_m( \theta ) 
&=
	\frac{ 1 }{ \Batchsize }
	\Bigg[
		\sum_{ \BatchVariable = 1 }^{ \Batchsize }
		\left|
			\LRVANN_{\MLPreference}( 
				\ParamRV_{ 
					m, \BatchVariable 
				},
				( \MLPRVs^{ m, \BatchVariable}_1,  \MLPRVs^{ m, \BatchVariable}_2, \ldots,  \MLPRVs^{ m, \BatchVariable}_{\MLPcount_\MLPreference})
			)
			-
			\LRVANN_\MLPlevel(
				\ParamRV_{ 
					m, \BatchVariable 
				},
				\theta
			)
		\right|^2
	\Bigg],
\end{split}
\end{equation}
for every 
	$m \in \N$
let 
$ 
  G_m 
  \colon \R^{\MLPApproxDim} \times \Omega 
  \to \R^{\MLPApproxDim} 
$
satisfy 
for all 
	$ \theta \in \R^{\MLPApproxDim} $, $ \omega \in \Omega $ 
that
$
  G_m( \theta, \omega )
  =
  ( \nabla_{ \theta } F_m )( \theta, \omega )
$,
and
let 
$ 
  \Theta \colon \N_0 \times \Omega \to \R^{ \MLPApproxDim} 
$ 
satisfy 
for all $ m \in \N $ that 
$ 
  \Theta_0 = 
  (\MLPsingleRV_1, \MLPsingleRV_2, \ldots, \MLPsingleRV_{\MLPcount_\MLPlevel})
$
and
$
  \Theta_{ m }
  =
  \Theta_{ m - 1 }
  - 
  \gamma_m 
  G_m( \Theta_{ m - 1 } ) 
$.
\end{algo}

\section{LRV strategy in the case of a general proposal algorithm}
\label{sect:general}

In this section we derive and formulate the LRV strategy in its most general form, which contains all the algorithms derived in the previous sections as special cases.
Roughly speaking, we want to approximate a target function 
	(cf.\ $u \colon \Param \to \R^{\dimSolution}$ in \cref{general:target}) 
for which we already have a generic stochastic approximation algorithm 
	(cf.\ $\ApproxAlg \colon \Param \times \R^{\ApproxAlgdim} \to \R^{\dimSolution}$ and $ \ApproxAlgRV \colon \Omega \to \R^\ApproxAlgdim $ in \eqref{general:eq1} in \cref{general:target}). 
We refer to this algorithm as \emph{proposal algorithm}. 
Moreover, we assume that we are able to generate random reference solutions which approximate the target function at every point in expectation 
	(cf.\ $\Param \times \Omega \ni (p, \omega) \mapsto \ApproxRef(p, \ApproxRefRV(\omega)) \in \R^{\dimSolution}$ in \eqref{general:eq2} in \cref{general:target}).

In the next few sentences we briefly sketch in words the LRV strategy in this general case.
The first step of the LRV strategy is to consider the random variables in the stochastic approximation algorithm as parameters for a parametric family of functions (corresponding to 
	$\Param \ni p \mapsto \ApproxAlg(p, \theta) \in \R^\dimSolution$, 
		$\theta \in \R^\ApproxAlgdim$, 
in \cref{general:target}); 
see \cref{general:Loss}.
The goal of the LRV strategy is then to "learn" parameters whose corresponding function yields a good approximation of the target function 
$u \colon \Param \to \R^{\dimSolution}$.
Taking this into account, the second step of the LRV strategy is to minimize a loss function (cf.\ \eqref{general:eq3} in \cref{general:Loss}) measuring the distance between the approximating function and the approximate reference solutions with an SGD-type optimization method; see \cref{general:SGD}.
As initial guess for the SGD-type learning procedure we suggest to randomly choose the parameters according to the distribution of the random variables appearing in the proposal algorithm, since we know that this already results in a passable approximation. 
This feature of the LRV strategy is an important advantage when compared to standard deep learning methods
in the sense that the LRV strategy has already in the beginning of the training procedure a relatively small loss function.
The entire approach is presented in one single framework in \cref{general:description}.

One of the differences between this section and the previous sections is that in the previous sections we only used, for simplicity, the plain vanilla SGD method, however in this section we allow for various more sophisticated SGD-type optimization methods (cf.\ \eqref{general:eq4} in \cref{general:SGD}).
Some of these more sophisticated SGD-type methods are presented in \cref{general:SGDtype} as special cases of the framework in \cref{general:description}.

\subsection{Stochastic approximations for general target functions related to parametric expectations}
\label{general:target}
Let $ \dimParam, \ApproxAlgdim, \dimSolution, \ApproxRefdim \in \N $,
let $ \Param \subseteq \R^{ \dimParam } $ be measurable, 
let 
	$\ApproxAlg \colon \Param \times \R^{\ApproxAlgdim} \to \R^{\dimSolution}$, 
	$\ApproxRef \colon \Param \times \R^{\ApproxRefdim} \to \R^{\dimSolution}$,
	and
	$ u \colon \Param \to \R^{\dimSolution} $
be measurable,
assume for all $p \in \Param$ that $\R^{ \ApproxAlgdim } \ni w \mapsto \ApproxAlg( p, w ) \in \R^{\dimSolution}$ is continuously differentiable, 
let $ ( \Omega, \F, \P ) $ be a probability space, 
let 
$ \ApproxAlgRV \colon \Omega \to \R^\ApproxAlgdim $ and 
$ \ApproxRefRV \colon \Omega \to \R^\ApproxRefdim $ be independent random variables,
and assume for all $p \in \Param$ that $\E[\|\ApproxRef( p, \ApproxRefRV)\|]<\infty$.
We think of $(\ApproxAlg, \ApproxAlgRV)$ as a
stochastic approximation algorithm for $u \colon \Param \to \R^\dimSolution$ in the sense that for all $p \in \Param$ it holds that
\begin{equation}
\label{general:eq1} 
	u(p)
\approx
	\ApproxAlg( 
		p, \ApproxAlgRV
	)
\end{equation}
and for every $p \in \Param$ we think of $\E[\ApproxRef(p, \ApproxRefRV)]$ as a suitable approximation
\begin{equation}
\label{general:eq2} 
	u(p)
\approx
	\E[
		\ApproxRef( 
			p, \ApproxRefRV
		)
	]
\end{equation}
of $u(p)$.
The goal of this section is to derive an  algorithm to approximately compute the function
$
	u \colon \Param \to \R
$.

\subsection{Random loss functions for fixed random variables in stochastic approximations}
\label{general:Loss}
Let $\Batchsize \in \N$, 
let $\costfct \colon \R^{\dimSolution}\times \R^{\dimSolution} \to [0,\infty)$
continuously differentiable,
let 
$  
  \ParamRV_{ 
    m, \BatchVariable 
  }
  \colon \Omega \to \Param
$, $ m, \BatchVariable \in \N $, 
be i.i.d.\ random variables, 
let 
$ 
  W^{ m, \MCvariable } \colon \Omega \to \R^{ \ApproxRefdim } 
$, 
$ m, \MCvariable \in \N $, 
be i.i.d.\ random variables 
which satisfy for all $ A \in \mathcal{B}( \R^{ \dimBM } ) $ that 
$
  \P( W^{ 0, 0 } \in A ) = \P( \ApproxRefRV \in A ) 
$,
assume that 
$
  (
    \ParamRV_{ 
      m, \BatchVariable 
    }
  )_{
    (m, \BatchVariable) \in \N^2
  }
$
and 
$
	 (
	 	W^{ m, \MCvariable }
	 )_{ (m, \MCvariable) \in \N^2 }
$
are independent,
and for every 
	$m \in \N$
let
$
F_m 
  \colon \allowbreak\R^{ \ApproxAlgdim } 
  \times\Omega\to\R
$
satisfy 
for all 
$
\theta 
  \in \R^{ \ApproxAlgdim} 
$
that
\begin{equation}
\label{general:eq3}
\begin{split} 
	F_m( \theta ) 
&=
	\frac{ 1 }{ \Batchsize }
	\Bigg[
		\sum_{ \BatchVariable = 1 }^{ \Batchsize }
		\costfct\left(
			\ApproxRef \big( 
				\ParamRV_{ 
					m, \BatchVariable 
				},
				W^{m,  \BatchVariable}
			\big)
			,
			\ApproxAlg(
				\ParamRV_{ 
					m, \BatchVariable 
				},
				\theta
			)
		\right)
	\Bigg],
\end{split}
\end{equation}
and for every 
	$m \in \N$
let 
$ 
  G_m 
  \colon \R^{\ApproxAlgdim} \times \Omega 
  \to \R^{\ApproxAlgdim} 
$
satisfy 
for all 
	$ \theta \in \R^{\ApproxAlgdim } $, $ \omega \in \Omega $ 
that
\begin{equation}
  G_m( \theta, \omega )
  =
  ( \nabla_{ \theta } F_m )( \theta, \omega ).
\end{equation}

\subsection{Learning the random variables with stochastic gradient descent-type methods}
\label{general:SGD}
Let 
$
  \psi_m \colon \R^{m \ApproxAlgdim} \to \R^{\ApproxAlgdim}
$, $m \in \N$,
be functions and
let 
$ 
  \Theta \colon \N_0 \times \Omega \to \R^{ \ApproxAlgdim } 
$ 
satisfy 
for all $ m \in \N $ that 
$ 
  \Theta_0 = 
    \ApproxAlgRV
$
and
\begin{equation}
\label{general:eq4}
  \Theta_{ m }
  =
  \Theta_{ m - 1 }
   - 
   \psi_m((G_1(\Theta_0), G_2(\Theta_1), \dots, G_{m}(\Theta_{m-1}))).
\end{equation}
\new{
For every sufficiently large 
$m \in \N$
we propose to employ the random function 
$
	\Param \times \Omega \ni (p, \omega) \mapsto \ApproxAlg(p, \Theta_m(\omega)) \in \R
$
as an approximation for the target function
$
	\Param \ni p \mapsto u(p) \in \R
$.}

\subsection{Description of the proposed approximation algorithm}
\label{general:description}

\begin{algo}\label{algo:general}
Let 
	$ \dimParam, \ApproxAlgdim, \ApproxRefdim, \dimSolution, \Batchsize \in \N $,
let $ \Param \subseteq \R^{ \dimParam } $ be measurable, 
let
$\ApproxAlg \colon \Param \times \R^{\ApproxAlgdim} \to \R^{\dimSolution}$,
$\ApproxRef \colon \Param \times \R^{\ApproxRefdim} \to \R^{\dimSolution}$,
and 
$\costfct \colon \R^{\dimSolution} \times \R^{\dimSolution} \to [0,\infty)$
be functions,
let 
$
  \psi_m = (\psi_m^{(1)}, \ldots, \psi_m^{(\ApproxAlgdim)}) \colon \R^{m \ApproxAlgdim} \to \R^{\ApproxAlgdim}
$, $m \in \N$,
be functions,
let $ ( \Omega, \F, \P ) $ be a probability space, 
let 
$ \ApproxAlgRV \colon \Omega \to \R^\ApproxAlgdim $ be a random variable,
let 
$  
  \ParamRV_{ 
    m, \BatchVariable 
  }
  \colon \Omega \to \Param
$, $ m, \BatchVariable \in \N $, 
be i.i.d.\ random variables, 
let 
$ 
  W^{ m, \MCvariable } \colon \Omega \to \R^{ \ApproxRefdim } 
$, 
$ m, \MCvariable \in \N $, 
be i.i.d.\ random variables,
assume that 
$\ApproxAlgRV$,
$
  (
    \ParamRV_{ 
      m, \BatchVariable 
    }
  )_{
    (m, \BatchVariable) \in \N^2
  }
$, and
$
	 (
	 	W^{ m, \MCvariable }
	 )_{ (m, \MCvariable) \in \N^2 }
$
are independent,
for every 
	$m \in \N$
let
$
F_m 
  \colon \allowbreak\R^{ \ApproxAlgdim } 
  \times\Omega\to\R
$
satisfy 
for all 
$
\theta 
  \in \R^{ \ApproxAlgdim} 
$
that
\begin{equation}
\label{T_B_D}
\begin{split} 
	F_m( \theta ) 
&=
	\frac{ 1 }{ \Batchsize }
	\Bigg[
		\sum_{ \BatchVariable = 1 }^{ \Batchsize }
		\costfct \! \left(
			\ApproxRef \big( 
				\ParamRV_{ 
					m, \BatchVariable 
				},
				W^{m,  \BatchVariable}
			\big)
			,
			\ApproxAlg(
				\ParamRV_{ 
					m, \BatchVariable 
				},
				\theta
			)
		\right)
	\Bigg],
\end{split}
\end{equation}
for every 
	$ m \in \N $ 
let 
$ 
  G_m = (G_m^{(1)},\ldots, G_m^{(\ApproxAlgdim)})
  \colon \R^{\ApproxAlgdim} \times \Omega 
  \to \R^{\ApproxAlgdim} 
$
satisfy 
for all 
	$ \omega \in \Omega $, $ \theta \in \{ v \in \R^{\ApproxAlgdim } \colon F_m( \cdot, \omega)  \text{ is} \linebreak \text{differentiable at } v\} $
that
$
  G_m( \theta, \omega )
  =
  ( \nabla_{ \theta } F_m )( \theta, \omega )
$,
and
let 
$ 
  \Theta = (\Theta^{(1)},\ldots, \Theta^{(\ApproxAlgdim)}) \colon \N_0 \times \Omega \to \R^{ \ApproxAlgdim } 
$ 
be a stochastic process which satisfies  
for all $ m \in \N $ that 
$ 
  \Theta_0 = 
    \ApproxAlgRV
$
and
\begin{equation}
  \Theta_{ m }
  =
  \Theta_{ m - 1 }
   - 
   \psi_m(G_1(\Theta_0), G_2(\Theta_1), \dots, G_{m}(\Theta_{m-1})).
\end{equation}

\end{algo}

\subsection{Explicit descriptions of some popular stochastic gradient descent-type methods}
\label{general:SGDtype}

\subsubsection{Standard stochastic gradient descent (SGD)}
\label{sect:vanillaSGD}

\begin{lemma}
\label{vanillaSGD}
Assume \cref{algo:general},
let $(\gamma_m)_{m \in \N} \subseteq (0,\infty)$,
and assume for all 	
	$m \in \N$,
	$g_1, g_2, \ldots, g_m \in \R^{\ApproxAlgdim}$
that
$
	\psi_m(g_1, g_2, \ldots, g_m) = \gamma_m g_m
$.
Then it holds for all
	$ m \in \N $ 
that
\begin{equation}
  \Theta_{ m }
  =
  \Theta_{ m - 1 }
  - 
  \gamma_m 
  G_m( \Theta_{ m - 1 } ) 
  .
\end{equation}
\end{lemma}

\subsubsection{Stochastic gradient descent with momentum (SGD with momentum)}
\label{sect:Momentum}

\begin{lemma}
\label{Momentum}
Assume \cref{algo:general},
let $(\gamma_m)_{m \in \N} \subseteq (0,\infty)$, $\alpha \in (0,1)$,
and assume for all 	
	$m \in \N$,
	$g_1, g_2, \ldots, g_m \in \R^{\ApproxAlgdim}$
that
$
	\psi_m(g_1, g_2, \ldots, g_m) = \gamma_m \sum_{k = 1}^m \alpha^{m-k}(1-\alpha) g_k
$.
Then 
there exists $\mathbf{m} \colon \N_0 \times \Omega \to \R^d$ such that 
for all 
	$m \in \N$ 
it holds that 
\begin{equation}
	\mathbf{m}_0 = 0,
\qquad
	\mathbf{m}_m = \alpha \mathbf{m}_{m-1} + (1-\alpha) G_m(\Theta_{m-1}), 
\qandq
	\Theta_n = \Theta_{n-1} - \gamma_n  \mathbf{m}_n. 
\end{equation}
\end{lemma}

\subsubsection{Adaptive stochastic gradient descent (Adagrad)}
\label{sect:Adagrad}

\begin{lemma}
\label{Adagrad}
Assume \cref{algo:general},
let 
	$(\gamma_m)_{m \in \N} \subseteq (0,\infty)$, 
	$\varepsilon \in (0,\infty)$, 
and assume for all
	$m \in \N$,
	$i \in \{1, 2,\ldots,\ApproxAlgdim \}$,
	$
		g_1 = (g_1^{(1)},\ldots,g_1^{(\ApproxAlgdim)}), 		g_2 = (g_2^{(1)},\ldots,g_2^{(\ApproxAlgdim)}),
		\ldots, 
		g_m = (g_m^{(1)},\ldots,g_m^{(\ApproxAlgdim)})
	\in 
		\R^{\ApproxAlgdim}
	$
that
\begin{equation}
\begin{split} 
	\psi_m^{(i)}(g_1, g_2, \ldots, g_m)
=
	\left[\frac{\gamma_m}{
		\big(\varepsilon + \sum_{k = 1}^{m} |g_k^{(i)}|^2\big)^{\nicefrac{1}{2}}
	}\right]
	g_m^{(i)}.
\end{split}
\end{equation}
Then
it holds for all $m \in \N$, $i \in \{1,2,\ldots,\ApproxAlgdim\}$ that 
\begin{equation}
	\Theta_m^{(i)}
= 
	\Theta_{m-1}^{(i)} - 
	\left[
	\frac{\gamma_m}{\big(\varepsilon +
	\sum_{k = 1}^m |G_k^{(i)}(\Theta_{k-1})|^2
	 \big)^{\nicefrac{1}{2}}} 
	 \right]G_m^{(i)}(\Theta_{m-1}).
\end{equation}
\end{lemma}

\subsubsection{Root mean square error propagation stochastic gradient descent (RMSprop)}
\label{sect:RMSprop}

\begin{lemma}
\label{RMSprop}
Assume \cref{algo:general},
let 
	$(\gamma_m)_{m \in \N} \subseteq (0,\infty)$, 
	$\varepsilon \in (0,\infty)$, 
	$\beta \in (0,1)$,
and assume for all
	$m \in \N$,
	$i \in \{1, 2,\ldots,\ApproxAlgdim \}$,
	$
		g_1 = (g_1^{(1)},\ldots,g_1^{(\ApproxAlgdim)}), 		g_2 = (g_2^{(1)},\ldots,g_2^{(\ApproxAlgdim)}),
		\ldots, 
		g_m = (g_m^{(1)},\ldots,g_m^{(\ApproxAlgdim)})
	\in 
		\R^{\ApproxAlgdim}
	$
that
\begin{equation}
\begin{split} 
	\psi_m^{(i)}(g_1, g_2, \ldots, g_m)
=
	\left[\frac{\gamma_m}{
		\big(\varepsilon + \sum_{k = 1}^{m} \beta^{m-k}(1-\beta)|g_k^{(i)}|^2\big)^{\nicefrac{1}{2}}
	}\right]
	g_m^{(i)}.
\end{split}
\end{equation}
Then
there exists  
$\mathbb{M} = (\mathbb{M}^{(1)},\ldots,\mathbb{M}^{(\ApproxAlgdim)}) \colon \N_0\times \Omega \to \R^\ApproxAlgdim$ 
such that for all $m \in \N$, $i \in \{1,2,\ldots,\ApproxAlgdim\}$ it holds that 
\begin{equation}
	\mathbb{M}_0 = 0, 
\qquad
	\mathbb{M}_m^{(i)} 
= 
	\beta\mathbb{M}_{m-1}^{(i)} + (1-\beta)|G_m^{(i)}(\Theta_{m-1})|^2,
\end{equation}
\begin{equation}
\andq
	\Theta_m^{(i)}
= 
	\Theta_{m-1}^{(i)} 
	- 
	\left[ 
		\frac{\gamma_m}{(\varepsilon +\mathbb{M}_m^{(i)} )^{\nicefrac{1}{2}}} 
	\right] 
	G_m^{(i)}(\Theta_{m-1}).
\end{equation}
\end{lemma}

\subsubsection{Adadelta stochastic gradient descent (Adadelta)}
\label{sect:Adadelta}

\begin{lemma}
\label{Adadelta}
Assume \cref{algo:general},
let 
	$\varepsilon \in (0,\infty)$, 
	$\beta, \delta \in (0,1)$,
and assume for all
	$m \in \N$,
	$i \in \{1, 2,\ldots,\ApproxAlgdim \}$,
	$
		g_1 = (g_1^{(1)},\ldots,g_1^{(\ApproxAlgdim)}), 		g_2 = (g_2^{(1)},\ldots,g_2^{(\ApproxAlgdim)}),
		\ldots, 
		g_m = (g_m^{(1)},\ldots,g_m^{(\ApproxAlgdim)})
	\in 
		\R^{\ApproxAlgdim}
	$
that
\begin{equation}
\begin{split} 
	\psi_m^{(i)}(g_1, g_2, \ldots, g_m)
=
	\left[
		\frac{
			\varepsilon + \sum_{k = 1}^{m-1} \delta^{m-1-k}(1-\delta) |\psi_k^{(i)}(g_1, g_2, \ldots, g_k)|^2
		}{
			\varepsilon + \sum_{k = 1}^{m} \beta^{m-k}(1-\beta) |g_k^{(i)}|^2
		}
	\right]^{\nicefrac{1}{2}}
	g_m^{(i)}.
\end{split}
\end{equation}
Then
there exist 
	$\mathbb{M} = (\mathbb{M}^{(1)},\ldots,\mathbb{M}^{(d)})$
	and
	$\Delta = (\Delta^{(1)}, \ldots, \Delta^{(d)}) \colon \N_0\times \Omega \to \R^d$ 
	such that for all 
	$m \in \N$, 
	$i \in \{1,2,\ldots,d\}$ 
it holds that 
\begin{equation}
\mathbb{M}_0 = 0, \qquad \Delta_0 = 0, 
\qquad
	\mathbb{M}_m^{(i)}
= 
	\beta \,\mathbb{M}_{m-1}^{(i)} + (1-\beta)|G_m^{(i)}(\Theta_{m-1})|^2,
\end{equation}
\begin{equation}
	\Theta_m^{(i)} 
= 
	\Theta_{m-1}^{(i)} - \bigg[\frac{\varepsilon +\Delta_{m-1}^{(i)} }{\varepsilon +\mathbb{M}_m^{(i)} }\bigg]^{\nicefrac{1}{2}} 
	G_m^{(i)}(\Theta_{m-1}), 
\qandq 
	\Delta_m^{(i)} = \delta\Delta_{m-1}^{(i)} + (1-\delta)|\Theta_m^{(i)}-\Theta_{m-1}^{(i)}|^2. 
\end{equation}

\end{lemma}

\subsubsection{Adamax stochastic gradient descent (Adamax)}
\label{sect:Adamax}

\begin{lemma}
Assume \cref{algo:general},
let 
	$\varepsilon \in (0,\infty)$, 
	$\alpha,\beta \in (0,1)$,
and assume for all
	$m \in \N$,
	$i \in \{1, 2,\ldots,\ApproxAlgdim \}$,
	$
		g_1 = (g_1^{(1)},\ldots,g_1^{(\ApproxAlgdim)}), 		g_2 = (g_2^{(1)},\ldots,g_2^{(\ApproxAlgdim)}),
		\ldots, 
		g_m = (g_m^{(1)},\ldots,g_m^{(\ApproxAlgdim)})
	\in 
		\R^{\ApproxAlgdim}
	$
that
\begin{equation}
\begin{split} 
	\psi_m^{(i)}(g_1, g_2, \ldots, g_m)
=
	\gamma_m
	\left[
		\frac{
			\sum_{k = 1}^m \alpha^{m-k}(1-\alpha) g_k^{(i)}
		}{
			1-\alpha^m
		}
	\right]
	\left[
		\varepsilon +
		\max\big\{|g_m^{(i)}|, \beta |g_{m-1}^{(i)}|, \ldots, \beta^{m-1} |g_1^{(i)}| \big\}
	\right]^{-1}. 
\end{split}
\end{equation}
Then
there exist 
	$\mathbf{m}  = (\mathbf{m}^{(1)},\ldots,\mathbf{m}^{(d)})$ and 
	$\mathbb{M} = (\mathbb{M}^{(1)}, \ldots, \mathbb{M}^{(d)})  \colon \N_0\times \Omega \to \R^d$ 
such that for all 
	$m \in \N$, 
	$i \in \{1,2,\ldots,d\}$ 
it holds that 
\begin{equation}
	\mathbf{m}_0 = 0, 
\qquad
	\mathbf{m}_m
= 
	\alpha\mathbf{m}_{m-1} + (1-\alpha) G_m(\Theta_{m-1}),
\end{equation}
\begin{equation} 
	\mathbb{M}_0 = 0, 
\qquad
	\mathbb{M}_m^{(i)} = \max\{ \beta\,\mathbb{M}_{m-1}^{(i)}, |G_m^{(i)}(\Theta_{m-1})|^2\},
\end{equation}
\begin{equation}
\andq 
	\Theta_m^{(i)}
= 
	\Theta_{m-1}^{(i)} 
	-  
	\gamma_m
	\left[
		\tfrac{\mathbf{m}_m^{(i)}}{1-\alpha^m}
	\right]
	\left[
		\varepsilon + \mathbb{M}_m^{(i)}
	\right]^{-1}
	.
\end{equation}
\end{lemma}

\subsubsection{Adaptive moment estimation stochastic gradient desent (Adam)}
\label{sect:Adam}

\begin{lemma}
Assume \cref{algo:general},
let 
	$\varepsilon \in (0,\infty)$, 
	$\alpha,\beta \in (0,1)$,
and assume for all
	$m \in \N$,
	$i \in \{1, 2,\ldots,\ApproxAlgdim \}$,
	$
		g_1 = (g_1^{(1)},\ldots,g_1^{(\ApproxAlgdim)}), 		g_2 = (g_2^{(1)},\ldots,g_2^{(\ApproxAlgdim)}),
		\ldots, 
		g_m = (g_m^{(1)},\ldots,g_m^{(\ApproxAlgdim)})
	\in 
		\R^{\ApproxAlgdim}
	$
that
\begin{equation}
\begin{split} 
	\psi_m^{(i)}(g_1, g_2, \ldots, g_m)
=
	\gamma_m
	\left[
		\frac{
			\sum_{k = 1}^m \alpha^{m-k}(1-\alpha) g_k^{(i)}
		}{
			1-\alpha^m
		}
	\right]
	\left[
		\varepsilon +\bigg[
			\frac{
				\sum_{k = 1}^m \beta^{m-k}(1-\beta) |g_k^{(i)}|^2
			}{
				1- \beta^m
			}
		\bigg]^{\nicefrac{1}{2}} 
	\right]^{-1}. 
\end{split}
\end{equation}
Then
there exist 
	$\mathbf{m}  = (\mathbf{m}^{(1)},\ldots,\mathbf{m}^{(\ApproxAlgdim)})\colon \N_0\times \Omega \to \R^\ApproxAlgdim$ and 
	$\mathbb{M} = (\mathbb{M}^{(1)}, \ldots, \mathbb{M}^{(\ApproxAlgdim)}) \colon \N_0\times \Omega \to \R^\ApproxAlgdim$ 
such that for all 
	$m \in \N$, 
	$i \in \{1,2,\ldots,\ApproxAlgdim\}$ 
it holds that 
\begin{equation}
	\mathbf{m}_0 = 0, 
\qquad
	\mathbf{m}_m
= 
	\alpha\mathbf{m}_{m-1} + (1-\alpha) G_m(\Theta_{m-1}),
\end{equation}
\begin{equation} 
	\mathbb{M}_0 = 0, 
\qquad
	\mathbb{M}_m^{(i)} = \beta\,\mathbb{M}_{m-1}^{(i)} + (1-\beta)  |G_m^{(i)}(\Theta_{m-1})|^2,
\end{equation}
\begin{equation}
\andq 
	\Theta_m^{(i)}
= 
	\Theta_{m-1}^{(i)} 
	-  
	\gamma_m
	\left[
		\tfrac{\mathbf{m}_m^{(i)}}{1-\alpha^m}
	\right]
	\left[
		\varepsilon +\big( \tfrac{\mathbb{M}_m^{(i)}}{1- \beta^m}\big)^{\!\!\nicefrac{1}{2}} 
	\right]^{-1}
	.
\end{equation}
\end{lemma}

\section{Numerical examples}
\label{sect:numerics}

\newcommand{\unclear}[1]{{\color{blue}#1}}

In this section we apply the LRV strategy to different parametric approximation problems from the literature.
Specifically,
	we consider the classical parametric Black-Scholes model for the pricing of European call options in \cref{simul:BS1},
	we consider a parametric Black-Scholes model for the pricing of worst-of basket put options on three underlying assets in \cref{simul:BSworstPut},
	we consider a parametric Black-Scholes model for the pricing of average basket put options on three underlying assets with knock-in barriers in \cref{simul:BSaverageBarrier}, and
	we consider a parametric stochastic Lorentz equation in \cref{simul:Lorentz}.
In the literature there are already a number of simulation results for SGD-based deep learning methods regarding the high-dimensional pricing of financial derivative contracts.
In particular 
	we refer to, e.g., 
	\cite{berner2020numerically,BeckJafaari21,Lokeshwar22,Ferguson2018,Germain2021,Biagini2021} 
	for parametric pricing results for European options,	
	we refer to, e.g.,  \cite{Andersson21,Salvador2020a,Becker2018,becker2019pricingPublished,Lind2022,Chen21,Lapeyre21,Gaspar20,Ye2019,Becker2019published} 
	for the pricing of American options, and 
	we refer, e.g., to \cite{ruf2020neural} for further references.

In \cref{sect:anti} we briefly recall the antithetic MC method (cf., e.g., Glasserman \cite{Glasserman03}) and some well-known properties of it. 
This is a variance reduction technique for MC methods which we will employ in some of the proposal algorithms for the LRV strategy in case of some of the above mentioned approximation problems.
In each of the considered numerical examples we also compare the LRV strategy with existing approximation techniques from the literature such as
	the deep learning method induced by Becker et al.\ \cite{BeckJafaari21}, 
	MC methods, and
	QMC methods.

All the simulations in this section were run in {\sc Python} using {\sc TensorFlow} 2.12 on remote machines from \texttt{https://vast.ai} equipped with a single
\textsc{NVIDIA GeForce RTX 4090} GPU with 24 GB Graphics RAM.
The {\sc Python} source codes which were employed to produce all the results in this section can be downloaded as part of the sources of the arXiv version of this article at \url{https://arxiv.org/e-print/2202.02717}.
Specifically 
the codes in the folder {\it 1\_BS1} were employed to produce all the results in \cref{simul:BS1},
the codes in the folder {\it 2\_BS\_eur\_put\_basket} were employed to produce all the results in \cref{simul:BSworstPut},
the codes in the folder {\it 3\_BS\_barrier\_put\_basket\_avg} were employed to produce all the results in \cref{simul:BSaverageBarrier}, and
the codes in the folder {\it 4\_Lorentz} were employed to produce all the results in \cref{simul:Lorentz}.

\subsection{Antithetic Monte Carlo approximations}
\label{sect:anti}

In this section we recall a special case of antithetic variates for MC methods (cf., e.g., Glasserman \cite[Section 4.2]{Glasserman03}) when the distribution of the sampled random variables is symmetric around the origin.
The following result, \cref{lem:anti1} below, shows that the resulting antithetic MC method achieves a higher or equal $L^2$-accuracy than the standard MC method when the same number of MC samples are used for both methods.
The subsequent result, \cref{lem:anti2} below, then provides a sufficient condition for the antithetic MC method to achieve a \emph{strictly} higher $L^2$-accuracy than the standard MC method when the same number of MC samples are used for both methods.

\begin{lemma}
\label{lem:anti1}
Let $\dimProblem, \NumMC \in \N$, 
let $ ( \Omega, \mathcal{F}, \P ) $ be a probability space, 
let $X_m \colon \Omega \to \R^\dimProblem$, $m \in \{1, 2, \ldots, \NumMC\}$, be i.i.d.\ random variables,
assume for all 
	$B \in \Borel(\R^\dimProblem)$ 
that $\P(X_1 \in B) = \P(-X_1 \in B)$,
let $f \colon \R^\dimProblem \to \R$ be measurable,  
assume
$\Exp{|f(X_1)|^2} < \infty$,
and let 
	$M \colon \Omega \to \R$
and
	$A \colon \Omega \to \R$
satisfy 
\begin{equation}
\label{anti1:ass1}
\begin{split} 
	M
=
	\left[ \frac{1}{\NumMC} \sum_{m = 1}^{\NumMC} f(X_m) \right]
\qandq
	A
=
	\frac{1}{2\NumMC} \left[ \sum_{m = 1}^{\NumMC} \bigl( f(X_m) + f(-X_m) \bigr) \right].
\end{split}
\end{equation}
Then
\begin{enumerate}[label=(\roman *)]
\item \label{anti1:item1}
it holds that 
$
	\Exp{|M-\Exp{f(X_1)}|^2}
=
	\frac{\Var{f(X_1)}}{\NumMC}
$
and

\item \label{anti1:item2}
it holds that 
$
	\Exp{|A-\Exp{f(X_1)}|^2}
=
	\frac{\Var{f(X_1)} + \Cov{f(X_1)}{f(-X_1)}}{2\NumMC}
\leq
	\frac{\Var{f(X_1)}}{\NumMC}
$.
\end{enumerate}
\end{lemma}

\begin{proof}[Proof of \cref{lem:anti1}]
First, note that the fact that $X_m$, $m \in \{1, 2, \ldots, \NumMC\}$, are i.i.d.\ implies that for all
	$m,n \in \{1, 2, \ldots, \NumMC\}$
with $m \neq n$ it holds that
\begin{equation}
\label{anti1:eq1}
\begin{split} 
	\Cov{f(X_m)}{f(X_n)} = 0
\qandq
	\Cov{f(X_m) + f(-X_m)}{f(X_n) + f(-X_n)} = 0
\end{split}
\end{equation}
(cf., e.g., Klenke \cite[Theorem 5.4]{Klenke14}).
This, 
the fact that $\Exp{M} = \Exp{f(X_1)}$, and 
the Bienaym\'e formula (cf, e.g., Klenke \cite[Theorem 5.7]{Klenke14}) assure that
\begin{equation}
\label{anti1:eq2}
\begin{split} 
	\Exp{|M-\Exp{f(X_1)}|^2}
&=
	\Var{M}
=
	\frac{1}{\NumMC^2}
	\Var{ {\textstyle \sum_{m = 1}^{\NumMC} f(X_m)} } \\
&=
	\frac{1}{\NumMC^2}
	{\textstyle \left(
		\left[ 
			\sum_{m = 1}^{\NumMC} \Var{f(X_m)} 
		\right]
		+
			\sum_{m,n \in \{1, 2, \ldots, \NumMC\}, m \neq n}
				\Cov{f(X_m)}{f(X_n)}
	\right)}\\
&=
	\frac{1}{\NumMC^2} \,
	{\textstyle
			\sum_{m = 1}^{\NumMC} \Var{f(X_1)} 
	}
=
	\frac{\Var{f(X_1)}}{\NumMC}.
\end{split}
\end{equation}
This proves \cref{anti1:item1}. 
Next observe that the fact that $X_1,X_2, \ldots, X_\NumMC, -X_1, -X_2, \ldots, -X_\NumMC$ are identically distributed implies that 
\begin{equation}
\label{anti1:eq3}
\begin{split} 
	\Exp{A}
=
	\frac{1}{2\NumMC} \sum_{m = 1}^{\NumMC} \left(\Exp{f(X_m)} + \Exp{f(-X_m)}\right)
=
	\frac{1}{2\NumMC} \sum_{m = 1}^{\NumMC} 2\Exp{f(X_1)}
=
	\Exp{f(X_1)}.
\end{split}
\end{equation}
This,
\eqref{anti1:eq1},
the fact that $X_1,X_2, \ldots, X_\NumMC, -X_1, -X_2, \ldots, -X_\NumMC$ are identically distributed, and
the Bienaym\'e formula (cf, e.g., Klenke \cite[Theorem 5.7]{Klenke14})
assure
that 
\begin{equation}
\label{anti1:eq4}
\begin{split} 
	\Exp{|A-\Exp{f(X_1)}|^2}
&=
	\Var{A}
=
	\frac{1}{(2\NumMC)^2}
	\Var{ {\textstyle \sum_{m = 1}^{\NumMC} f(X_m) + f(-X_m)} } \\
&=
	\frac{1}{4\NumMC^2}
	{\textstyle \bigg(
		\left[ 
			\sum_{m = 1}^{\NumMC} \Var{f(X_m) + f(-X_m)} 
		\right]}\\
&\quad
		+
		{\textstyle
		\left[	
			\sum_{m,n \in \{1, 2, \ldots, \NumMC\}, m \neq n}
				\Cov{f(X_m) + f(-X_m)}{f(X_n) + f(-X_n)}
		\right]
	\bigg)}\\
&=
	\frac{1}{4\NumMC^2} \,
	{\textstyle 			\sum_{m = 1}^{\NumMC} \Var{f(X_1) + f(-X_1)} 
	} \\
&=
	\frac{\Var{f(X_1) + f(-X_1)} }{4\NumMC} \\
&=
	\frac{\Var{f(X_1)} + 2 \Cov{f(X_1)}{f(-X_1)} + \Var{f(-X_1)}}{4\NumMC} \\
&=
	\frac{\Var{f(X_1)} + \Cov{f(X_1)}{f(-X_1)}}{2\NumMC}.
\end{split}
\end{equation}
Moreover, note that 
the fact that $X_1$ and $-X_1$ are identically distributed and 
the Cauchy-Schwarz inequality (cf., e.g., Klenke \cite[Theorem 5.8]{Klenke14}) 
assure that
\begin{equation}
\label{anti1:eq5}
\begin{split} 
	(\Cov{f(X_1)}{f(-X_1)})^2
\leq
	\Var{f(X_1)}\Var{f(-X_1)}
=
	(\Var{f(X_1)})^2.
\end{split}
\end{equation}
Combining this with \eqref{anti1:eq4} demonstrates that 
\begin{equation}
\label{anti1:eq6}
\begin{split} 
	\Exp{|A-\Exp{f(X_1)}|^2}
&=
	\frac{\Var{f(X_1)} + \Cov{f(X_1)}{f(-X_1)}}{2\NumMC}\\
&\leq
	\frac{\Var{f(X_1)} + \Var{f(X_1)}}{2\NumMC}
=
	\frac{\Var{f(X_1)}}{\NumMC}.
\end{split}
\end{equation}
This establishes \cref{anti1:item2}.
The proof of \cref{lem:anti1} is thus complete.
\end{proof}

\begin{corollary}
\label{lem:anti2}
Let $\dimProblem, \NumMC \in \N$, 
let $ ( \Omega, \mathcal{F}, \P ) $ be a probability space, 
let $X_m \colon \Omega \to \R^\dimProblem$, $m \in \{1, 2, \ldots, \NumMC\}$, be i.i.d.\ random variables,
assume for all 
	$B \in \Borel(\R^\dimProblem)$ 
that $\P(X_1 \in B) = \P(-X_1 \in B)$,
let $f \colon \R^\dimProblem \to \R$ be measurable,  
assume
for all 
	$a,b,c \in \R$ 
that $\Exp{|f(X_1)|^2} < \infty$
and
$
	\P(af(X_1) + b f(-X_1) + c = 0) < 1
$,
and let 
	$M \colon \Omega \to \R$
and
	$A \colon \Omega \to \R$
satisfy 
\begin{equation}
\label{anti2:ass1}
\begin{split} 
	M
=
	\left[ \frac{1}{\NumMC} \sum_{m = 1}^{\NumMC} f(X_m) \right]
\qandq
	A
=
	\frac{1}{2\NumMC} \left[ \sum_{m = 1}^{\NumMC} \bigl( f(X_m) + f(-X_m) \bigr) \right].
\end{split}
\end{equation}
Then
\begin{enumerate}[label=(\roman *)]
\item \label{anti2:item1}
it holds that 
$
	\Exp{|M-\Exp{f(X_1)}|^2}
=
	\frac{\Var{f(X_1)}}{\NumMC}
$
and

\item \label{anti2:item2}
it holds that 
$
	\Exp{|A-\Exp{f(X_1)}|^2}
=
	\frac{\Var{f(X_1)} + \Cov{f(X_1)}{f(-X_1)}}{2\NumMC}
<
	\frac{\Var{f(X_1)}}{\NumMC}
$.
\end{enumerate}
\end{corollary}

\begin{proof}[Proof of \cref{lem:anti2}]
First, note \cref{anti1:item1} in \cref{lem:anti1} establishes \cref{anti2:item1}.
Moreover, observe that 
the assumption that 
for all 
	$a,b,c \in \R$ 
it holds that
$
	\P(af(X_1) + b f(-X_1) + c = 0) < 1
$,
the fact that $X_1$ and $-X_1$ are identically distributed, and 
the strict Cauchy-Schwarz inequality (cf., e.g., Klenke \cite[Theorem 5.8]{Klenke14}) 
assure that
\begin{equation}
\label{anti2:eq5}
\begin{split} 
	(\Cov{f(X_1)}{f(-X_1)})^2
<
	\Var{f(X_1)}\Var{f(-X_1)}
=
	(\Var{f(X_1)})^2.
\end{split}
\end{equation}
Combining this with \cref{anti1:item2} in \cref{lem:anti1} demonstrates that 
\begin{equation}
\label{anti2:eq6}
\begin{split} 
	\Exp{|A-\Exp{f(X_1)}|^2}
&=
	\frac{\Var{f(X_1)} + \Cov{f(X_1)}{f(-X_1)}}{2\NumMC}\\
&<
	\frac{\Var{f(X_1)} + \Var{f(X_1)}}{2\NumMC}
=
	\frac{\Var{f(X_1)}}{\NumMC}.
\end{split}
\end{equation}
This establishes \cref{anti2:item2}.
The proof of \cref{lem:anti2} is thus complete.
\end{proof}

\subsection[Parametric Black-Scholes equations for European call options]{Parametric Black-Scholes partial differential equations for European call options}
\label{simul:BS1}

\newcommand{\numMCsamplesBS}{8\,192\,000}
\newcommand{\nrtrainstepsBS}{140000}
\newcommand{\numRefAlgSamplesBS}{??}

In this section we apply the LRV strategy to the problem of approximating the fair price of an European call option in the classical Black-Scholes model (cf.\ Black \& Scholes \cite{BlackScholes73} and Merton \cite{Merton73}).
A brief summary of the numerical results of this subsection can be found in \cref{table:BS1_Summary} in the introduction.
We start by introducing the Black-Scholes model in the context of \cref{algo:general}.

Assume \cref{algo:general},
assume
	$\dimParam = 5$
and
\begin{equation}
\begin{split} 
	\Param = {[90, 110] \times [0.01,1] \times [-0.1,0.1] \times [0.01,0.5] \times  [90, 110]},
\end{split}
\end{equation}
let $\normalcdf \colon \R\to \R$ satisfy for all $z\in \R$ that
$
  \normalcdf(z) =\tfrac{1}{\sqrt{2\pi}} \int_{-\infty}^z \exp({-\frac{y^2}{2}}) \d y
$,
and let $u \colon \Param \to \R$ satisfy for all
	$p = (\BSinitprice, T, \BSrate, \sigma, K) \in \Param$
that
\begin{equation}
\label{BS1:eq0}
\begin{split} 
	u(p)
=
	\BSinitprice \,
	\normalcdf \Big(
		\tfrac{ (\BSrate + \frac{ \sigma^2 }{2}) T + \ln( {\BSinitprice}/{K} )}{ \sigma \sqrt{T} }
	\Big)
	-
	\exp({-\BSrate T})K
	\,
	\normalcdf \Bigl(
		\tfrac{ (\BSrate - \frac{ \sigma^2 }{2}) T + \ln( {\BSinitprice}/{K} )}{ \sigma \sqrt{T} }
	\Bigr).
\end{split}
\end{equation}
Note that \cref{BS1:eq0} corresponds to the famous Black-Scholes formula for European call options.
In the economic interpretation of the Black-Scholes model, for every 
	$p = (\BSinitprice, T, \BSrate, \sigma, K) \in \Param$
the number 
	$u(p) \in \R$
thus corresponds to the fair price of a European call option with
	initial price $\BSinitprice$,	
	time of maturity $T$,
	drift rate $\BSrate$,
	volatility $\sigma$,
	and strike price $K$.

We now specify the mathematical objects in the LRV strategy appearing in \cref{algo:general} 
to approximately calculate the target function $u \colon \Param \to \R$ in \cref{BS1:eq0}.
Specifically, in addition to the assumptions above,
let 
	$\NumMC, \NumRefMC \in \N$,
	$\antithetic, \exact \in \{0,1\}$,
assume
	$\ApproxAlgdim =  \NumMC$,
	$\ApproxRefdim =  \NumRefMC$,
	$\dimSolution = 1$,
for every 
	$\BSinitprice, \BSrate, \sigma, w \in \R$,
	$T \in [0,\infty)$
let $X^{\BSrate, \sigma, \BSinitprice, w}_T \in \R$
satisfy
\begin{equation}
\begin{split} 
	X^{\BSrate, \sigma, \BSinitprice, w}_T
=
	\BSinitprice \exp \! \big(
		(\BSrate - \tfrac{\sigma^2}{2})T + \sigma \sqrt{T} w
	\big),
\end{split}
\end{equation}
let 
	$ \phi_\mathfrak{a} \colon \Param \times \R^{ \dimWRV } \to \R $, 
		$\mathfrak{a} \in \{0, 1\}$,
satisfy for all
	$p = (\BSinitprice, T, \BSrate, \sigma, K) \in \Param$,
	$w \in \R$
that
\begin{equation}
\label{BS1:eq0.1}
\begin{split} 
	\phi_0(p, w) 
=
	\exp({-\BSrate T})\max \{ X^{\BSrate, \sigma, \BSinitprice, w}_T - K,0\}
\qand
\end{split}
\end{equation}
\begin{equation}
\label{BS1:eq0.2}
\begin{split} 
\textstyle
	\phi_1(p, w) 
=
	\frac{\phi_0(p, w) + \phi_0(p, -w)}{2}
=
	\frac{\exp({-\BSrate T})}{2}(\max \{ X^{\BSrate, \sigma, \BSinitprice, w}_T - K,0\} + \max \{ X^{\BSrate, \sigma, \BSinitprice, -w}_T - K,0\}),
\end{split}
\end{equation}
assume for all 
	$p = (\BSinitprice, T, \BSrate, \sigma, K) \in \Param$,
	$w = (w_1, w_2, \ldots, w_{\NumRefMC}) \in \R^{\NumRefMC}$,
	$\theta = (\theta_1, \ldots, \theta_{\NumMC}) \in \R^{\NumMC}$
that
\begin{equation}
\label{BS1:eq1}
\begin{split} 
	\ApproxAlg(p, \theta)
=
	\frac{ 1 }{ \NumMC }
	{\textstyle
	\left[ 
		\sum\limits_{ \MCvariable = 1 }^{ \NumMC }
			\phi_a( 
				p, 
				\theta_{ \MCvariable}         
			)
	\right]}
\qandq
	\ApproxRef(p, w)
=
	\frac{ \mathbbm{1}_{\{0\}}(e)  }{ \NumRefMC }
	{\textstyle
	\left[ 
		\sum\limits_{ \MCvariable = 1 }^{ \NumRefMC }
			\phi_1( 
				p, 
				w_{ \MCvariable}         
			)
	\right]}
+
	\mathbbm{1}_{\{1\}}(e) 
	u(p),
\end{split}
\end{equation}
assume for all
	$x, y \in \R$
that
	$\costfct(x, y) = | x-y |^2$,
let
	$\mathbf{a} = {\frac{9}{10}}$,
	$\mathbf{b}  = {\frac{999}{1000}}$,
	$ \varepsilon \in (0,\infty)$,
let
	$(\gamma_m)_{m \in \N} \subseteq (0,\infty)$
satisfy for all 
	$j \in \{1, 2, \ldots, 7\}$,
	$m \in \N \cap (20000 (j-1),  20000 j]$
that
$
	{\gamma_m = 10^{-j}}
$,
assume for all
	$m \in \N$,
	$i \in \{1, 2,\ldots,\ApproxAlgdim \}$,
	$g_1 = (g_1^{(1)},\ldots,g_1^{(\ApproxAlgdim)})$, 	$g_2 = (g_2^{(1)},\ldots,g_2^{(\ApproxAlgdim)})$,
	$\dots$, 
	$g_m = (g_m^{(1)},\ldots,g_m^{(\ApproxAlgdim)})
	\in 
		\R^{\ApproxAlgdim}
	$
that
\begin{equation}
\label{BS:eq2}
\begin{split} 
	\psi_m^{(i)}(g_1, g_2, \ldots, g_m)
=
	\gamma_m
	\left[
		\frac{
			\sum_{k = 1}^m \mathbf{a}^{m-k}(1-\mathbf{a}) g_k^{(i)}
		}{
			1-\mathbf{a}^m
		}
	\right]
	\left[
		\varepsilon +\bigg[
			\frac{
				\sum_{k = 1}^m \mathbf{b}^{m-k}(1-\mathbf{b}) |g_k^{(i)}|^2
			}{
				1- \mathbf{b}^m
			}
		\bigg]^{\nicefrac{1}{2}} 
	\right]^{-1},
\end{split}
\end{equation}	
assume that  
$ \ApproxAlgRV$ is a standard normal random vector,
assume that  
$  
  \ParamRV_{ 
    1, 1 
  }
$ is $\mathcal{U}_{\Param}$-distributed, and
assume that 
$ 
  W^{ 1, 1 }
$
is a standard normal random variable.

Let us add some comments regarding the setup introduced above.
Observe that 
	\eqref{BS1:eq0}, 
	\eqref{BS1:eq0.1}, 
	\eqref{BS1:eq0.2}, and 
	\cref{cor:black_scholes} 
assure that for all 
	$p = (\BSinitprice, T, \BSrate, \sigma, K) \in \Param$,
it holds that
\begin{equation}
\begin{split} 
	u(p)
&=
	\EXPP{
		\exp({-\BSrate T})\max \{ X^{\BSrate, \sigma, \BSinitprice, \WRVs^{1, 1}}_T - K,0\}
	}
=
	\Exp{
		\phi_0(p, \WRVs^{1, 1})
	}
=
	\Exp{
		\phi_1(p, \WRVs^{1, 1})
	}.
\end{split}
\end{equation}
Moreover, note that in the case $\antithetic = 0$ the proposal algorithm on the left hand side of \eqref{BS1:eq1} corresponds to the standard MC method with $\NumMC$ samples and
that in the case $\antithetic = 1$ the proposal algorithm on the left hand side of \eqref{BS1:eq1} corresponds to the antithetic MC method with $\NumMC$ samples (cf.\ \cref{sect:anti}).
In addition, observe that in the case $\exact = 0$ the reference solutions on the right hand side of \eqref{BS1:eq1} correspond to antithetic MC approximations with $\NumRefMC$ samples and that in the case $\exact = 1$ the reference solutions on the right hand side of \eqref{BS1:eq1} are given by the exact solution.
Furthermore, observe that \cref{BS:eq2} describes the Adam optimizer in the setup of \cref{algo:general} (cf.\ Kingma \& Ba \cite{Kingma2014} and \cref{sect:Adam}).

In \cref{fig:BS_plot_1,fig:BS_plot_2,table:BS_LRV_exact,table:BS_LRV} we approximately present for different choices of
	$\NumMC \in \{2^5, 2^6, \ldots, 2^{13}\}$,
	$\NumRefMC \in \{2^{10}, 2^{11}, 2^{12}\}$,
	$\Batchsize \in  \{2^{11}, 2^{12}, 2^{13}\}$,
	$\antithetic, \exact \in \{0,1\}$
random realizations of the $L^1(\lambda_{\Param};\R)$-approximation error
\begin{equation}
\label{BS_lrv:eq1}
\begin{split} 
\textstyle
	\int_\Param 
		|u(p) - \ApproxAlg(p, \Theta_{\nrtrainstepsBS})|
	\d p
\end{split}
\end{equation}
(\nth{3} column in \cref{table:BS_LRV,table:BS_LRV_exact}),
random realizations of the $L^2(\lambda_{\Param};\R)$-approximation error
\begin{equation}
\label{BS_lrv:eq2}
\begin{split} 
\textstyle
	\bigl[
		\int_\Param 
			|u(p) - \ApproxAlg(p, \Theta_{\nrtrainstepsBS})|^2
		\d p
	\bigr]^{\nicefrac{1}{2}}
\end{split}
\end{equation}
(\cref{fig:BS_plot_1,fig:BS_plot_2} and \nth{4} column in \cref{table:BS_LRV,table:BS_LRV_exact}),
random realizations of the $L^\infty(\lambda_{\Param};\R)$-approximation error
\begin{equation}
\label{BS_lrv:eq3}
\begin{split} 
\textstyle
	\sup_{p \in \Param}
		|u(p) - \ApproxAlg(p, \Theta_{\nrtrainstepsBS})|,
\end{split}
\end{equation}
(\nth{5} column in \cref{table:BS_LRV,table:BS_LRV_exact}),
the time to compute $\Theta_{\nrtrainstepsBS}$
(\nth{6} column in \cref{table:BS_LRV,table:BS_LRV_exact}),
and the time to compute 
	$\numMCsamplesBS$ 
evaluations of the function
$
	\Param \ni p \mapsto \ApproxAlg(p, \Theta_{\nrtrainstepsBS}) \in \R
$
(\nth{7} column in \cref{table:BS_LRV,table:BS_LRV_exact}).
We approximated the integrals in \eqref{BS_lrv:eq1} and \eqref{BS_lrv:eq2} with the MC method based on $\numMCsamplesBS$ samples
and we approximated the suprema in \eqref{BS_lrv:eq3} based on $\numMCsamplesBS$ random samples 
(cf., e.g., Beck et al.\ \cite[Lemma 3.5]{BeckJafaari21} and Beck et al.\ \cite[Section 3.3]{Beck2019published}).

To compare the LRV strategy with existing approximation techniques from the literature, we also employ several other methods to approximate the function $u \colon \Param \to \R$ in \eqref{BS1:eq0}.
Specifically, 
	in \cref{table:BS_ANNs} we present numerical simulations for the deep learning method induced by Becker et al.\ \cite{BeckJafaari21}
	(with training values given by the exact solution,
	Adam $140000$ training steps, 
	batch size $8192$, 
	learning rate schedule 
	$
		\N \ni j 
		\mapsto 
			{\mathbbm{1}_{(0,20000]}(m)}{10^{-2}}
			+
			{\mathbbm{1}_{(20000,50000]}(m)}{10^{-3}}
			+
			{\mathbbm{1}_{(50000,80000]}(m)}{10^{-4}}
			+
			{\mathbbm{1}_{(80000,100000]}(m)}{10^{-5}}
			+
			{\mathbbm{1}_{(100000,120000]}(m)}{10^{-6}}
			+
			{\mathbbm{1}_{(120000,140000]}(m)}{10^{-7}}
	$, and
	GELU activation function),
	in \cref{table:BS_MC} we present numerical simulations for the standard and the antithetic MC method, and
	in \cref{table:BS_QMC} we present numerical simulations for the standard and the antithetic QMC method with Sobol sequences.
In \cref{table:BS_ANNs,table:BS_MC,table:BS_QMC} 
	we have approximated the $L^1(\lambda_{\Param};\R)$-approximation errors 
		of the respective approximation methods with the MC method based on $\numMCsamplesBS$ samples,
	we have approximated the $L^2(\lambda_{\Param};\R)$-approximation errors
		of the respective approximation methods with the MC method based on $\numMCsamplesBS$ samples, and 
	we have approximated the $L^\infty(\lambda_{\Param};\R)$-approximation errors
		of the respective approximation methods based on $\numMCsamplesBS$ random samples
		(cf., e.g., Beck et al.\ \cite[Lemma 3.5]{BeckJafaari21} and Beck et al.\ \cite[Section 3.3]{Beck2019published}).
		
Next we discuss the empirical distributions of the random variables learned by the LRV methodology and compare them to empirical distributions of MC and QMC samples. 
Specifically, in \cref{fig:Histograms_LRV} we visualize 
for $\NumMC \in \{ 2^9, 2^{10}, 2^{11}\}$, $\antithetic = 0$
realizations of 
empirical distributions of \emph{learned random variables} 
	$\Theta^{(1)}_{140000}, \Theta^{(2)}_{140000}, \ldots, \Theta^{(\NumMC)}_{140000}$
in the case of three different kinds of training procedures
and in \cref{fig:Histograms_MC} we visualize 
realizations of 
empirical distributions of random variables in the MC method and in QMC method based on Sobol sequences. Very roughly speaking, it seems that with increasingly precise reference solutions and thereby smaller $L^2$-errors, the LRV method produces learned random variables whose histograms approximate the density of the normal distribution more closely, in particular more closely than the realizations of the standard MC samples with which the LRV method is initialized.
This suggests that the LRV strategy is learning random variables which in some sense try to approximate the normal distribution.
On the other hand, we note that the QMC samples seem to have the most regular histograms, but still have much worse $L^2$-errors when compared to the learned random variables. 
One explanation for this could be that even though the histograms of the QMC method seem to approximate the normal density very accurately, the empirical moments of the QMC samples are a worse approximation of the moments of the normal distribution than the empirical moments of the learned random variables, and so the QMC samples effectively do not approximate the normal distribution as well as the learned random variables.

\begin{figure}[H]
\includegraphics[width=\linewidth]{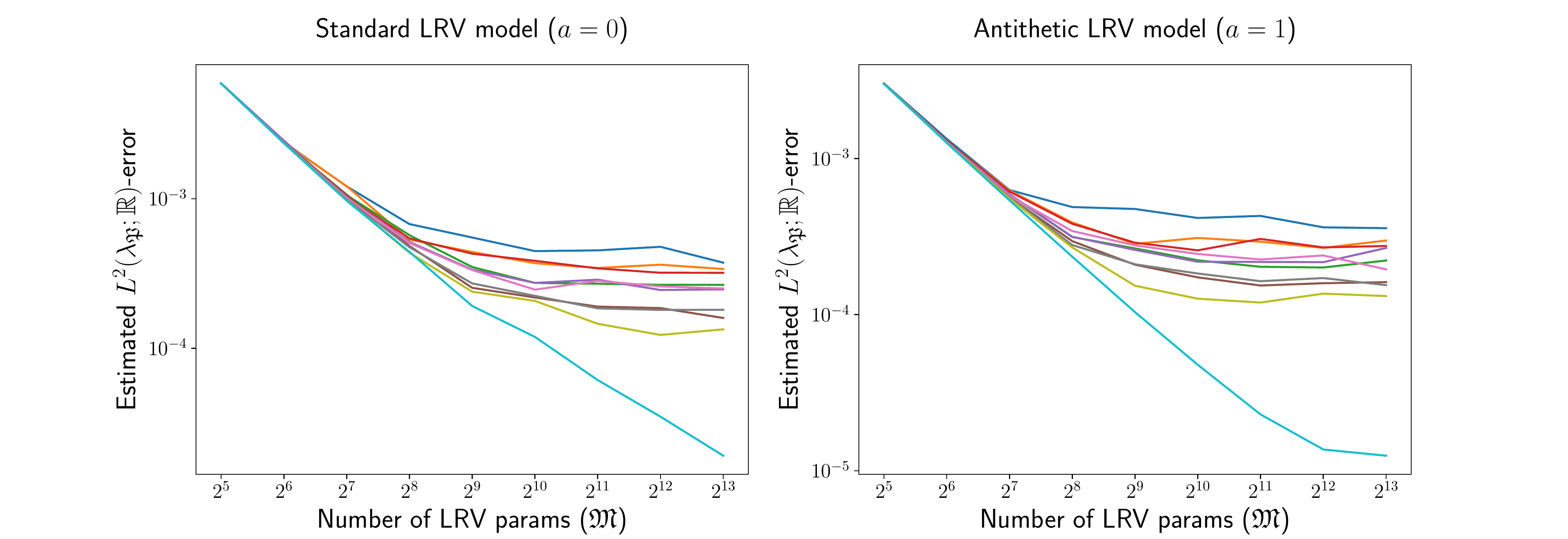}
\caption{\label{fig:BS_plot_1}
Numerical simulations for the LRV strategy in case of the Black-Scholes model for European call options on one underlying described in \cref{simul:BS1} (5-dimensional approximation problem).
See \cref{fig:BS_plot_2} below for the legend.
}
\end{figure}

\begin{figure}[H]
\includegraphics[width=\linewidth]{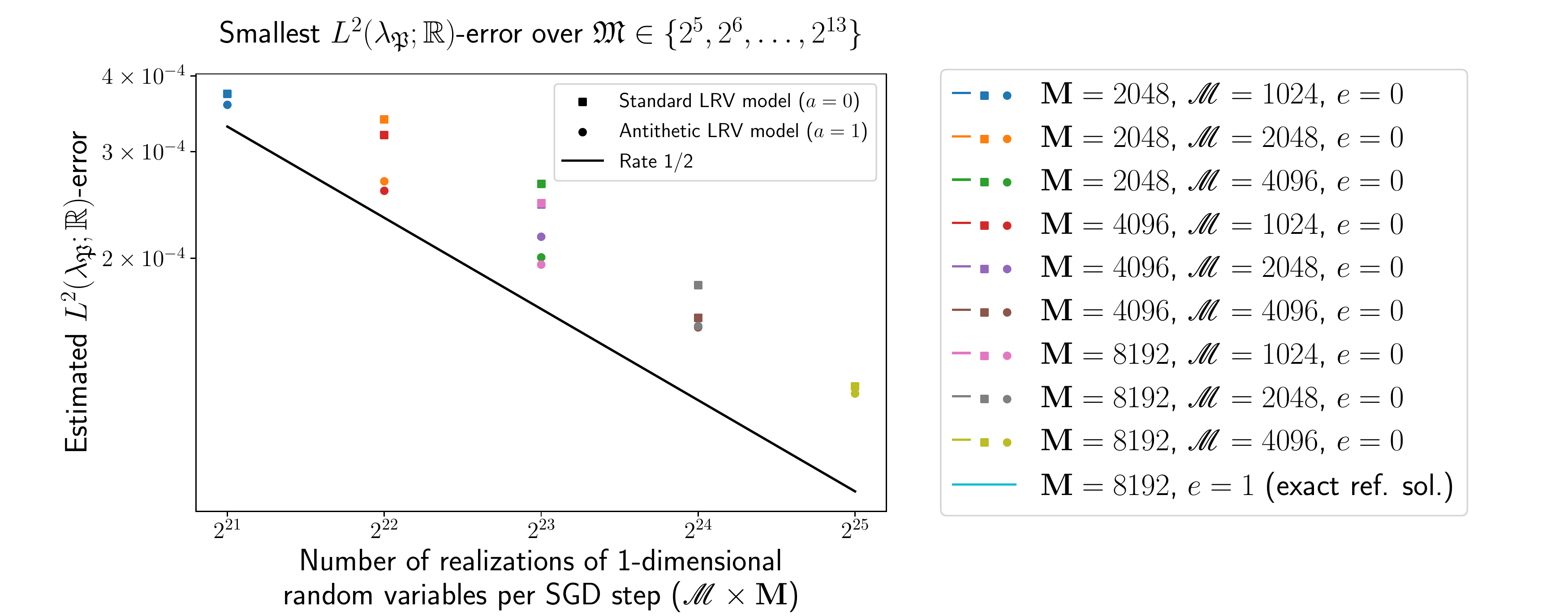}
\caption{\label{fig:BS_plot_2}Smallest estimated $L^2(\lambda_{\Param};\R)$-errors over $\NumMC$ for different choices of training parameters. Each dot corresponds to the minimum of a line of the same color in \cref{fig:BS_plot_1}.
The convergence rate 1/2 is inspired by the convergence rate 1/2 of the MC method which is strongly related to SGD.}
\end{figure}

\begin{table}[H] \tiny
\resizebox{\textwidth}{!}{
\begin{filecontents*}{Table_2.tex}
num_lrv_variables, mc_samples, antithetic, batch_size, l1_error, l_2_error, l_inf_error, train_time, eval_time, dtype, num_samples_mc_training, training_antithetic
32,8192000,0,8192,0.003183,0.005893,0.068957,728.00,0.60, float32,4096,1
32,8192000,1,8192,0.001504,0.003016,0.029396,731.45,0.62, float32,4096,1
64,8192000,0,8192,0.001200,0.002350,0.024142,726.46,0.62, float32,4096,1
64,8192000,1,8192,0.000639,0.001268,0.012222,735.63,0.68, float32,4096,1
128,8192000,0,8192,0.000494,0.000977,0.009113,739.05,0.57, float32,4096,1
128,8192000,1,8192,0.000281,0.000555,0.005386,749.49,0.62, float32,4096,1
256,8192000,0,8192,0.000228,0.000436,0.004850,742.17,0.61, float32,4096,1
256,8192000,1,8192,0.000147,0.000270,0.002682,774.72,0.60, float32,4096,1
512,8192000,0,8192,0.000122,0.000239,0.006127,768.41,0.59, float32,4096,1
512,8192000,1,8192,0.000085,0.000153,0.001942,818.77,0.65, float32,4096,1
1024,8192000,0,8192,0.000099,0.000208,0.007090,811.28,0.66, float32,4096,1
1024,8192000,1,8192,0.000068,0.000127,0.002126,950.48,0.89, float32,4096,1
2048,8192000,0,8192,0.000080,0.000146,0.002312,944.56,0.87, float32,4096,1
2048,8192000,1,8192,0.000060,0.000120,0.001818,1241.45,2.32, float32,4096,1
4096,8192000,0,8192,0.000061,0.000123,0.002020,1237.02,2.32, float32,4096,1
4096,8192000,1,8192,0.000068,0.000136,0.001901,1782.10,4.20, float32,4096,1
8192,8192000,0,8192,0.000065,0.000134,0.002138,1773.54,4.17, float32,4096,1
8192,8192000,1,8192,0.000064,0.000132,0.002210,2799.30,7.73, float32,4096,1
\end{filecontents*}
\csvreader[tabular=|c|c|c|c|c|c|c|c|c|,
     table head=
     \hline $\NumMC$ &  $a$ & \thead{Number \\of \\trainable \\parameters}
     & \thead{$L^1$-approx.\\ error}& \thead{$L^2$-approx. \\error}&\thead{ $L^\infty$-approx. \\ error} & \thead{Training \\ time \\ in seconds} &\thead{Evaluation \\ time  for \\ $\numMCsamplesBS$\\ evaluations \\ in seconds} \\\hline,
    late after line=\\\hline]{Table_2.tex}
{
	num_lrv_variables=\nsamp, 
	antithetic=\anti, 
	l1_error=\lll, 
	l_2_error=\llll, 
	l_inf_error=\linf, 
	train_time=\train, 
	eval_time=\eval  ,
	dtype = \dtype
}
{\nsamp& \anti& \nsamp&  \lll & \llll & \linf &\train&\eval }}
\caption{\label{table:BS_LRV}Numerical simulations for the LRV strategy in case of the Black-Scholes model for European call options on one underlying described in \cref{simul:BS1} (5-dimensional approximation problem) trained with 
batch size $\Batchsize = 8192$ 
and
approximate reference solutions ($\exact = 0$) 
based on $\NumRefMC = 4096$ antithetic MC samples (cf.\ yellow dots and lines in \cref{fig:BS_plot_1,fig:BS_plot_2})}
\end{table}

\begin{table}[H] \tiny
\resizebox{\textwidth}{!}{
\begin{filecontents*}{Table_3.tex}
num_lrv_variables, mc_samples, antithetic, batch_size, l1_error, l_2_error, l_inf_error, train_time, eval_time, dtype
32,8192000,0,8192,0.003186,0.005893,0.065741,149.13,0.58, float32
32,8192000,1,8192,0.001502,0.003013,0.028984,153.95,0.62, float32
64,8192000,0,8192,0.001211,0.002352,0.023104,148.42,0.64, float32
64,8192000,1,8192,0.000635,0.001262,0.012718,156.98,0.61, float32
128,8192000,0,8192,0.000482,0.000971,0.009106,152.34,0.60, float32
128,8192000,1,8192,0.000273,0.000541,0.005257,152.41,0.60, float32
256,8192000,0,8192,0.000206,0.000441,0.007746,154.84,0.61, float32
256,8192000,1,8192,0.000118,0.000235,0.002228,170.88,0.68, float32
512,8192000,0,8192,0.000089,0.000192,0.003382,163.79,0.66, float32
512,8192000,1,8192,0.000052,0.000104,0.000969,179.32,0.67, float32
1024,8192000,0,8192,0.000049,0.000119,0.003899,184.64,0.71, float32
1024,8192000,1,8192,0.000024,0.000048,0.000420,305.44,0.91, float32
2048,8192000,0,8192,0.000027,0.000061,0.002447,304.91,0.87, float32
2048,8192000,1,8192,0.000012,0.000023,0.000206,656.63,2.36, float32
4096,8192000,0,8192,0.000015,0.000035,0.001529,652.68,2.34, float32
4096,8192000,1,8192,0.000008,0.000014,0.000122,1199.47,4.20, float32
8192,8192000,0,8192,0.000010,0.000019,0.000862,1187.81,4.18, float32
8192,8192000,1,8192,0.000007,0.000012,0.000145,2201.75,7.84, float32
\end{filecontents*}
\csvreader[tabular=|c|c|c|c|c|c|c|c|c|,
     table head=
     \hline $\NumMC$ &  $a$ & \thead{Number \\of \\trainable \\parameters}
     & \thead{$L^1$-approx.\\ error}& \thead{$L^2$-approx. \\error}&\thead{ $L^\infty$-approx. \\ error} & \thead{Training \\ time \\ in seconds} &\thead{Evaluation \\ time  for \\ $\numMCsamplesBS$\\ evaluations \\ in seconds} \\\hline,
    late after line=\\\hline]{Table_3.tex}
{
	num_lrv_variables=\nsamp, 
	antithetic=\anti, 
	l1_error=\lll, 
	l_2_error=\llll, 
	l_inf_error=\linf, 
	train_time=\train, 
	eval_time=\eval,
	dtype = \dtype
}
{\nsamp& \anti& \nsamp&  \lll & \llll & \linf &\train&\eval }}
\caption{\label{table:BS_LRV_exact}Numerical simulations for the LRV strategy in case of the Black-Scholes model for European call options on one underlying described in \cref{simul:BS1} (5-dimensional approximation problem) trained 
with batch size $\Batchsize = 8192$ and exact ($\exact = 1$) reference solutions (cf.\ bright blue lines in \cref{fig:BS_plot_1,fig:BS_plot_2})}
\end{table}

\begin{table}[H] \tiny
\resizebox{\textwidth}{!}{
\begin{filecontents*}{Table_4.tex}
layers, inner_dim, num_params, mc_samples, l1_error, l_2_error, l_inf_error, train_time, eval_time, dtype
2,8,129,8192000,0.068139,0.099111,1.308262,182.59,0.91, float64
2,16,385,8192000,0.019551,0.029058,0.626703,184.90,0.81, float64
2,32,1281,8192000,0.006874,0.010994,0.294981,182.46,0.73, float64
2,64,4609,8192000,0.003855,0.006479,0.266781,181.95,0.78, float64
2,128,17409,8192000,0.002519,0.004411,0.197840,191.05,0.74, float64
3,8,201,8192000,0.043577,0.064298,1.055705,210.32,0.96, float64
3,16,657,8192000,0.011687,0.017968,0.505980,201.72,0.77, float64
3,32,2337,8192000,0.003797,0.006030,0.262204,205.58,0.76, float64
3,64,8769,8192000,0.001761,0.002791,0.124099,196.73,0.81, float64
3,128,33921,8192000,0.001244,0.001813,0.097105,199.35,0.84, float64
\end{filecontents*}
\csvreader[tabular=|c|c|c|c|c|c|c|c|c|c|,
     table head=
     \hline \thead{Number \\ of \\ hidden \\ layers}& \thead{Number of \\ neurons \\ on  each \\ hidden \\ layer}& \thead{Number \\ of \\ trainable \\ parameters} & \thead{$L^1$-approx. \\ error}& \thead{$L^2$-approx. \\error}&\thead{ $L^\infty$-approx. \\ error} & \thead{Training \\ time \\ in \\ seconds} &\thead{Evaluation \\ time  for \\ $\numMCsamplesBS$\\ evaluations \\ in seconds} \\\hline,
    late after line=\\\hline]{Table_4.tex}
{
	layers = \lyrs,
	inner_dim = \inner,
	num_params = \weights,
	mc_samples = \mcsamp, 
	l1_error=\lll, 
	l_2_error=\llll, 
	l_inf_error=\linf, 
	train_time=\train, 
	eval_time=\eval  ,
	dtype = \dtype
}
{\pgfmathparse{int(\lyrs )}\pgfmathresult & \inner &\weights & \lll & \llll & \linf &\train & \eval }}
\caption{\label{table:BS_ANNs}Numerical simulations for the deep learning method induced by Becker et al.\ \cite{BeckJafaari21} in case of the Black-Scholes model European call options on one underlying described in \cref{simul:BS1} (5-dimensional approximation problem)}
\end{table}

\begin{table}[H] \tiny
\resizebox{\textwidth}{!}{
\begin{filecontents*}{Table_5.tex}
num_samples, mc_samples, antithetic, l1_error, l_2_error, l_inf_error, time, dtype
32,8192000,0,1.631470,2.620134,41.858883,1.69, float32
32,8192000,1,0.857496,1.443121,21.981831,0.65, float32
64,8192000,0,1.154554,1.851235,28.730389,0.62, float32
64,8192000,1,0.607091,1.019940,13.166348,0.60, float32
128,8192000,0,0.816507,1.308833,17.123699,0.59, float32
128,8192000,1,0.429601,0.721680,8.736799,0.66, float32
256,8192000,0,0.578006,0.925793,12.759777,0.63, float32
256,8192000,1,0.303619,0.510076,6.837311,0.66, float32
512,8192000,0,0.408652,0.654518,8.377937,0.63, float32
512,8192000,1,0.214879,0.360810,5.057564,0.67, float32
1024,8192000,0,0.289187,0.463213,6.491411,0.63, float32
1024,8192000,1,0.151856,0.254823,3.114460,1.09, float32
2048,8192000,0,0.204313,0.327113,3.708378,0.94, float32
2048,8192000,1,0.107378,0.180250,2.206034,2.57, float32
4096,8192000,0,0.144438,0.231285,2.808947,2.56, float32
4096,8192000,1,0.075998,0.127575,1.839296,4.87, float32
8192,8192000,0,0.102109,0.163498,1.972134,4.67, float32
8192,8192000,1,0.053718,0.090149,1.181227,9.20, float32
16384,8192000,0,0.072263,0.115714,1.562056,8.89, float32
16384,8192000,1,0.037983,0.063745,0.771879,17.78, float32
32768,8192000,0,0.051127,0.081914,1.079885,17.12, float32
32768,8192000,1,0.026839,0.045043,0.593046,35.07, float32
\end{filecontents*}
\csvreader[tabular=|c|c|c|c|c|c|c|c|,
     table head=
     \hline \thead{Number \\ of \\MC \\samples}& \thead{MC Method \\ 0: standard \\ 1: antithetic} & \thead{$L^1$-approx. \\ error}& \thead{$L^2$-approx. \\error}&\thead{ $L^\infty$-approx. \\ error} &\thead{Evaluation \\ time  for \\ $\numMCsamplesBS$\\ evaluations \\ in seconds} \\\hline,
    late after line=\\\hline]{Table_5.tex}
{
	num_samples=\nsamp, 
	mc_samples = \mcsamp, 
	antithetic=\anti, 
	l1_error=\lll, 
	l_2_error=\llll, 
	l_inf_error=\linf, 
	time=\train, 
	dtype = \dtype
}
{\nsamp&\anti& \lll & \llll & \linf &\train }}
\caption{\label{table:BS_MC}Numerical simulations for the standard and the antithetic MC method in case of the Black-Scholes model for European call options on one underlying described in \cref{simul:BS1} (5-dimensional approximation problem)}
\end{table}

\begin{table}[H] \tiny
\resizebox{\textwidth}{!}{
\begin{filecontents*}{Table_6.tex}
num_samples, mc_samples, antithetic, l1_error, l_2_error, l_inf_error, time, dtype
32,8192000,0,0.857061,1.117799,4.344944,0.85, float32
32,8192000,1,0.237763,0.366264,1.820515,0.64, float32
64,8192000,0,0.493400,0.642453,2.530792,0.64, float32
64,8192000,1,0.122386,0.188951,0.944592,0.64, float32
128,8192000,0,0.276729,0.360582,1.431679,0.69, float32
128,8192000,1,0.062681,0.097161,0.494709,0.64, float32
256,8192000,0,0.152859,0.199531,0.797966,0.66, float32
256,8192000,1,0.032036,0.049889,0.257656,0.75, float32
512,8192000,0,0.083264,0.109023,0.448029,0.71, float32
512,8192000,1,0.016332,0.025576,0.135283,0.78, float32
1024,8192000,0,0.044996,0.059125,0.249985,0.75, float32
1024,8192000,1,0.008340,0.013135,0.070560,1.08, float32
2048,8192000,0,0.024177,0.031904,0.137550,0.93, float32
2048,8192000,1,0.004245,0.006730,0.037239,2.55, float32
4096,8192000,0,0.012906,0.017105,0.074234,2.35, float32
4096,8192000,1,0.002163,0.003453,0.019505,4.70, float32
8192,8192000,0,0.006863,0.009143,0.040550,4.23, float32
8192,8192000,1,0.001101,0.001769,0.010147,8.65, float32
16384,8192000,0,0.003636,0.004868,0.022209,7.70, float32
16384,8192000,1,0.000561,0.000908,0.005302,16.68, float32
32768,8192000,0,0.001920,0.002585,0.011879,15.48, float32
32768,8192000,1,0.000286,0.000466,0.002769,33.65, float32
\end{filecontents*}
\csvreader[tabular=|c|c|c|c|c|c|c|c|,
     table head=
     \hline 
     	\thead{Number \\of \\MC \\samples}& \thead{QMC Method \\ 0: standard \\ 1: antithetic} & \thead{$L^1$-approx. \\ error}& \thead{$L^2$-approx. \\error}&\thead{ $L^\infty$-approx. \\ error} &\thead{Evaluation \\ time for \\  $\numMCsamplesBS$\\ evaluations \\ in seconds} 
     \\\hline,
    late after line=\\\hline]{Table_6.tex}
{
	num_samples=\nsamp, 
	mc_samples = \mcsamp, 
	antithetic=\anti, 
	l1_error=\lll, 
	l_2_error=\llll, 
	l_inf_error=\linf, 
	time=\train, 
	dtype = \dtype
}
{\nsamp&\anti& \lll & \llll & \linf &\train }}
\caption{\label{table:BS_QMC}Numerical simulations for the standard and the antithetic QMC method with Sobol sequences in case of the Black-Scholes model for European call options on one underlying described in \cref{simul:BS1} (5-dimensional approximation problem)}
\end{table}

\begin{figure}[H]
\includegraphics[width=\linewidth]{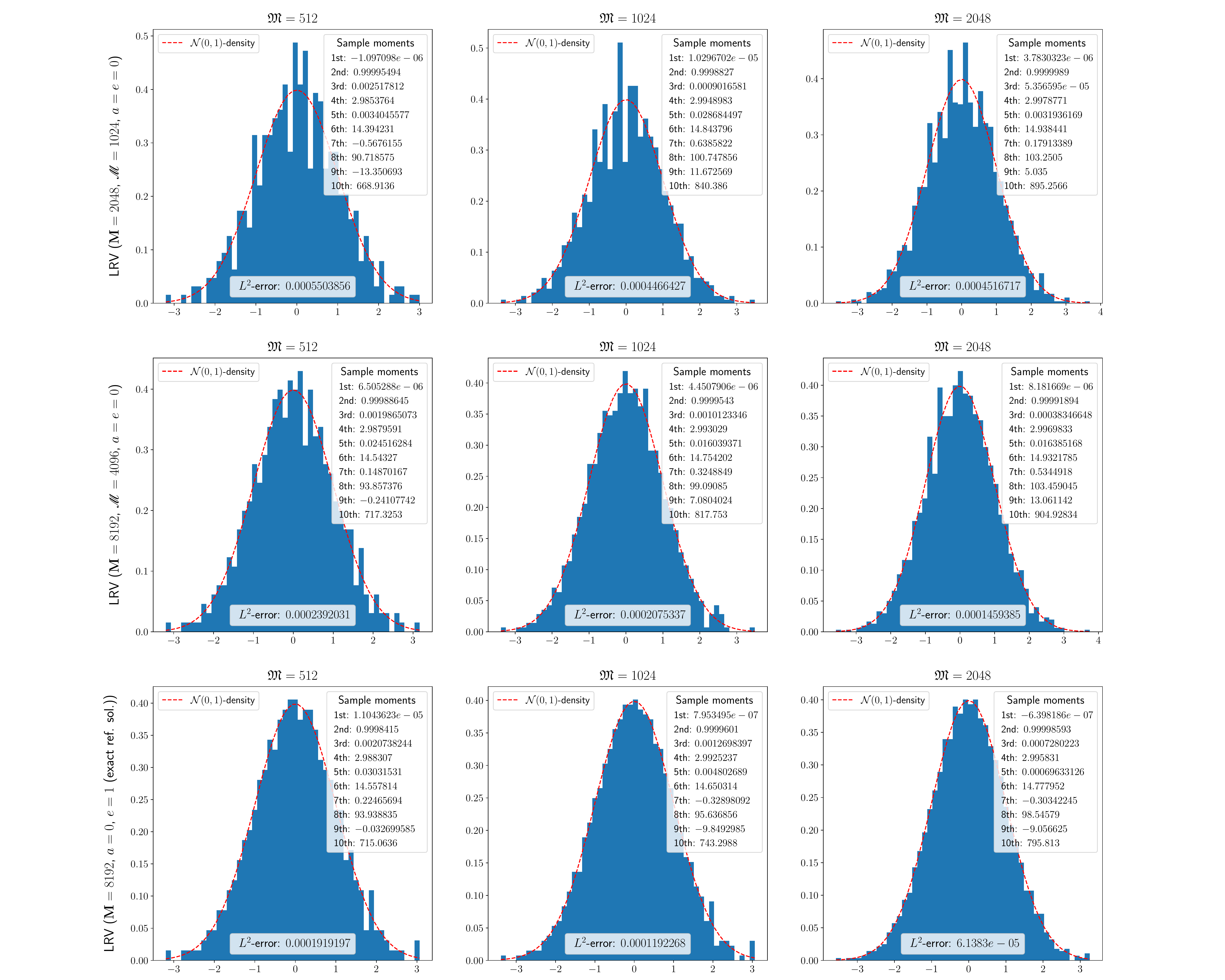}
\caption{\label{fig:Histograms_LRV}
Histograms and sample moments of realizations of \emph{learned random variables} 
	$\Theta^{(1)}_{140000}, \allowbreak \Theta^{(2)}_{140000}, \ldots, \Theta^{(\NumMC)}_{140000}$
in the LRV strategy in case of the Black-Scholes model for European call options on one underlying described in \cref{simul:BS1}.
See \cref{fig:Histograms_MC} for the moments of the standard normal distribution.
}
\end{figure}

\begin{figure}[H]
\includegraphics[width=\linewidth]{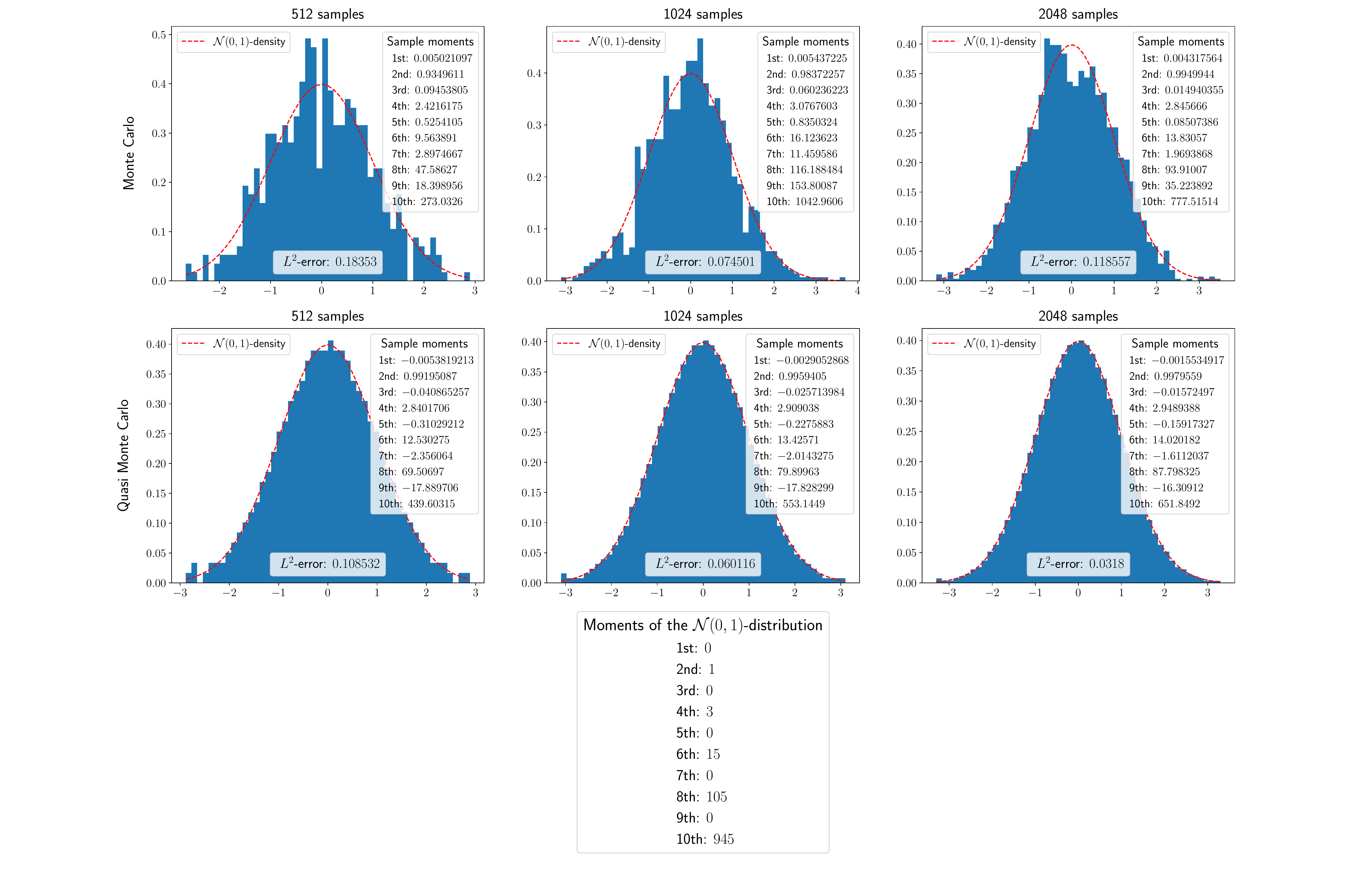}
\caption{\label{fig:Histograms_MC}
Histograms and sample moments of realizations of random variables appearing in the MC method and QMC method in case of the Black-Scholes model for European call options on one underlying described in \cref{simul:BS1}
}
\end{figure}

\subsection[Parametric Black-Scholes equations for multi-asset worst-of basket put options]{Parametric Black-Scholes partial differential equations for multi-asset worst-of basket put options}
\label{simul:BSworstPut}

\newcommand{\numMCsamplesBSworstPut}{128000}
\newcommand{\numRefSamplesBSworstPut}{209715200}
\newcommand{\numRefAlgSamplesBSworstPut}{8192}
\newcommand{\numTrainstepsBSworstPut}{30000}

\newcommand{\covMatrix}{Q}
\newcommand{\Choleski}{L}
\newcommand{\CorrBM}{\mathcal{B}}
\newcommand{\Minfunction}{\mathbf{l}}

In this section we apply the LRV strategy for the approximation of the price of European worst-of basket put options in the Black-Scholes option pricing model (cf.\ Black \& Scholes \cite{BlackScholes73} and Merton \cite{Merton73}).
We start by introducing the Black-Scholes model for this pricing problem in the context of \cref{algo:general}.

Assume \cref{algo:general},
let 
	$\dimProblem = 3$, 
	$e_1 = (1, 0, 0)$, 
	$e_2 = (0, 1, 0)$, 
	$e_3 = (0, 0, 1)$,
let $\covMatrix \colon \R^\dimProblem \to \R^{\dimProblem \times \dimProblem}$ satisfy for all 
	$\rho = (\rho_1, \rho_2, \rho_3) \in \R^\dimProblem$
that
\begin{equation}
\begin{split} 
	\covMatrix(\rho) 
= 
	\begin{pmatrix}
		1 &\rho_1 &\rho_2 \\
		\rho_1 &1 &\rho_3 \\
		\rho_2 &\rho_3 &1
	\end{pmatrix},
\end{split}
\end{equation}
let 
$
	R
\subseteq 
	\R^\dimProblem
$
satisfy
$
	R 
= 
	\bigl\{
		\rho = (\rho_1, \rho_2, \rho_3) \in [-\frac{95}{100},\frac{95}{100}]^\dimProblem 
		\colon  
		1 - |\rho_2|^2 - \frac{(\rho_3 - \rho_1\rho_2)^2}{(1-|\rho_1|^2)^{1/2}} \geq 0
	\bigr\}
$,
assume
	$\dimParam = 15$
and
\begin{equation}
\begin{split} 
	\Param 
= 
	[90, 110]^3 \times
	[\tfrac{1}{100}, 1] \times
	[-\tfrac{1}{20}, \tfrac{1}{20}] \times
	[0, \tfrac{1}{10}]^3 \times
	[\tfrac{1}{100}, \tfrac{1}{2}]^3 \times
	R \times 
	[90, 110],
\end{split}
\end{equation}
let $\Choleski \colon R \to \R^{\dimProblem \times \dimProblem}$ satisfy for all 
	$\rho = (\rho_1, \rho_2, \rho_3) \in R$
that
\begin{equation}
\begin{split} 
	\Choleski(\rho)
=
	\begin{pmatrix}
		1 & 0 & 0 \\
		\rho_1& (1 - |\rho_1|^2)^{{1}/{2}} &0 \\
		\rho_2& \frac{\rho_3 - \rho_1\rho_2}{(1 - |\rho_1|^2)^{{1}/{2}}} & \big(1 - |\rho_2|^2 - \frac{(\rho_3 - \rho_1\rho_2)^2}{1 - |\rho_1|^2}\big)^{{1}/{2}}
	\end{pmatrix},
\end{split}
\end{equation}
for every
	$\BSinitprice = (\BSinitprice_1, \BSinitprice_2, \BSinitprice_3) \in \R^\dimProblem$,
	$T \in [0,\infty)$,
	$\BSrate \in \R$,
	$\BSdividend = (\BSdividend_1, \BSdividend_2, \BSdividend_3) \in \R^\dimProblem$,
	$\sigma = (\sigma_1, \sigma_2, \sigma_3) \in \R^\dimProblem$,
	$\rho \in R$,
	$w \in \R^\dimProblem$
let 
$ 
	X^{\BSinitprice, \BSrate, \BSdividend, \sigma, \rho, w}_T
= 
	(
	X^{\BSrate, \BSinitprice, \BSdividend, \sigma, \rho, w, 1}_T, 
	X^{\BSrate, \BSinitprice, \BSdividend, \sigma, \rho, w, 2}_T, 
	X^{\BSrate, \BSinitprice, \BSdividend, \sigma, \rho, w, 3}_T
	) 
	\in \R^{ \dimProblem }
$ 
satisfy\footnote{Note that for all
	$n \in \N$,
	$v= (v_1, \ldots, v_n)$, 
	$w = (w_1, \ldots, w_n) \in \R^n$
it holds that 
$
	\langle v, w \rangle = \sum_{i = 1}^n v_i w_i
$.
} for all
	$i \in \{1, 2, 3\}$
that
\begin{equation}
\label{BSworstPut:eq1}
\begin{split} 
	X^{\BSinitprice, \BSrate, \BSdividend, \sigma, \rho, w, i}_T
=
	\BSinitprice_i 
	\exp\Bigl(
		\bigl[\BSrate - \BSdividend_i -  \tfrac{|\sigma_i|^2}{2}\bigr]T
		+
		\sqrt{T}
		\sigma_i 
		\langle 
			\Choleski(\rho) w
			,
			e_i
		\rangle 
	\Bigr),
\end{split}
\end{equation}
let 
	$\Minfunction \colon \R^\dimProblem \to \R$ 
satisfy for all 
	$x = (x_1, x_2, x_3) \in \R^\dimProblem$
that
$\Minfunction(x) = \min\{ x_1, x_2, x_3 \}$,
let 
$
	\phi \colon \Param \times \R^\dimProblem \to \R
$
satisfy for all
	$p = (\BSinitprice, T, \BSrate, \BSdividend, \sigma, \rho, K) \in \Param$,
	$w \in \R^\dimProblem$
that
\begin{equation}
\label{BSworstPut:eq2.2}
\begin{split} 
	\phi(p, w)
=
	\exp(-\BSrate T)
	\max\bigl\{
		K - 
		\Minfunction(
			X^{\BSinitprice, \BSrate, \BSdividend, \sigma, \rho, w}_T
		)
		,
		0
	\bigr\},
\end{split}
\end{equation}
let $\WRV \colon \Omega \to \R^{ \dimProblem } $ be a standard normal random vector,
and let $u \colon \Param \to \R$ satisfy for all
	$p \in \Param$
that
\begin{equation}
\label{BSworstPut:eq3}
\begin{split} 
	u(p)
=
	\Exp{
		\phi(p, \WRV)
	}.
\end{split}
\end{equation}
In the economic interpretation of the Black-Scholes model for every 
	$p = (\BSinitprice, T, \BSrate, \BSdividend, \sigma, \rho, K) \in \Param$
the number $u(p) \in \R$ corresponds to the fair price of a worst-of basekt put option on three underlying assets with 
	initial prices $\BSinitprice$,
	time of maturity $T$,
	risk free rate $\BSrate$,
	dividend yields of the respective underlying assets $\BSdividend$,
	volatilities of the respective underlying assets $\sigma$,
	covariance matrix of the Brownian motions $\covMatrix(\rho)$,
	and
	strike price $K$.

We now specify the mathematical objects in the LRV strategy appearing in \cref{algo:general} 
to approximately calculate the target function $u \colon \Param \to \R$ in \cref{BSworstPut:eq3}.
Specifically, in addition to the assumptions above, 
let 
	$\NumMC \in \N$,
	$\NumRefMC =  \numRefAlgSamplesBSworstPut$,
assume
	$\ApproxAlgdim =  \NumMC \dimProblem$,
	$\ApproxRefdim =  \NumRefMC \dimProblem$,
	$\dimSolution = 1$,
	$\Batchsize = {4096}$,
assume for all 
	$ p = (\BSinitprice, T, \BSrate, \BSdividend, \sigma, \rho, K) \in \Param $,
	$ w = (w_1,\ldots, w_{\NumRefMC \dimProblem}) \in \R^{\NumRefMC \dimProblem}$,
	$ \theta = (\theta_1,\ldots, \theta_{\NumMC \dimProblem}) \in \R^{\NumMC \dimProblem}$
that
\begin{equation}
\label{BSworstPut:eq7}
\begin{split} 
	\ApproxRef(p, w)
=
	\frac{ 1 }{ \NumRefMC }
	\left[ 
		\textstyle \sum\limits_{ \MCvariable = 1 }^{ \NumRefMC }
			\phi\big( 
				p, 
				( 
				w_{ ( \MCvariable - 1 )\dimProblem + k }
				)_{k \in \{1, 2, 3 \}}         
			\big)
	\right]
\,\,\,\text{and}\,\,\,
	\ApproxAlg(p, \theta)
=
	\frac{ 1 }{ \NumMC }
	\left[ 
		\textstyle \sum\limits_{ \MCvariable = 1 }^{ \NumMC }
			\phi\big( 
				p, 
				( 
				\theta_{ ( \MCvariable - 1 )\dimProblem + k }
				)_{k \in \{1, 2, 3 \}}         
			\big)
	\right],
\end{split}
\end{equation}
assume for all
	$x, y \in \R$
that
	$\costfct(x, y) = | x-y |^2$,
let
	$\mathbf{a} = {\frac{9}{10}}$,
	$\mathbf{b}  = {\frac{999}{1000}}$,
	$ \varepsilon \in (0,\infty)$,
let
	$(\gamma_m)_{m \in \N} \subseteq (0,\infty)$
satisfy for all 
	$m \in \{1, 2, \ldots, \numTrainstepsBSworstPut \}$
that
$
	\gamma_m 
=  
	{\mathbbm{1}_{(0,15000]}(m)}{10^{-3}}
	+
	{\mathbbm{1}_{(15000,25000]}(m)}{10^{-4}}
	+
	{\mathbbm{1}_{(25000,30000]}(m)}{10^{-5}}
$,
assume for all
	$m \in \N$,
	$i \in \{1, 2,\ldots,\ApproxAlgdim \}$,
	$g_1 = (g_1^{(1)},\ldots,g_1^{(\ApproxAlgdim)})$, 	$g_2 = (g_2^{(1)},\ldots,g_2^{(\ApproxAlgdim)})$,
	$\dots$, 
	$g_m = (g_m^{(1)},\ldots,g_m^{(\ApproxAlgdim)})
	\in 
		\R^{\ApproxAlgdim}
	$
that
\begin{equation}
\label{BSworstPut:eq7.2}
\begin{split} 
	\psi_m^{(i)}(g_1, g_2, \ldots, g_m)
=
	\gamma_m
	\left[
		\frac{
			\sum_{k = 1}^m \mathbf{a}^{m-k}(1-\mathbf{a}) g_k^{(i)}
		}{
			1-\mathbf{a}^m
		}
	\right]
	\left[
		\varepsilon +\bigg[
			\frac{
				\sum_{k = 1}^m \mathbf{b}^{m-k}(1-\mathbf{b}) |g_k^{(i)}|^2
			}{
				1- \mathbf{b}^m
			}
		\bigg]^{\nicefrac{1}{2}} 
	\right]^{-1},
\end{split}
\end{equation}	
assume that  
$ \ApproxAlgRV$ is a standard normal random vector,
assume that  
$  
  \ParamRV_{ 
    1, 1 
  }
$ is $\mathcal{U}_{\Param}$-distributed,
and
assume that 
$ 
  W^{ 1, 1 }
$
is a standard normal random vector.

Let us add some comments regarding the setup introduced above.
Note that the proposal algorithm in \cref{BSworstPut:eq3} corresponds to the MC method. 
Furthermore, observe that \cref{BSworstPut:eq7.2} describes the Adam optimizer in the setup of \cref{algo:general} (cf.\ Kingma \& Ba \cite{Kingma2014} and \cref{sect:Adam}).

In \cref{table:BSworstPut_LRV} we approximately present for
	$\NumMC \in \{2^5, 2^6, \ldots, 2^{13}\}$
one random realization of the $L^1(\lambda_{\Param};\R)$-approximation error
\begin{equation}
\label{BSworstPut_LRV:eq1}
\begin{split} 
\textstyle
	\int_\Param 
		|u(p) - \ApproxAlg(p, \Theta_{\numTrainstepsBSworstPut})|
	\d p
\end{split}
\end{equation}
(\nth{3} column in \cref{table:BSworstPut_LRV}),
one random realization of the $L^2(\lambda_{\Param};\R)$-approximation error
\begin{equation}
\label{BSworstPut_LRV:eq2}
\begin{split} 
\textstyle
	\bigl[
		\int_\Param 
			|u(p) - \ApproxAlg(p, \Theta_{\numTrainstepsBSworstPut})|^2
		\d p
	\bigr]^{\nicefrac{1}{2}}
\end{split}
\end{equation}
(\nth{4} column in \cref{table:BSworstPut_LRV}),
one random realization of the $L^\infty(\lambda_{\Param};\R)$-approximation error
\begin{equation}
\label{BSworstPut_LRV:eq3}
\begin{split} 
\textstyle
	\sup_{p \in \Param}
		|u(p) - \ApproxAlg(p, \Theta_{\numTrainstepsBSworstPut})|,
\end{split}
\end{equation}
(\nth{5} column in \cref{table:BSworstPut_LRV}), and
the time to compute $\Theta_{\numTrainstepsBSworstPut}$
(\nth{6} column in \cref{table:BSworstPut_LRV}).
For every 	
	$\NumMC \in \{2^5, 2^6, \ldots, 2^{13}\}$
we approximated the integrals in \eqref{BSworstPut_LRV:eq1} and \eqref{BSworstPut_LRV:eq2} with the MC method based on $\numMCsamplesBSworstPut$ samples
and we approximated the supremum in \eqref{BSworstPut_LRV:eq3} based on $\numMCsamplesBSworstPut$ random samples 
(cf., e.g., Beck et al.\ \cite[Lemma 3.5]{BeckJafaari21} and Beck et al.\ \cite[Section 3.3]{Beck2019published}).
In our approximations of \eqref{BSworstPut_LRV:eq1}, \eqref{BSworstPut_LRV:eq2}, and \eqref{BSworstPut_LRV:eq3} we have approximately computed for all required sample points 
	$p \in \Param$
the value $u(p)$ of the unknown exact solution 
by means of an MC approximation with $ \numRefSamplesBSworstPut$ MC samples.

Besides the LRV strategy we also employed the standard MC method to approximate the function $u \colon \Param \to \R$ in \eqref{BSworstPut:eq3}. 
In \cref{table:BSworstPut_MC} we present the corresponding numerical simulation results.
In \cref{table:BSworstPut_MC}
	we have approximated the $L^1(\lambda_{\Param};\R)$-approximation error
		of the MC method with the MC method based on $\numMCsamplesBSworstPut$ samples,
	we have approximated the $L^2(\lambda_{\Param};\R)$-approximation errors
		of the MC method with the MC method based on $\numMCsamplesBSworstPut$ samples, and 
	we have approximated the $L^\infty(\lambda_{\Param};\R)$-approximation errors
		of the MC method based on $\numMCsamplesBSworstPut$ random samples
		(cf., e.g., Beck et al.\ \cite[Lemma 3.5]{BeckJafaari21} and Beck et al.\ \cite[Section 3.3]{Beck2019published}).
In our approximations of the above mentioned approximation errors we have approximately computed for all required sample points 
	$p \in \Param$
the value $u(p)$ of the unknown exact solution by means of an MC approximation with $\numRefSamplesBSworstPut$ MC samples.

\begin{table}[H] \tiny
\resizebox{\textwidth}{!}{
\begin{filecontents*}{Table_7.tex}
num_lrv_samples, num_params, total_nr_eval_solutions, l1_error, l_2_error, l_inf_error, rel_l1_error, rel_l_2_error, rel_l_inf_error, time, dtype, train_time
32,96,128000,0.072564,0.097067,0.693916, nan, nan,1.0,4.44, float32,311.69
64,192,128000,0.035654,0.047869,0.335918, nan, nan,1.0,4.58, float32,313.22
128,384,128000,0.018049,0.024193,0.169165, nan, nan,1.0,3.75, float32,314.48
256,768,128000,0.009603,0.012866,0.089918, nan, nan,1.6131962538,4.15, float32,318.79
512,1536,128000,0.005317,0.007125,0.064743, nan, nan,1.0,4.04, float32,337.47
1024,3072,128000,0.003129,0.004171,0.029337, nan, nan,1.86762321,4.16, float32,369.79
2048,6144,128000,0.002132,0.002827,0.021137, nan, nan,1.0,4.02, float32,443.43
4096,12288,128000,0.001594,0.002108,0.014027, nan, nan,1.3932343721,4.13, float32,594.14
8192,24576,128000,0.001242,0.001642,0.011484, nan, nan,1.0,4.12, float32,900.10
\end{filecontents*}
\csvreader[tabular=|c|c|c|c|c|c|,
     table head=
     \hline $\NumMC$ & \thead{Number \\of \\trainable \\parameters}&\thead{$L^1$-approx. \\ error}& \thead{$L^2$-approx. \\error}&\thead{ $L^\infty$-approx. \\ error} 
     & \thead{Training \\ time \\ in seconds} 
     \\\hline,
    late after line=\\\hline,
    filter expr={
		test{\ifnumless{\thecsvinputline}{11}}
    }
    ]
{Table_7.tex}
{
	num_lrv_samples=\nsamp, 
	num_params =\nparams,
	l1_error=\lll, 
	l_2_error=\llll, 
	l_inf_error=\linf, 
	train_time=\train,
	time=\eval  ,
	dtype = \dtype
}
{ \nsamp& \nparams &  \lll & \llll & \linf &\train
}
}
\caption{\label{table:BSworstPut_LRV}Numerical simulations for the LRV strategy in case of the Black-Scholes model for European worst-of basket put options on three underlyings described in \cref{simul:BSworstPut} (15-dimensional approximation problem)
}
\end{table}

\begin{table}[H] \tiny
\resizebox{\textwidth}{!}{
\begin{filecontents*}{Table_8.tex}
num_samples, num_RVs, total_nr_eval_solutions, l1_error, l_2_error, l_inf_error, l1_rel_error, l_2_rel_error, l_inf_rel_error, time, dtype
32,96,128000,1.543287,2.080271,12.685066, nan, nan,75.5249633789,4.45, float32
64,192,128000,1.091341,1.470045,9.742924, nan, nan,9.6147603989,3.57, float32
128,384,128000,0.770266,1.038834,6.845669, nan, nan,47.6968917847,3.68, float32
256,768,128000,0.544692,0.736095,4.272978, nan, nan,11.5005512238,4.17, float32
512,1536,128000,0.386441,0.521668,2.919704, nan, nan,6.9790463448,4.18, float32
1024,3072,128000,0.272983,0.369065,2.299841, nan, nan,2.5993952751,4.02, float32
2048,6144,128000,0.192779,0.260462,1.656979, nan, nan,6.1522140503,4.34, float32
4096,12288,128000,0.136122,0.183617,1.268164, nan, nan,13.0596475601,4.09, float32
8192,24576,128000,0.096641,0.130191,0.825016, nan, nan,4.6959347725,4.03, float32
16384,49152,128000,0.068338,0.092281,0.577732, nan, nan,5.2178397179,4.27, float32
32768,98304,128000,0.048270,0.065142,0.452351, nan, nan,10.7215461731,4.37, float32
\end{filecontents*}
\csvreader[tabular=|c|c|c|c|c|c|c|c|,
     table head=
     \hline 
     \thead{Number \\ of \\MC \\samples}&  
     \thead{Number of \\ scalar random \\ variables \\per evaluation} 
     &\thead{$L^1$-approx. \\ error} 
     &\thead{$L^2$-approx. \\error}
     &\thead{ $L^\infty$-approx. \\ error} 
     \\\hline,
    late after line=\\\hline]{Table_8.tex}
{
	num_samples=\nsamp, 
	num_RVs = \ninputs,
	l1_error=\lll, 
	l_2_error=\llll, 
	l_inf_error=\linf, 
	time=\train, 
	dtype = \dtype
}
{
	\nsamp
	& \ninputs 
	& \lll  
	& \llll 
	& \linf 
}
}
\caption{\label{table:BSworstPut_MC}Numerical simulations for the standard MC method in case of the Black-Scholes model for European worst-of basket put options on three underlyings described in \cref{simul:BSworstPut} (15-dimensional approximation problem)}
\end{table}

\subsection[Parametric Black-Scholes equations for multi-asset average put options with barriers]{Parametric Black-Scholes partial differential equations for multi-asset average put options with knock-in barriers}
\label{simul:BSaverageBarrier}

\newcommand{\numMCsamplesBSaverageBarrier}{12800}
\newcommand{\numRefSamplesBSaverageBarrier}{52428800}
\newcommand{\numRefAlgSamplesBSaverageBarrier}{1024}
\newcommand{\numTrainstepsBSaverageBarrier}{40000}

\newcommand{\BarrierSet}{\mathfrak{B}}
\newcommand{\Maxfunction}{\mathbf{h}}
\newcommand{\Middlefunction}{\mathbf{m}}

\newcommand{\Trigger}{\mathscr{T}}
\newcommand{\UppercrossProb}{\mathscr{U}}
\newcommand{\LowercrossProb}{\mathscr{L}}
\newcommand{\crossProb}{\mathscr{P}}

In this section we apply the LRV strategy to a more complicated and practically relevant Black-Scholes option pricing problem. 
Specifically, we approximate the fair price of a European average basket put option on three underlyings with a knock-in barrier in the Black-Scholes model (cf.\ Black \& Scholes \cite{BlackScholes73} and Merton \cite{Merton73}).
We start by introducing the Black-Scholes model for this pricing problem in the context of \cref{algo:general}.

Assume \cref{algo:general},
let 
	$\dimProblem = 3$, 
	$e_1 = (1, 0, 0)$, 
	$e_2 = (0, 1, 0)$, 
	$e_3 = (0, 0, 1)$,
let $\covMatrix \colon \R^\dimProblem \to \R^{\dimProblem \times \dimProblem}$ satisfy for all 
	$\rho = (\rho_1, \rho_2, \rho_3) \in \R^\dimProblem$
that
\begin{equation}
\begin{split} 
	\covMatrix(\rho) 
= 
	\begin{pmatrix}
		1 &\rho_1 &\rho_2 \\
		\rho_1 &1 &\rho_3 \\
		\rho_2 &\rho_3 &1
	\end{pmatrix},
\end{split}
\end{equation}
let 
$
	R
\subseteq 
	\R^\dimProblem
$
satisfy
$
	R 
= 
	\bigl\{
		\rho = (\rho_1, \rho_2, \rho_3) \in [-\frac{95}{100},\frac{95}{100}]^\dimProblem 
		\colon  
		1 - |\rho_2|^2 - \frac{(\rho_3 - \rho_1\rho_2)^2}{(1-|\rho_1|^2)^{1/2}} \geq 0
	\bigr\}
$,
assume
	$\dimParam = 16$
and
\begin{equation}
\begin{split} 
	\Param 
= 
	[90, 110]^3 \times
	[\tfrac{1}{2}, 1] \times
	[-\tfrac{1}{20}, \tfrac{1}{20}] \times
	[0, \tfrac{1}{10}]^3 \times
	[\tfrac{1}{100}, \tfrac{1}{2}]^3 \times 
	R \times 
	[90, 110] \times
	[70, 80],
\end{split}
\end{equation}
let $\Choleski \colon R \to \R^{\dimProblem \times \dimProblem}$ satisfy for all 
	$\rho = (\rho_1, \rho_2, \rho_3) \in R$
that
\begin{equation}
\begin{split} 
	\Choleski(\rho)
=
	\begin{pmatrix}
		1 & 0 & 0 \\
		\rho_1& (1 - |\rho_1|^2)^{{1}/{2}} &0 \\
		\rho_2& \frac{\rho_3 - \rho_1\rho_2}{(1 - |\rho_1|^2)^{{1}/{2}}} & \big(1 - |\rho_2|^2 - \frac{(\rho_3 - \rho_1\rho_2)^2}{1 - |\rho_1|^2}\big)^{{1}/{2}}
	\end{pmatrix},
\end{split}
\end{equation}
let $\WRV \colon \Omega \to C([0, \infty), \R^{ \dimProblem }) $ be a standard Brownian motion with continuous sample paths,
for every
	$\BSinitprice = (\BSinitprice_1, \BSinitprice_2, \BSinitprice_3) \in \R^\dimProblem$,
	$\BSrate \in \R$,
	$\BSdividend = (\BSdividend_1, \BSdividend_2, \BSdividend_3) \in \R^\dimProblem$,
	$\sigma = (\sigma_1, \sigma_2, \sigma_3) \in \R^\dimProblem$,
	$\rho \in \R^\dimProblem$
let 
$ 
	X^{\BSinitprice, \BSrate, \BSdividend, \sigma, \rho} 
= 
	(
	X^{\BSrate, \BSinitprice, \BSdividend, \sigma, \rho, 1}, 
	X^{\BSrate, \BSinitprice, \BSdividend, \sigma, \rho, 2}, 
	X^{\BSrate, \BSinitprice, \BSdividend, \sigma, \rho, 3}
	) 
	\colon 
	C([0, \infty), \R^{ \dimProblem })
	\to 
	C([0, \infty), \R^{ \dimProblem })
$ 
satisfy for all
	$w \in C([0, \infty), \R^{ \dimProblem })$,
	$t \in [0,\infty)$,
	$i \in \{1, 2, 3\}$
that
\begin{equation}
\label{BSaverageBarrier:eq1}
\begin{split} 
	X^{\BSinitprice, \BSrate, \BSdividend, \sigma, \rho, i}_t(w)
=
	\BSinitprice_i 
	\exp\Bigl(
		\bigl[\BSrate - \BSdividend_i -  \tfrac{|\sigma_i|^2}{2} \bigr]t 
		+
		\sigma_i 
		\langle 
			\Choleski(\rho) w_t
			,
			e_i
		\rangle
	\Bigr),
\end{split}
\end{equation}
let 
	$\Minfunction \colon \R^\dimProblem \to \R$ and 
	$\Middlefunction \colon \R^\dimProblem \to \R$ 
satisfy for all 
	$x = (x_1, x_2, x_3) \in \R^\dimProblem$
that
$\Minfunction(x) = \min\{ x_1, x_2, x_3 \}$ and
$\Middlefunction(x) = \tfrac{ x_1 + x_2 + x_3 }{3}$,
for every 
	$T \in [0,\infty)$,
	$B \in \R$
let $\BarrierSet_{T, B} \subseteq C([0, \infty), \R^{ \dimProblem })$ satisfy
\begin{equation}
	\label{BSaverageBarrier:eq2}
	\BarrierSet_{T, B}
=
	\bigl\{
		w \in C([0, \infty), \R^{ \dimProblem })
		\colon 
		(
			\exists \, t \in [0,T] 
			\colon 
			\Minfunction(w_t)
			< B
		)  
	\bigr\},
\end{equation}
let 
$
	\phi \colon \Param \times C([0,\infty), \R^\dimProblem) \to \R
$
satisfy for all
	$p = (\BSinitprice, T, \BSrate, \BSdividend, \sigma, \rho, K, B) \in \Param$,
	$w \in C([0,\infty), \R^\dimProblem)$
that
\begin{equation}
\label{BSaverageBarrier:eq2.2}
\begin{split} 
	\phi(p, w)
=
	\mathbbm{1}_{\BarrierSet_{T,B}}(X^{\BSinitprice, \BSrate, \BSdividend, \sigma, \rho}(w))
	\exp(-\BSrate T)
	\max\bigl\{
		K - 
		\Middlefunction(
			X^{\BSinitprice, \BSrate, \BSdividend, \sigma, \rho}_T(w)
		)
		,
		0
	\bigr\},
\end{split}
\end{equation}
and let $u \colon \Param \to \R$ satisfy for all
	$p \in \Param$
that
\begin{equation}
\label{BSaverageBarrier:eq3}
\begin{split} 
	u(p)
=
	\Exp{
		\phi(p, \WRV)
	}.
\end{split}
\end{equation}
In the economic interpretation of the Black-Scholes model for every 
	$p = (\BSinitprice, T, \BSrate, \BSdividend, \sigma, \rho, K, B) \in \Param$
the number $u(p) \in \R$ corresponds to the fair price of a average basekt put option on three underlying assets with 
	initial prices $\BSinitprice$,
	time of maturity $T$,
	risk free rate $\BSrate$,
	dividend yields of the respective underlying assets $\BSdividend$,
	volatilities of the respective underlying assets $\sigma$,
	covariance matrix of the Brownian motions $\covMatrix(\rho)$,
	strike price $K$, and
	knock-in barrier $B$.

We now specify the mathematical objects in the LRV strategy appearing in \cref{algo:general} 
to approximately calculate the target function $u \colon \Param \to \R$ in \cref{BSaverageBarrier:eq3}.
Specifically, in addition to the assumptions above, 
let 
	$\numTimePoints = 10$,
	$\NumRefMC = \numRefAlgSamplesBSaverageBarrier$,
	$\NumMC \in \N$,
assume
	$\ApproxAlgdim =  \NumMC \numTimePoints \dimProblem$,
	$\ApproxRefdim = \NumRefMC \numTimePoints \dimProblem$,
	$\dimSolution = 1$,
	$\Batchsize = 1024$,
for every 
	$p = (\BSinitprice, T, \BSrate, \BSdividend, \sigma, \rho, K, B) \in \Param$,
	$w = (w_1, \ldots, w_\numTimePoints) \in \R^{\numTimePoints \dimProblem}$
let 
$
	\mathcal{X}^{p, w} 
= 
	(\mathcal{X}^{p, w, 1}, \mathcal{X}^{p, w, 2}, \mathcal{X}^{p, w, 3} ) 
	\colon 
	\{0, 1, \ldots, \numTimePoints\} \to \R^\dimProblem
$
satisfy for all 
	$n \in \{1, 2, \ldots, \numTimePoints \}$,
	$i \in \{1, 2, 3\}$
that
$
	\mathcal{X}^{p, w}_0 = \BSinitprice	
$
and
\begin{equation}
\label{BSaverageBarrier:eq4}
\begin{split} 
	\mathcal{X}^{p, w, i}_n
&= 
	\mathcal{X}^{p, w, i}_{n-1}
	\exp\Bigl(
		\tfrac{T}{N}
		\bigl[\BSrate - \BSdividend_i -  \tfrac{|\sigma_i|^2}{2}\bigr]
		+
		\big[\tfrac{T}{N}\big]^{1/2}
		\sigma_i 
		\langle 
			\Choleski(\rho) w_n
			,
			e_i
		\rangle
	\Bigr),
\end{split}
\end{equation}
for every 
	$p = (\BSinitprice, T, \BSrate, \BSdividend, \sigma_1, \sigma_2, \sigma_3, \rho, K, B) \in \Param$ 
let 
	$ \Trigger_p = (\Trigger_{p,1}, \Trigger_{p,2}, \Trigger_{p,3}) \colon (0,\infty)^\dimProblem \times (0,\infty)^\dimProblem \to \R^\dimProblem$,
	$ \UppercrossProb_p \colon (0,\infty)^\dimProblem \times (0,\infty)^\dimProblem \to [0,1]$,
	$ \LowercrossProb_p \colon (0,\infty)^\dimProblem \times (0,\infty)^\dimProblem \to [0,1]$, and
	$ \crossProb_p \colon ((0,\infty)^\dimProblem)^{N+1} \to [0,1]$	
satisfy for all	
	$x = (x_1, x_2, x_3)$, 
	$y  = (y_1, y_2, y_3) \in (0,\infty)^\dimProblem$,
	$\mathbf{x}_0, \mathbf{x}_1, \ldots, \mathbf{x}_N \in (0,\infty)^\dimProblem$,
	$i \in \{1, 2, 3\}$
that
\begin{equation}
\label{BSaverageBarrier:eq5.1}
\begin{split} 
	\Trigger_{p,i}(x, y) 
= 
	\begin{cases}
		1 &\colon \min\{ x, y \} < B \\
		\exp\!\left(
			-\frac{
				2 \ln(x_i/B)\ln(y_i/B)
			}{
				(\sigma_i)^2 T/N
			}
		\right)
		&\colon \min\{ x, y \} \geq B,
	\end{cases}
\end{split}
\end{equation}
\begin{equation}
\label{BSaverageBarrier:eq5.2}
\begin{split} 
	\UppercrossProb_{p}(x, y) 
= 
	1 - \max\bigl( \cup_{j \in \{1, 2, 3\}} \{\Trigger_{p,j}(x, y) \} \bigr), \qquad
	\LowercrossProb_{p}(x, y) 
= 
	\max\bigl\{ 1 - \textstyle \sum_{j = 1}^3 \Trigger_{p,j}(x, y), 0  \bigr\},
\end{split}
\end{equation}
\begin{equation}
\label{BSaverageBarrier:eq5.3}
\begin{split} \textstyle
\andq
	\crossProb_{p}(\mathbf{x}_0, \mathbf{x}_1, \ldots, \mathbf{x}_N) 
= 
	\frac{1}{2}
	\left(
		2 
		- 
		\left[ 
			\prod\limits_{n = 1}^N
				\UppercrossProb_{p}(\mathbf{x}_{n-1}, \mathbf{x}_{n}) 
		\right]
		-
		\left[ 
			\prod\limits_{n = 1}^N
				\LowercrossProb_{p}(\mathbf{x}_{n-1}, \mathbf{x}_{n}) 
		\right]
	\right),
\end{split}
\end{equation}
let 
$
	\eulerapprox \colon
	\Param \times \R^{\numTimePoints \dimBM} \to
	\R
$
satisfy for all 
	$ p = (\BSinitprice, T, \BSrate, \BSdividend, \sigma, \rho, K, B) \in \Param $,
	$ w \in \R^{\numTimePoints \dimProblem}$
that
\begin{equation}
\label{BSaverageBarrier:eq6}
\begin{split} 
	\eulerapprox(p, w)
=
	\crossProb_p(\mathcal{X}^{p, w}_{0}, \mathcal{X}^{p, w}_1, \ldots, \mathcal{X}^{p, w}_N)
	\exp(-\BSrate T)
	\max\bigl\{
		K - 
		\Middlefunction(
			\mathcal{X}^{p, w}_N
		)
		,
		0
	\bigr\},
\end{split}
\end{equation}
assume for all 
	$ p = (\BSinitprice, T, \BSrate, \BSdividend, \sigma, \rho, K, B) \in \Param $,
	$ w = (w_1,\ldots, w_{\NumRefMC \numTimePoints \dimProblem}) \in \R^{\NumRefMC \numTimePoints \dimProblem}$,
	$ \theta = (\theta_1,\ldots, \theta_{\NumMC \numTimePoints \dimProblem}) \in \R^{\NumMC \numTimePoints \dimProblem}$
that
\begin{equation}
\label{BSaverageBarrier:eq7.0}
\begin{split} 
	\ApproxRef(p, w)
=
	\frac{ 1 }{ \NumRefMC }
	\left[ 
		{\textstyle \sum\limits_{ \MCvariable = 1 }^{ \NumRefMC } }
			\eulerapprox\big( 
				p, 
				( 
				w_{ ( \MCvariable - 1 ) \numTimePoints\dimProblem + k }
				)_{k \in \{1, 2, \ldots, \numTimePoints\dimProblem \}}         
			\big)
	\right]
\end{split}
\end{equation}
\begin{equation}
\label{BSaverageBarrier:eq7}
\begin{split} 
\qandq
	\ApproxAlg(p, \theta)
=
	\frac{ 1 }{ \NumMC }
	\left[ 
		{\textstyle \sum\limits_{ \MCvariable = 1 }^{ \NumMC }}
			\eulerapprox\big( 
				p, 
				( 
				\theta_{ ( \MCvariable - 1 ) \numTimePoints\dimProblem + k }
				)_{k \in \{1, 2, \ldots, \numTimePoints\dimProblem \}}         
			\big)
	\right],
\end{split}
\end{equation}
assume for all
	$x, y \in \R$
that
	$\costfct(x, y) = | x-y |^2$,
let
	$\mathbf{a} = {\frac{9}{10}}$,
	$\mathbf{b}  = {\frac{999}{1000}}$,
	$ \varepsilon \in (0,\infty)$,
let
	$(\gamma_m)_{m \in \N} \subseteq (0,\infty)$
satisfy for all 
	$m \in \{1, 2, \ldots, \numTrainstepsBSaverageBarrier \}$
that
$
	\gamma_m 
=  
	{\mathbbm{1}_{(0,20000]}(m)}{10^{-3}}
	+
	{\mathbbm{1}_{(20000,30000]}(m)}{10^{-4}}
	+
	{\mathbbm{1}_{(30000,40000]}(m)}{10^{-5}}
$,
assume for all
	$m \in \N$,
	$i \in \{1, 2,\ldots,\ApproxAlgdim \}$,
	$g_1 = (g_1^{(1)},\ldots,g_1^{(\ApproxAlgdim)})$, 	$g_2 = (g_2^{(1)},\ldots,g_2^{(\ApproxAlgdim)})$,
	$\dots$, 
	$g_m = (g_m^{(1)},\ldots,g_m^{(\ApproxAlgdim)})
	\in 
		\R^{\ApproxAlgdim}
	$
that
\begin{equation}
\label{BSaverageBarrier:eq7.2}
\begin{split} 
	\psi_m^{(i)}(g_1, g_2, \ldots, g_m)
=
	\gamma_m
	\left[
		\frac{
			\sum_{k = 1}^m \mathbf{a}^{m-k}(1-\mathbf{a}) g_k^{(i)}
		}{
			1-\mathbf{a}^m
		}
	\right]
	\left[
		\varepsilon +\bigg[
			\frac{
				\sum_{k = 1}^m \mathbf{b}^{m-k}(1-\mathbf{b}) |g_k^{(i)}|^2
			}{
				1- \mathbf{b}^m
			}
		\bigg]^{\nicefrac{1}{2}} 
	\right]^{-1},
\end{split}
\end{equation}	
assume that  
$ \ApproxAlgRV$ is a standard normal random vector,
assume that  
$  
  \ParamRV_{ 
    1, 1 
  }
$ is $\mathcal{U}_{\Param}$-distributed,
and
assume that 
$ 
  W^{ 1, 1 }
$
is a standard normal random vector.

Let us add some comments regarding the setup introduced above.
The functions 
	$ \crossProb_p \colon ((0,\infty)^\dimProblem)^{N+1} \to [0,1]$,
	$p \in \Param$,
in \cref{BSaverageBarrier:eq4} are employed to estimate crossing probabilities of Brownian bridges as proposed in Shevchenko \cite{Shevchenko2003} (see also, e.g., Gobet \cite{Gobet2009}).
Specifically, note that Gobet \cite[Displays (12), (13)]{Gobet2009} suggests that for all
	$p = (\BSinitprice, T, \BSrate, \BSdividend, \sigma, \rho, K, B) \in \Param$
we have $\P$-a.s.\ that
\begin{equation}
\label{T_B_D}
\begin{split} 
	&\crossProb_p(
		X^{\BSinitprice, \BSrate, \BSdividend, \sigma, \rho}_0(\WRV), 
		X^{\BSinitprice, \BSrate, \BSdividend, \sigma, \rho}_{T/N}(\WRV), 
		\ldots, 
		X^{\BSinitprice, \BSrate, \BSdividend, \sigma, \rho}_T(\WRV)
	) \\
&\approx
	\P\bigl(
			\exists \, t \in [0, T] 
			\colon 
			\Minfunction(
				X^{\BSinitprice, \BSrate, \BSdividend, \sigma, \rho}_t(\WRV)
			)	
			<
			B
		\,\big|\,
			(
				\WRV_0, 
				\WRV_{T/N},
				\ldots, 
				\WRV_T 
			)
	\bigr).
\end{split}
\end{equation}
Combining this with 
	\cref{BSaverageBarrier:eq1},
	\cref{BSaverageBarrier:eq4},
	\cref{BSaverageBarrier:eq6}, 
	the tower property for conditional expectations, and
	the fact that Brownian motions have independent increments
suggests for all
	$p = (\BSinitprice, T, \BSrate, \BSdividend, \sigma, \rho, K, B) \in \Param$
that
\begin{equation}
\label{T_B_D}
\begin{split} 
	&\Exp{\eulerapprox(p, \eulerRV)}
=
	\Exp{
		\eulerapprox\big( 
			p, \sqrt{{\numTimePoints}/{T}}\,
			(\WRV_{T / \numTimePoints} - \WRV_0, \WRV_{ 2 T / \numTimePoints } - \WRV_{ T / \numTimePoints }, \dots, \WRV_T -\WRV_{ (\numTimePoints-1) T / \numTimePoints } )
		\big)
	} \\
&=
	\E \Bigg[
		\crossProb_p(
			X^{\BSinitprice, \BSrate, \BSdividend, \sigma, \rho}_0(\WRV), 
			X^{\BSinitprice, \BSrate, \BSdividend, \sigma, \rho}_{T/N}(\WRV), 
			\ldots, 
			X^{\BSinitprice, \BSrate, \BSdividend, \sigma, \rho}_T(\WRV)
		)\\
&\quad
		\exp(-\BSrate T)
		\max\bigl\{
			K - 
			\Middlefunction(
				X^{\BSinitprice, \BSrate, \BSdividend, \sigma, \rho}_{T}(\WRV)
			)
			,
			0
		\bigr\}
	\Bigg]\\
&\approx
	\E \Bigg[
		\P\bigl(
				\exists \, t \in [0, T] 
				\colon 
				\Minfunction(
					X^{\BSinitprice, \BSrate, \BSdividend, \sigma, \rho}_t(\WRV) 
				)	
				<
				B
			\,\big|\,
				(
					\WRV_0, 
					\WRV_{T/N},
					\ldots, 
					\WRV_T 
				)
		\bigr)\\
&\quad
		\exp(-\BSrate T)
		\max\bigl\{
			K - 
			\Middlefunction(
				X^{\BSinitprice, \BSrate, \BSdividend, \sigma, \rho}_{T}(\WRV)
			)
			,
			0
		\bigr\}
	\Bigg]\\
&=
	\E \Bigg[
		\P\bigl(
			X^{\BSinitprice, \BSrate, \BSdividend, \sigma, \rho}(\WRV) \in \BarrierSet_{T, B}
			\,\big|\,
				\WRV
		\bigr)
		\exp(-\BSrate T)
		\max\bigl\{
			K - 
			\Middlefunction(
				X^{\BSinitprice, \BSrate, \BSdividend, \sigma, \rho}_{T}(\WRV)
			)
			,
			0
		\bigr\}
	\Bigg]\\
&=
	\E \Bigg[
		\mathbbm{1}_{\BarrierSet_{T, B}}(\WRV)
		\exp(-\BSrate T)
		\max\bigl\{
			K - 
			\Middlefunction(
				X^{\BSinitprice, \BSrate, \BSdividend, \sigma, \rho}_{T}(\WRV)
			)
			,
			0
		\bigr\}
	\Bigg]
=
	\Exp{
		\phi(p, \WRV)
	}
=
	u(p).
\end{split}
\end{equation}
The proposal algorithm in \cref{BSaverageBarrier:eq3} thus corresponds to the MC method based on approximated MC samples.
Furthermore, observe that \cref{BSaverageBarrier:eq7.2} describes the Adam optimizer in the setup of \cref{algo:general} (cf.\ Kingma \& Ba \cite{Kingma2014} and \cref{sect:Adam}).

In \cref{table:BSaverageBarrier_LRV} we approximately present for
	$\NumMC \in \{2^5, 2^6, \ldots, 2^{13}\}$
one random realization of the $L^1(\lambda_{\Param};\R)$-approximation error
\begin{equation}
\label{BSaverageBarrier_LRV:eq1}
\begin{split} 
\textstyle
	\int_\Param 
		|u(p) - \ApproxAlg(p, \Theta_{\numTrainstepsBSaverageBarrier})|
	\d p
\end{split}
\end{equation}
(\nth{3} column in \cref{table:BSaverageBarrier_LRV}),
one random realization of the $L^2(\lambda_{\Param};\R)$-approximation error
\begin{equation}
\label{BSaverageBarrier_LRV:eq2}
\begin{split} 
\textstyle
	\bigl[
		\int_\Param 
			|u(p) - \ApproxAlg(p, \Theta_{\numTrainstepsBSaverageBarrier})|^2
		\d p
	\bigr]^{\nicefrac{1}{2}}
\end{split}
\end{equation}
(\nth{4} column in \cref{table:BSaverageBarrier_LRV}),
one random realization of the $L^\infty(\lambda_{\Param};\R)$-approximation error
\begin{equation}
\label{BSaverageBarrier_LRV:eq3}
\begin{split} 
\textstyle
	\sup_{p \in \Param}
		|u(p) - \ApproxAlg(p, \Theta_{\numTrainstepsBSaverageBarrier})|,
\end{split}
\end{equation}
(\nth{5} column in \cref{table:BSaverageBarrier_LRV}), and
the time to compute $\Theta_{\numTrainstepsBSaverageBarrier}$
(\nth{6} column in \cref{table:BSaverageBarrier_LRV}).
For every 	
	$\NumMC \in \{2^5, 2^6, \ldots, 2^{13}\}$
we approximated the integrals in \eqref{BSaverageBarrier_LRV:eq1} and \eqref{BSaverageBarrier_LRV:eq2} with the MC method based on $\numMCsamplesBSaverageBarrier$ samples
and we approximated the supremum in \eqref{BSaverageBarrier_LRV:eq3} based on $\numMCsamplesBSaverageBarrier$ random samples 
(cf., e.g., Beck et al.\ \cite[Lemma 3.5]{BeckJafaari21} and Beck et al.\ \cite[Section 3.3]{Beck2019published}).
In our approximations of \eqref{BSaverageBarrier_LRV:eq1}, \eqref{BSaverageBarrier_LRV:eq2}, and \eqref{BSaverageBarrier_LRV:eq3} we have approximately computed for all required sample points 
	$p \in \Param$
the value $u(p)$ of the unknown exact solution 
by means of an MC approximation with $\numRefSamplesBSaverageBarrier$ MC samples.

Besides the LRV strategy we also employed the standard MC method to approximate the function $u \colon \Param \to \R$ in \eqref{BSaverageBarrier:eq3}. 
In \cref{table:BSaverageBarrier_MC} we present the corresponding numerical simulation results.
In \cref{table:BSaverageBarrier_MC}
	we have approximated the $L^1(\lambda_{\Param};\R)$-approximation error
		of the MC method with the MC method based on $\numMCsamplesBSaverageBarrier$ samples,
	we have approximated the $L^2(\lambda_{\Param};\R)$-approximation errors
		of the MC method with the MC method based on $\numMCsamplesBSaverageBarrier$ samples, and 
	we have approximated the $L^\infty(\lambda_{\Param};\R)$-approximation errors
		of the MC method based on $\numMCsamplesBSaverageBarrier$ random samples
		(cf., e.g., Beck et al.\ \cite[Lemma 3.5]{BeckJafaari21} and Beck et al.\ \cite[Section 3.3]{Beck2019published}).
In our approximations of the above mentioned approximation errors we have approximately computed for all required sample points 
	$p \in \Param$
the value $u(p)$ of the unknown exact solution by means of an MC approximation with $ \numRefSamplesBSaverageBarrier$ MC samples.

\begin{table}[H] \tiny
\resizebox{\textwidth}{!}{
\begin{filecontents*}{Table_9.tex}
num_lrv_samples, num_params, total_nr_eval_solutions, l1_error, l_2_error, l_inf_error, rel_l1_error, rel_l_2_error, rel_l_inf_error, time, dtype, train_time
32,960,12800,0.102353,0.144240,0.989026, nan, nan,8.0203084946,1.49, float32,1092.81
64,1920,12800,0.059373,0.082772,0.600640, nan, nan,1.1092998981,1.40, float32,1092.24
128,3840,12800,0.034955,0.049141,0.410478, nan, nan,1.0,1.29, float32,1101.67
256,7680,12800,0.020976,0.029907,0.222575, nan, nan,2.2638046741,1.37, float32,1123.90
512,15360,12800,0.013157,0.018535,0.139017, nan, nan,5.3633036613,1.37, float32,1102.52
1024,30720,12800,0.008828,0.012484,0.095372, nan, nan,2.183224678,1.32, float32,1428.70
2048,61440,12800,0.006280,0.008905,0.057179, nan, nan,5.582672596,1.37, float32,2289.94
4096,122880,12800,0.004984,0.007025,0.054062, nan, nan,1.0,1.45, float32,4380.77
8192,245760,12800,0.004056,0.005799,0.050417, nan, nan,1.6618262529,1.32, float32,8742.66
\end{filecontents*}
\csvreader[tabular=|c|c|c|c|c|c|,
     table head=
     \hline $\NumMC$ & \thead{Number \\of \\trainable \\parameters}&\thead{$L^1$-approx. \\ error}& \thead{$L^2$-approx. \\error}&\thead{ $L^\infty$-approx. \\ error} 
     & \thead{Training \\ time \\ in seconds} 
     \\\hline,
    late after line=\\\hline]{Table_9.tex}
{
	num_lrv_samples=\nsamp, 
	num_params =\nparams,
	l1_error=\lll, 
	l_2_error=\llll, 
	l_inf_error=\linf, 
	train_time=\train,
	time=\eval  ,
	dtype = \dtype
}
{ \nsamp& \nparams &  \lll & \llll & \linf &\train }
}
\caption{\label{table:BSaverageBarrier_LRV}Numerical simulations for the LRV strategy in case of the Black-Scholes model for European average put options with knock-in barriers described in \cref{simul:BSaverageBarrier} (16-dimensional approximation problem)}
\end{table}

\begin{table}[H] \tiny
\resizebox{\textwidth}{!}{
\begin{filecontents*}{Table_10.tex}
num_samples, num_RVs, total_nr_eval_solutions, l1_error, l_2_error, l_inf_error, l1_rel_error, l_2_rel_error, l_inf_rel_error, time, dtype
32,960,12800,1.131512,1.550692,8.077627, nan, nan, 12.2661542892,1.43, float32
64,1920,12800,0.780559,1.080267,6.354096, nan, nan, 27.7228260040,1.26, float32
128,3840,12800,0.555417,0.762485,4.352852, nan, nan, 6.8394060135,1.40, float32
256,7680,12800,0.392276,0.540796,3.551247, nan, nan, 91.5822906494,1.29, float32
512,15360,12800,0.280314,0.385225,2.428810, nan, nan, inf,1.40, float32
1024,30720,12800,0.197006,0.270320,1.669893, nan, nan, 74.4789352417,1.30, float32
2048,61440,12800,0.139361,0.192542,1.172441, nan, nan, inf,1.36, float32
4096,122880,12800,0.098571,0.135519,0.899519, nan, nan, inf,1.42, float32
8192,245760,12800,0.070336,0.097068,0.597193, nan, nan, inf,1.37, float32
16384,491520,12800,0.049755,0.068624,0.537441, nan, nan, inf,1.85, float32
\end{filecontents*}
\csvreader[tabular=|c|c|c|c|c|c|c|c|c|c|,
     table head=
     \hline 
     \thead{Number \\ of \\MC \\samples}
     &\thead{Number of \\ scalar random \\ variables \\per evaluation} 
     &\thead{$L^1$-approx. \\ error} 
     & \thead{$L^2$-approx. \\error}
     &\thead{ $L^\infty$-approx. \\ error} 
     \\\hline,
    late after line=\\\hline]{Table_10.tex}
{
	num_samples=\nsamp, 
	num_RVs = \ninputs,
	l1_error=\lll, 
	l_2_error=\llll, 
	l_inf_error=\linf, 
	time=\train, 
	dtype = \dtype
}
{
	\nsamp
	& \ninputs 
	& \lll 
	& \llll 
	& \linf 
}}
\caption{\label{table:BSaverageBarrier_MC}Numerical simulations for the standard MC method in case of the Black-Scholes model for European average put options with knock-in barriers described in \cref{simul:BSaverageBarrier} (16-dimensional approximation problem)}
\end{table}

\subsection{Parametric stochastic Lorentz equations}
\label{simul:Lorentz}

\newcommand{\nrtrainstepsLorentz}{10000}
\newcommand{\numMCsamplesLorentz}{12800} \newcommand{\numRefAlgSamplesLorentz}{{512}}
\newcommand{\NumRefMCLorentz}{{8\,196\,000}}

In this section we apply the LRV strategy to the parametric stochastic Lorentz equation (cf., e.g, Schmallfuss \cite{Schmallfuss1997} and Hutzenthaler \& Jentzen \cite[Section 4.4]{Hutzenthaler12}).
A brief summary of the numerical results of this subsection can be found in \cref{table:Lorentz_summary} below.
We first introduce the parametric stochastic Lorentz equation in the context of \cref{algo:general}.

\begin{table}[H] \tiny
\resizebox{\textwidth}{!}{
\begin{filecontents*}{Table_11.tex}
method;num-samples; num-params; l1-error; l-2-error; l-inf-error; train-time; eval-time
ANNs (2 hidden layers, 128 neurons each);-;18049;4.884268;13.618965;377.537964;1869.14;0.65
ANNs (2 hidden layers, 256 neurons each);-;68865;3.218077;9.487827;106.434830;1882.02;0.66
ANNs (3 hidden layers, 128 neurons each);-;34561;5.824964;15.087503;200.595459;1903.12;0.63
ANNs (3 hidden layers, 256 neurons each);-;134657;93.630541;110.419151;313.946106;1891.98;0.63
Monte Carlo;32768;-;0.010552;0.014373;0.114243;0;4.80
Monte Carlo (antithetic);32768;-;0.000241;0.000457;0.006233;0;12.02
LRV;1024;76800;0.000537;0.000958;0.015106;327.63;0.82
LRV (antithetic);1024;76800;0.000112;0.000223;0.004013;518.68;1.34


\end{filecontents*}
\csvreader[tabular=|c|c|c|c|c|c|c|c|,
	separator=semicolon,
     table head=
     \hline Approximation method& \thead{Number of \\ trainable \\ parameters} & \thead{Number\\ of \\ MC \\ samples} & \thead{$L^2$-approx. \\ error}& \thead{$L^\infty$-approx. \\ error}& \thead{Training \\ time in \\ seconds} 
     \\\hline,
    late after line=\\\hline]{Table_11.tex}
{
	method = \method, 
	num-samples = \nsamp, 
	num-params = \nparams,
	l1-error=\lll, 
	l-2-error=\llll, 
	l-inf-error=\linf, 
	train-time=\train, 
	eval-time=\eval
}
{ \method & \nparams &\nsamp& \llll & \linf &\train 
}
}
\caption{\label{table:Lorentz_summary}Parametric stochastic Lorentz equation}
\end{table}

Assume \cref{algo:general},
let $\dimProblem = 3$, 
let $g \colon \R^\dimProblem \to \R$ satisfy for all
	$x \in \R^\dimProblem$
that
	$g(x) = \|x\|^2$,
for every let 
	$\alpha = (\alpha_1, \alpha_2, \alpha_3) \in \R^\dimProblem$
let $\mu_\alpha \colon \R^\dimProblem \to \R^\dimProblem$ satisfy for all
	$x = (x_1, x_2, x_3) \in \R^\dimProblem$
that
\begin{equation}
\begin{split} 
	\mu_\alpha(x) 
=
	(\alpha_1(x_2 - x_1), \alpha_2 x_1 - x_2 - x_1x_3, x_1x_2 - \alpha_3 x_3),
\end{split}
\end{equation}
and
for every 
	$\alpha = (\alpha_1, \alpha_2, \alpha_3), \beta = (\beta_1, \beta_2, \beta_3) \in \R^\dimProblem$
let $u_{\alpha, \beta}=(u_{\alpha, \beta}(t,x))_{(t,x)\in [0,\infty)\times\R^d}\in C^{1,2}([0,\infty)\times\R^d,\R)$
be an at most polynomially growing function which satisfies
for all 
	$t\in [0,\infty)$, 
	$x\in\R^d$ 
that 
\begin{multline}
\label{Lorentz:eq1}
(\tfrac{\partial u_{\alpha, \beta}}{\partial t})(t,x)
= 
\alpha_1(x_2-x_1) (\tfrac{\partial u_{\alpha, \beta}}{\partial x_1})(t,x)
+ 
(\alpha_2 x_1 - x_2 - x_1 x_3)(\tfrac{\partial u_{\alpha, \beta}}{\partial x_2})(t,x) \\
+ 
(x_1 x_2 - \alpha_3 x_3)(\tfrac{\partial u_{\alpha, \beta}}{\partial x_3})(t,x)
+ 
\sum_{i = 1}^\dimProblem \tfrac{(\beta_i)^2}{2}(\tfrac{\partial^2 u_{\alpha, \beta}}{\partial (x_i)^2})(t,x)
\end{multline}
and $u_{\alpha, \beta}(0,x) = g(x)$.
 
We now specify the mathematical objects in the LRV strategy appearing in \cref{algo:general} 
to approximately calculate $(u_{\alpha,\beta})_{(\alpha,\beta) \in \R^\dimProblem \times \R^\dimProblem}$.
Specifically, in addition to the assumptions above,
let $\odot \colon \R^\dimProblem \times \R^\dimProblem \to \R^\dimProblem$ satisfy for all
	$x = (x_1, x_2, x_3)$,
	$y = (y_1, y_2, y_3)$
that
$	
	x \odot y = (x_1 y_1, x_2 y_2, x_3 y_3)
$,
let 
	$\numTimePoints = {25}$,
	$\NumRefMC = \numRefAlgSamplesLorentz$,
	$\NumMC \in \N$,
	$\antithetic \in \{0,1\}$,
assume
	$\dimParam = 10$,
	$\ApproxAlgdim =  \NumMC \numTimePoints \dimProblem$,
	$\ApproxRefdim = \NumRefMC \numTimePoints \dimProblem$,
	$\dimSolution = 1$,
	$\Batchsize = {512}$, and
	\begin{equation}
	\begin{split} 
		\Param = [0.01,1] \times ([9, 11] \times [13, 15] \times [{1, 2}]) \times [0.05, 0.25]^3 \linebreak \times ([0.5, 2.5] \times [8, 10] \times [10, 12]),
	\end{split}
	\end{equation}
for every 
	$p = (T, \alpha, \beta, x) \in \Param$,
	$w = (w_1, \ldots, w_\numTimePoints) \in \R^{\numTimePoints \dimProblem}$
let 
	$\mathcal{X}^{p, w} \colon \{0, 1, \ldots, \numTimePoints\} \to \R^\dimProblem  $
satisfy for all 
	$n \in \{1, 2, \ldots, \numTimePoints \}$
that
$
	\mathcal{X}^{p, w}_0 = x	
$
and
\begin{equation}
\label{Lorentz:eq2}
\begin{split} 
	\mathcal{X}^{p, w}_n
&= 
	\mathcal{X}^{p, w}_{n-1} 
	+
	\sqrt{T/\numTimePoints}  (\beta \odot  w_n )
	+
	\tfrac{T}{2\numTimePoints} 
	\big(
		\mu_\alpha(\mathcal{X}^{p, w}_{n-1}) 
		+
		\mu_\alpha\big(
			\mathcal{X}^{p, w}_{n-1}
			+
			\tfrac{T}{\numTimePoints}
			\mu_\alpha(\mathcal{X}^{p, w}_{n-1}) 
			+
			\sqrt{T/\numTimePoints}  (\beta \odot w_n)
		\big) 
	\big),
\end{split}
\end{equation}
let 
$
	\eulerapprox_k \colon
	\Param \times \R^{\numTimePoints \dimBM} \to
	\R
$,
	$k \in \{0, 1\}$,
satisfy for all 
	$ p \in \Param $,
	$ w \in \R^{\numTimePoints \dimProblem}$
that
\begin{equation}
\begin{split} 
	\eulerapprox_0( p, w)
=
	g(\mathcal{X}^{p, w}_\numTimePoints)
\qandq
	\eulerapprox_1( p, w)
=
	\tfrac{1}{2}\big[g(\mathcal{X}^{p, w}_\numTimePoints) + g(\mathcal{X}^{p, -w}_\numTimePoints)\big],
\end{split}
\end{equation}
assume for all 
	$ p  \in \Param $,
	$ w = (w_1, \ldots, w_{\NumRefMC \numTimePoints \dimProblem}) \in \R^{\NumRefMC \numTimePoints \dimProblem}$,
	$ \theta = (\theta_1,\ldots, \theta_{\NumMC \numTimePoints \dimProblem}) \in \R^{\NumMC \numTimePoints \dimProblem}$
that
\begin{equation}
\label{Lorentz:eq3.0}
\begin{split} 
	\ApproxRef(p, w)
=
	\frac{ 1 }{ \NumRefMC }
	\left[ 
		{\textstyle \sum\limits_{ \MCvariable = 1 }^{ \NumRefMC }}
			\eulerapprox_a\big( 
				p, 
				( 
				w_{ ( \MCvariable - 1 ) \numTimePoints\dimProblem + k }
				)_{k \in \{1, 2, \ldots, \numTimePoints\dimProblem \}}         
			\big)
	\right]
\end{split}
\end{equation}
\begin{equation}
\label{Lorentz:eq3}
\begin{split} 
\andq
	\ApproxAlg(p, \theta)
=
	\frac{ 1 }{ \NumMC }
	\left[ 
		{\textstyle \sum\limits_{ \MCvariable = 1 }^{ \NumMC } }
			\eulerapprox_a\big( 
				p, 
				( 
				\theta_{ ( \MCvariable - 1 ) \numTimePoints\dimProblem + k }
				)_{k \in \{1, 2, \ldots, \numTimePoints\dimProblem \}}         
			\big)
	\right],
\end{split}
\end{equation}
assume for all
	$x, y \in \R$
that
	$\costfct(x, y) = | x-y |^2$,
let
	$\mathbf{a} = {\frac{9}{10}}$,
	$\mathbf{b}  = {\frac{999}{1000}}$,
	$ \varepsilon \in (0,\infty)$,
let
	$(\gamma_m)_{m \in \N} \subseteq (0,\infty)$
satisfy for all 
	$m \in \{1, 2, \ldots, \nrtrainstepsLorentz \}$
that
$
	\gamma_m 
=  
	{\mathbbm{1}_{(0,5000]}(m)}{10^{-3}}
	+
	{\mathbbm{1}_{(5000,8000]}(m)}{10^{-4}}
	+
	{\mathbbm{1}_{(8000,10000]}(m)}{10^{-5}}
$,
assume for all
	$m \in \N$,
	$i \in \{1, 2,\ldots,\ApproxAlgdim \}$,
	$g_1 = (g_1^{(1)},\ldots,g_1^{(\ApproxAlgdim)})$, 	$g_2 = (g_2^{(1)},\ldots,g_2^{(\ApproxAlgdim)})$,
	$\dots$, 
	$g_m = (g_m^{(1)},\ldots,g_m^{(\ApproxAlgdim)})
	\in 
		\R^{\ApproxAlgdim}
	$
that
\begin{equation}
\label{Lorentz:eq4}
\begin{split} 
	\psi_m^{(i)}(g_1, g_2, \ldots, g_m)
=
	\gamma_m
	\left[
		\frac{
			\sum_{k = 1}^m \mathbf{a}^{m-k}(1-\mathbf{a}) g_k^{(i)}
		}{
			1-\mathbf{a}^m
		}
	\right]
	\left[
		\varepsilon +\bigg[
			\frac{
				\sum_{k = 1}^m \mathbf{b}^{m-k}(1-\mathbf{b}) |g_k^{(i)}|^2
			}{
				1- \mathbf{b}^m
			}
		\bigg]^{\nicefrac{1}{2}} 
	\right]^{-1},
\end{split}
\end{equation}	
assume that  
$ \ApproxAlgRV$ is a standard normal random vector,
assume that  
$  
  \ParamRV_{ 
    1, 1 
  }
$ is $\mathcal{U}_{\Param}$-distributed,
and
assume that 
$ 
  W^{ 1, 1 }
$
is a standard normal random vector.

Let us add some comments regarding the setup introduced above.
Observe that \eqref{Lorentz:eq2} corresponds to a Heun discretization (cf., e.g., Kloeden \cite[(1.4) in Section 15]{KloedenPlaten13}) of the SDE associated to the Kolmogorov PDE in \eqref{Lorentz:eq1}.
In the case $a = 0$ the proposal algorithm in \eqref{Lorentz:eq3} thus corresponds to the MC-Heun method and 
in the case $a = 1$ the proposal algorithm in \eqref{Lorentz:eq3} thus corresponds to the antithetic MC-Heun method.
In particular, we remark that for all
	$p = (T, \alpha, \beta, x) \in \Param$
we have that
\begin{equation}
\begin{split} 
	  \frac{ 1 }{ \NumMC }
	  \left[
	    \sum_{ \MCvariable = 1 }^{ \NumMC }
	    \eulerapprox_0\big( 
	    p, 
	    \WRVs^{0,\MCvariable}
	    \big)
	  \right]  
\approx
	u_{\alpha, \beta}(T,x) 
\approx
	  \frac{ 1 }{ \NumMC }
	  \left[
	    \sum_{ \MCvariable = 1 }^{ \NumMC }
	    \eulerapprox_1\big( 
	    p, 
	    \WRVs^{0,\MCvariable}
	    \big)
	  \right].
\end{split}
\end{equation}
Furthermore, observe that \cref{Lorentz:eq4} describes the Adam optimizer in the setup of \cref{algo:general} (cf.\ Kingma \& Ba \cite{Kingma2014} and \cref{sect:Adam}).

In \cref{table:Lorentz_LRV} we approximately present for
	$\NumMC \in \{2^5, 2^6, \ldots, 2^{13}\}$,
	$a \in \{0,1\}$
one random realization of the $L^1(\lambda_{\Param};\R)$-approximation error
\begin{equation}
\label{Lorentz_lrv:eq1}
\begin{split} 
\textstyle
	\int_\Param 
		|u_{\alpha, \beta}(T, x) - \ApproxAlg((T, \alpha, \beta, x), \Theta_{\nrtrainstepsLorentz})|
	\d \, (T, \alpha, \beta, x)
\end{split}
\end{equation}
(\nth{3} column in \cref{table:Lorentz_LRV}),
one random realization of the $L^2(\lambda_{\Param};\R)$-approximation error
\begin{equation}
\label{Lorentz_lrv:eq2}
\begin{split} 
\textstyle
	\bigl[
		\int_\Param 
			|u_{\alpha, \beta}(T, x) - \ApproxAlg((T, \alpha, \beta, x), \Theta_{\nrtrainstepsLorentz})|^2
		\d \, (T, \alpha, \beta, x)
	\bigr]^{\nicefrac{1}{2}}
\end{split}
\end{equation}
(\nth{4} column in \cref{table:Lorentz_LRV}),
one random realization of the $L^\infty(\lambda_{\Param};\R)$-approximation error
\begin{equation}
\label{Lorentz_lrv:eq3}
\begin{split} 
	\sup_{(T, \alpha, \beta, x) \in \Param}
		|u_{\alpha, \beta}(T, x) - \ApproxAlg((T, \alpha, \beta, x), \Theta_{\nrtrainstepsLorentz})|
\end{split}
\end{equation}
(\nth{5} column in \cref{table:Lorentz_LRV}), and
the time to compute $\Theta_{\nrtrainstepsLorentz}$
(\nth{6} column in \cref{table:Lorentz_LRV}).
For every 	
	$\NumMC \in \{2^5, 2^6, \ldots, 2^{13}\}$,
	$a \in \{0,1\}$
we approximated the integrals in \eqref{Lorentz_lrv:eq1} and \eqref{Lorentz_lrv:eq2} with the MC method based on $\numMCsamplesLorentz$ samples
and we approximated the supremum in \eqref{Lorentz_lrv:eq3} based on $\numMCsamplesLorentz$ random samples
(cf., e.g., Beck et al.\ \cite[Lemma 3.5]{BeckJafaari21} and Beck et al.\ \cite[Section 3.3]{Beck2019published}).
In our approximations of \eqref{Lorentz_lrv:eq1}, \eqref{Lorentz_lrv:eq2}, and \eqref{Lorentz_lrv:eq3} we have approximately computed for all required sample points 
	$(T, \alpha, \beta, x) \in \Param$
the value $u_{\alpha, \beta}(T, x)$ of the unknown exact solution by means of an antithetic MC approximation with $\NumRefMCLorentz$ MC samples.

To compare the LRV strategy with existing approximation techniques from the literature, we also employ several other methods to approximate the function $\Param \ni (T, \alpha, \beta, x) \mapsto u_{\alpha, \beta}(T, x) \in \R$ in \eqref{Lorentz:eq1}.
Specifically, 
	in \cref{table:Lorentz_ANNs} we present numerical simulations for the deep learning method induced by Becker et al.\ \cite{BeckJafaari21} 
	(with Adam $160000$ training steps, 
	batch size $256$, 
	learning rate schedule 
	$
		\N \ni j 
		\mapsto 
			{\mathbbm{1}_{(0,50000]}(m)}{10^{-1}}
			+
			{\mathbbm{1}_{(50000,100000]}(m)}{10^{-2}}
			+
			{\mathbbm{1}_{(100000,120000]}(m)}{10^{-3}}
			+
			{\mathbbm{1}_{(120000,140000]}(m)}{10^{-4}}
			+
			{\mathbbm{1}_{(140000,160000]}(m)}{10^{-5}}
	$,
	and GELU activation function),
	and
	in \cref{table:Lorentz_MC} we present numerical simulations for the standard and the antithetic MC method.
In Tables~\ref{table:Lorentz_ANNs} and \ref{table:Lorentz_MC} 
	we have approximated the $L^1(\lambda_{\Param};\R)$-approximation errors
		of the respective approximation methods with the MC method based on $\numMCsamplesLorentz$ samples,
	we have approximated the $L^2(\lambda_{\Param};\R)$-approximation errors
		of the respective approximation methods with the MC method based on $\numMCsamplesLorentz$ samples, and 
	we have approximated the $L^\infty(\lambda_{\Param};\R)$-approximation errors
		of the respective approximation methods based on $\numMCsamplesLorentz$ random samples
		(cf., e.g., Beck et al.\ \cite[Lemma 3.5]{BeckJafaari21} and Beck et al.\ \cite[Section 3.3]{Beck2019published}).
In our approximations of the above mentioned approximation errors we have approximately computed for all required sample points 
	$(T, \alpha, \beta, x) \in \Param$
the value $u_{\alpha, \beta}(T, x)$ of the unknown exact solution by means of an antithetic MC approximation with $\NumRefMCLorentz$ MC samples

\begin{table}[H] \tiny
\resizebox{\textwidth}{!}{
\begin{filecontents*}{Table_12.tex}
num_lrv_samples, num_params, total_nr_eval_solutions, antithetic, l1_error, l_2_error, l_inf_error, rel_l1_error, rel_l_2_error, rel_l_inf_error, time, dtype, train_time
32,2400,12800,0,0.001903,0.003234,0.042107,1.18485e-05,2.53553e-05,0.0005349228,0.85, float32,321.32
32,2400,12800,1,0.000376,0.000714,0.011292,2.3531e-06,5.6802e-06,0.0001781952,1.31, float32,481.52
64,4800,12800,0,0.001706,0.003016,0.042892,1.05118e-05,2.08314e-05,0.0003797443,0.82, float32,323.65
64,4800,12800,1,0.000296,0.000587,0.010147,1.7466e-06,3.7666e-06,5.32117e-05,1.35, float32,481.12
128,9600,12800,0,0.001648,0.003062,0.062836,1.00918e-05,2.19483e-05,0.0004144138,0.84, float32,322.32
128,9600,12800,1,0.000276,0.000589,0.011658,1.6419e-06,3.6188e-06,4.24265e-05,1.36, float32,480.80
256,19200,12800,0,0.000801,0.001198,0.015732,4.6268e-06,8.7e-06,0.0001688115,0.84, float32,320.24
256,19200,12800,1,0.000166,0.000328,0.009003,1.076e-06,2.3699e-06,5.86465e-05,1.34, float32,480.83
512,38400,12800,0,0.000732,0.001265,0.020386,4.3458e-06,8.1951e-06,7.62993e-05,0.84, float32,322.96
512,38400,12800,1,0.000173,0.000376,0.008942,1.1835e-06,2.7412e-06,3.93168e-05,1.35, float32,482.99
1024,76800,12800,0,0.000537,0.000958,0.015106,3.1521e-06,6.4139e-06,9.3302e-05,0.82, float32,327.63
1024,76800,12800,1,0.000112,0.000223,0.004013,7.136e-07,1.6266e-06,2.49012e-05,1.34, float32,518.68
2048,153600,12800,0,0.000721,0.001086,0.014084,4.202e-06,7.3249e-06,9.44604e-05,0.83, float32,399.33
2048,153600,12800,1,0.000069,0.000118,0.001572,4.441e-07,9.633e-07,1.96249e-05,1.39, float32,968.62
4096,307200,12800,0,0.000627,0.000844,0.007401,3.7755e-06,6.2642e-06,7.89588e-05,0.86, float32,836.97
4096,307200,12800,1,0.000090,0.000155,0.001984,5.895e-07,1.2352e-06,1.52829e-05,1.36, float32,2174.38
8192,614400,12800,0,0.000177,0.000282,0.004089,1.0595e-06,2.0166e-06,2.42445e-05,0.87, float32,1924.66
8192,614400,12800,1,0.000085,0.000144,0.001419,5.585e-07,1.2528e-06,2.34625e-05,1.64, float32,4571.75
\end{filecontents*}
\csvreader[tabular=|c|c|c|c|c|c|c|c|,
     table head=
     \hline $\NumMC$ &  $a$ &\thead{Number\\of\\trainable\\parameters}& \thead{$L^1$-approx. \\ error}& \thead{$L^2$-approx. \\error}&\thead{ $L^\infty$-approx. \\ error} & \thead{Training \\ time \\ in \\seconds} 
     \\\hline,
    late after line=\\\hline]{Table_12.tex}
{
	num_lrv_samples=\nsamp, 
	num_params = \nparam,
	antithetic=\anti, 
	l1_error=\lll, 
	l_2_error=\llll, 
	l_inf_error=\linf, 
	rel_l1_error = \rel, 
	rel_l_2_error = \rell,
	rel_l_inf_error = \relinf,
	train_time = \train,
	time=\eval, 
	dtype = \dtype
}
{\nsamp& \anti &\nparam & \lll & \llll & \linf &\train
}
}
\caption{\label{table:Lorentz_LRV}Numerical simulations for the LRV strategy in case of the parametric stochastic Lorentz equation described in \cref{simul:Lorentz} (10-dimensional approximation problem)}
\end{table}

\begin{table}[H] \tiny
\resizebox{\textwidth}{!}{
\begin{filecontents*}{Table_13.tex}
layers, inner_dim, num_weights, mc_samples_per_ref_sol, l1_error, l_2_error, l_inf_error, rel_l1_error, rel_l_2_error, rel_l_inf_error, train_time, time, dtype
2,32,1441,512,2.299140,5.982908,104.470703,0.0103519911,0.0263128335,0.380205363,1876.86,0.64, float32
2,64,4929,512,4.730990,16.753924,199.115005,0.0208499725,0.0711423366,0.7372565269,1863.79,0.64, float32
2,128,18049,512,4.884268,13.618965,377.537964,0.0180955348,0.0517001519,1.6554224491,1869.14,0.65, float32
2,256,68865,512,3.218077,9.487827,106.434830,0.0140808362,0.0423752603,0.5129440427,1882.02,0.66, float32
3,32,2497,512,3.712238,11.086662,693.645264,0.0159050341,0.0511203142,4.0133786201,1880.95,0.65, float32
3,64,9089,512,28.689849,161.865602,3000.016357,0.0999487693,0.5341756985,11.0087432861,1881.48,0.66, float32
3,128,34561,512,5.824964,15.087503,200.595459,0.0263588804,0.0691843915,0.7880799174,1903.12,0.63, float32
3,256,134657,512,93.630541,110.419151,313.946106,0.4998307931,0.7023580147,3.5089039803,1891.98,0.63, float32
\end{filecontents*}
\csvreader[tabular=|c|c|c|c|c|c|c|c|c|,
     table head=
	\hline \thead{Number \\ of \\ hidden \\ layers}& \thead{Number  \\of neurons \\ on each \\ hidden \\ layer}& \thead{Number \\ of \\ trainable \\ parameters} & \thead{$L^1$-approx. \\ error}& \thead{$L^2$-approx. \\error}&\thead{ $L^\infty$-approx. \\ error} & \thead{Training time \\ in seconds} 
	\\\hline,
	 late after line=\\\hline]{Table_13.tex}
{
	layers = \lyrs,
	inner_dim = \inner,
	num_weights = \weights,
	mc_samples_per_ref_sol = \mcsamp, 
	l1_error=\lll, 
	l_2_error=\llll, 
	l_inf_error=\linf, 
	rel_l1_error = \rel, 
	rel_l_2_error = \rell, 
	rel_l_inf_error = \relinf,
	train_time=\train, 
	time=\eval  ,
	dtype = \dtype
}
{\lyrs & \inner &\weights & \lll & \llll & \linf &\train 
}}
\caption{\label{table:Lorentz_ANNs}Numerical simulations for the deep learning method induced by Becker et al.\ \cite{BeckJafaari21} in case of the parametric stochastic Lorentz equation described in \cref{simul:Lorentz} (10-dimensional approximation problem)}
\end{table}

\begin{table}[H] \tiny
\resizebox{\textwidth}{!}{
\begin{filecontents*}{Table_14.tex}
mc_samples, num_RVs, total_nr_eval_solutions, antithetic, l1_error, l_2_error, l_inf_error, rel_l1_error, rel_l_2_error, rel_l_inf_error, time, dtype
32,2400,12800,0,0.383773,0.534779,4.735626,0.0019744993,0.0032124505,0.0401851945,0.95, float32
32,2400,12800,1,0.008103,0.016280,0.210976,6.10689e-05,0.0001588967,0.0031736151,1.37, float32
64,4800,12800,0,0.275408,0.375781,3.149811,0.0014426349,0.0022590261,0.0212367103,0.81, float32
64,4800,12800,1,0.005694,0.011187,0.183968,4.28298e-05,0.0001132408,0.002829988,1.33, float32
128,9600,12800,0,0.204422,0.281741,2.723969,0.0010616539,0.0017012737,0.016499944,0.83, float32
128,9600,12800,1,0.003942,0.007754,0.103683,2.95079e-05,7.62854e-05,0.0016144606,1.33, float32
256,19200,12800,0,0.145995,0.200546,1.234207,0.0007440135,0.001150678,0.0097941384,0.82, float32
256,19200,12800,1,0.002641,0.005119,0.073952,1.99154e-05,5.00715e-05,0.0009122448,1.37, float32
512,38400,12800,0,0.093797,0.127197,1.015808,0.0004968632,0.0007958077,0.0073440815,0.82, float32
512,38400,12800,1,0.001872,0.003753,0.071739,1.39512e-05,3.65929e-05,0.0011684585,1.36, float32
1024,76800,12800,0,0.073213,0.101349,0.916580,0.0003904487,0.0006426135,0.0074001378,0.79, float32
1024,76800,12800,1,0.001386,0.002828,0.041504,1.04331e-05,2.75078e-05,0.0005829403,1.35, float32
2048,153600,12800,0,0.048225,0.067596,0.737793,0.0002567134,0.0004246453,0.004218868,0.83, float32
2048,153600,12800,1,0.000961,0.001928,0.030483,7.2184e-06,1.95324e-05,0.0005681491,1.34, float32
4096,307200,12800,0,0.034988,0.048361,0.357437,0.0001783916,0.0002903566,0.0030522763,0.82, float32
4096,307200,12800,1,0.000664,0.001317,0.020393,4.9617e-06,1.3134e-05,0.0003321588,1.37, float32
8192,614400,12800,0,0.025400,0.034397,0.242668,0.0001329272,0.0002128101,0.0025910088,0.83, float32
8192,614400,12800,1,0.000449,0.000922,0.016327,3.3673e-06,8.6848e-06,0.0001794489,1.66, float32
16384,1228800,12800,0,0.018963,0.026513,0.256882,0.0001008702,0.000168335,0.0017506861,1.52, float32
16384,1228800,12800,1,0.000357,0.000715,0.010361,2.6669e-06,6.826e-06,0.0001530914,5.40, float32
32768,2457600,12800,0,0.010552,0.014373,0.114243,5.52798e-05,8.83994e-05,0.0008744103,4.80, float32
32768,2457600,12800,1,0.000241,0.000457,0.006233,1.7503e-06,4.3949e-06,9.37636e-05,12.02, float32
\end{filecontents*}
\csvreader[tabular=|c|c|c|c|c|c|c|c|c|,
     table head=
     \hline \thead{Number \\ of \\MC \\samples}&\thead{MC Method \\ 0: standard \\ 1: antithetic} &\thead{Number of \\ scalar random \\ variables \\per evaluation}  & \thead{$L^1$-approx. \\ error}& \thead{$L^2$-approx. \\error}&\thead{ $L^\infty$-approx. \\ error} 
     \\\hline,
    late after line=\\\hline]{Table_14.tex}
{
	mc_samples=\nsamp, 
	num_RVs = \ninputs,
	total_nr_eval_solutions = \mcsamp, 
	antithetic=\anti, 
	l1_error=\lll, 
	l_2_error=\llll, 
	l_inf_error=\linf, 
	time=\train, 
	dtype = \dtype
}
{\nsamp&\anti & \ninputs & \lll & \llll & \linf 
}
}
\caption{\label{table:Lorentz_MC}Numerical simulations for the standard and the antithetic MC method in case of the parametric stochastic Lorentz equation described in \cref{simul:Lorentz} (10-dimensional approximation problem)}
\end{table}

\section*{Data availability statement}
The data that support the findings of this study are available from the corresponding author upon reasonable request.

\section*{Acknowledgments}
Nor Jaafari is gratefully acknowledged for his useful assistance regarding some numerical simulations.
The second author gratefully acknowledges the Cluster of Excellence EXC 2044-390685587, Mathematics M\"unster: Dynamics-Geometry-Structure funded by the Deutsche Forschungsgemeinschaft (DFG, German Research Foundation).
This work has been partially funded by the National Science Foundation of China (NSFC) under grant number 12250610192.

\bibliographystyle{acm}

\end{document}